\documentclass[11pt]{article}
\usepackage{amsfonts,mathrsfs,amssymb,amsthm,mathptm}
\usepackage{amsmath,amscd}
\usepackage{mathptm,pslatex}
\usepackage[all]{xy}
\usepackage{graphicx}
\usepackage{fancyhdr}
\usepackage{hyperref}
\usepackage{xcolor}
\usepackage{bm}
\pagecolor[rgb]{0.9, 0.99, 0.9}
\begin{document}
\newtheorem{Def}{Definition}[section]
\newtheorem{Ex}[Def]{Example}
\newtheorem{Prop}[Def]{Proposition}
\newtheorem{Theo}[Def]{Theorem}
\newtheorem{Lem}[Def]{Lemma}
\newtheorem{Coro}[Def]{Corollary}
\theoremstyle{definition}
\newtheorem{Rem}[Def]{Remark}

\newcommand{\add}{{\rm add}}
\newcommand{\gd}{{\rm gl.dim}}
\newcommand{\dm}{{\rm dom.dim}}
\newcommand{\E}{{\rm E}}
\newcommand{\Mor}{{\rm Morph}}
\newcommand{\End}{{\rm End}}
\newcommand{\ind}{{\rm ind}}
\newcommand{\rsd}{{\rm res.dim}}
\newcommand{\rd} {{\rm rep.dim}}
\newcommand{\ol}{\overline}
\newcommand{\overpr}{$\square$}
\newcommand{\rad}{{\rm rad}}
\newcommand{\soc}{{\rm soc}}
\renewcommand{\top}{{\rm top}}
\newcommand{\pd}{{\rm proj.dim}}
\newcommand{\id}{{\rm inj.dim}}
\newcommand{\fld}{{\rm flat.dim}}
\newcommand{\Fac}{{\rm Fac}}
\newcommand{\Gen}{{\rm Gen}}
\newcommand{\fd} {{\rm fin.dim}}
\newcommand{\DTr}{{\rm DTr}}
\newcommand{\cpx}[1]{#1^{\bullet}}
\newcommand{\D}[1]{{\mathscr D}(#1)}
\newcommand{\Dc}[1]{{\mathscr D}^c(#1)}
\newcommand{\Dz}[1]{{\mathscr D}^+(#1)}
\newcommand{\Df}[1]{{\mathscr D}^-(#1)}
\newcommand{\Db}[1]{{\mathscr D}^b(#1)}
\newcommand{\C}[1]{{\mathscr C}(#1)}
\newcommand{\Cz}[1]{{\mathscr C}^+(#1)}
\newcommand{\Cf}[1]{{\mathscr C}^-(#1)}
\newcommand{\Cb}[1]{{\mathscr C}^b(#1)}
\newcommand{\K}[1]{{\mathscr K}(#1)}
\newcommand{\Kz}[1]{{\mathscr K}^+(#1)}
\newcommand{\Kf}[1]{{\mathscr  K}^-(#1)}
\newcommand{\Kb}[1]{{\mathscr K}^b(#1)}
\newcommand{\Kac}[1]{{\mathscr  K}_{ac}(#1)}
\newcommand{\modcat}{\ensuremath{\mbox{{\rm -mod}}}}
\newcommand{\Modcat}{\ensuremath{\mbox{{\rm -Mod}}}}
\newcommand{\Tr}{{\rm Tr}}
\newcommand{\stmodcat}[1]{#1\mbox{{\rm -{\underline{mod}}}}}
\newcommand{\pmodcat}[1]{#1\mbox{{\rm -proj}}}
\newcommand{\imodcat}[1]{#1\mbox{{\rm -inj}}}
\newcommand{\Pmodcat}[1]{#1\mbox{{\rm -Proj}}}
\newcommand{\Imodcat}[1]{#1\mbox{{\rm -Inj}}}
\newcommand{\opp}{^{\rm op}}
\newcommand{\otimesL}{\otimes^{\rm\mathbb L}}
\newcommand{\rHom}{{\rm\mathbb R}{\rm Hom}\,}
\newcommand{\projdim}{\pd}
\newcommand{\Hom}{{\rm Hom}}
\newcommand{\Coker}{{\rm Coker}}
\newcommand{ \Ker  }{{\rm Ker}}
\newcommand{ \Img  }{{\rm Im}}
\newcommand{ \supp  }{{\rm Supp}}
\newcommand{ \amp  }{{\rm amp}}
\newcommand{ \dg  }{{\rm dg}}
\newcommand{ \Di }{{\rm D}}
\newcommand{\Dle}[2]{\mathscr D^{\le #1}(#2)}
\newcommand{\Kle}[2]{\mathscr K^{\le #1}(#2)}
\newcommand{\hocolim}{\rm hocolim}
\newcommand{\Dge}[2]{\mathscr D^{\ge #1}(#2)}
\newcommand{\holim}{\rm holim}
\newcommand{\Kbb}[2]{\mathscr K^{#1}_{\infty}(#2)}
\newcommand{\cE}{\mathscr E}
\newcommand{\cY}{\mathscr Y}
\newcommand{\cH}{\mathcal H}
\newcommand{\cX}{\mathcal X}
\newcommand{\cU}{\mathcal U}
\newcommand{\cV}{\mathcal V}
\newcommand{\Dint}[1]{\mathscr D^{[c,d]}(#1)}
\newcommand{\Id}{{\rm Id}}
\newcommand{\pretr}{{\rm pretr}}
\newcommand{\Fun}{{\rm Fun}}
\newcommand{\ex}{{\rm ex}}

\newcommand{\Ext}{{\rm Ext}}
\newcommand{\StHom}{{\rm \underline{Hom}}}
\def\vez{\varepsilon}
\def\bz{\bigoplus}
\def\sz {\oplus}
\def\epa{\xrightarrow}
\def\inja{\hookrightarrow}
\newcommand{\IZ}{\mathbb{Z}}
\newcommand{\con}{{\rm Con}}

\newcommand{\dotHom}{\Hom^{\bullet}}
\newcommand{\dotEnd}{\End^{\bullet}}
\newcommand{\thick}{{\rm thick}}
\newcommand{\Tria}{{\rm Tria}}
\newcommand{\lra}{\longrightarrow}
\newcommand{\lla}{\longleftarrow}
\newcommand{\lraf}[1]{\stackrel{#1}{\lra}}
\newcommand{\llaf}[1]{\stackrel{#1}{\lla}}
\newcommand{\ra}{\rightarrow}
\newcommand{\dk}{{\rm dim_{_{k}}}}

\newcommand{\colim}{{\rm colim\, }}
\newcommand{\limt}{{\rm lim\, }}
\newcommand{\Add}{{\rm Add }}
\newcommand{\Tor}{{\rm Tor}}
\newcommand{\Cogen}{{\rm Cogen}}
\newcommand{\simeqf}[1]{\stackrel{#1}{\simeq}}

{\Large \bf
\begin{center}
Recollements of derived categories from $n$-term big tilting complexes
\end{center}}
\medskip

\centerline{\bf Shengyong Pan$^1$, Huabo Xu$^{2*}$}
\medskip

\begin{center}\footnotesize
1.School of Mathematics and Statistics, Beijing Jiaotong University,\\
100044~Beijing, China
\smallskip

Beijing Key Laboratory of Biological Big Data and Topological Statistics, Beijing Jiaotong University,\\
 100044~
Beijing, China
\smallskip

2. School of Science, Beijing University of Civil Engineering and Architecture,\\ 102616
~ Beijing,
 China
\end{center}

\renewcommand{\thefootnote}{\alph{footnote}}
\setcounter{footnote}{-1} \footnote{
 Email:shypan@bjtu.edu.cn.}

\renewcommand{\thefootnote}{\alph{footnote}}
\setcounter{footnote}{-1} \footnote{
Corresponding author. Email: huabo0567@163.com.}
\renewcommand{\thefootnote}{\alph{footnote}}
\setcounter{footnote}{-1} 
\renewcommand{\thefootnote}{\alph{footnote}}
\setcounter{footnote}{-1} 
\renewcommand{\thefootnote}{\alph{footnote}}
\setcounter{footnote}{-1} 
\footnote{Date: version of August 01, 2026}

\begin{abstract}
Let $A$ be a ring, let ${\bf T}$ be a finite-term big tilting complex over
$A$, and choose a bounded projective representative $\cpx{P}$ of ${\bf T}$.
Put $B=\End_{\D{A}}(\cpx{P})$.
We prove that the kernel of the associated derived tensor functor
${\bf R}:\D{B}\to\D{A}$ is equivalent, through the derived inclusion,
to $\D{\mathscr E}$, where
$\mathscr E=\Ker({\bf R})\cap B\Modcat$ carries the inherited exact
structure. Consequently, $\D{B}$ admits a recollement by
$\D{\mathscr E}$ and $\D{A}$.
The proof combines a dg double-centralizer argument with a functorial
realisation of the kernel on unbounded complexes.
The category $\mathscr E$ is definable and $d$-symmetric whenever $d$
is at least the projective amplitude of the underlying perfect right
$B$-complex. We also characterise when the kernel is induced by a
homological ring epimorphism: this holds precisely when $\mathscr E$
is an abelian subcategory, or equivalently when the standard
$t$-structure restricts to the kernel. In amplitude at most one,
this recovers the universal-localisation description for two-term
big tilting complexes. Over semiperfect endomorphism rings, the
amplitude is detected by reduction modulo the Jacobson radical.
Finally, we construct non-compact examples of arbitrarily large
minimal projective amplitude. One family yields algebraic Calkin
quotients; another, over connected rings constructed from singular
local hypersurfaces, has nonzero, non-abelian exact kernels and cannot
be obtained from modules by derived equivalences between ordinary rings.
\end{abstract}

\tableofcontents

\section{Introduction}\label{sec:introduction}

A classical tilting complex over a ring $A$ identifies $\D{A}$ with
the derived category of its endomorphism ring by derived Morita theory
\cite{Rickard,keller}. For a tilting object which is not compact,
the associated derived tensor functor is instead a localisation.
Its kernel measures the failure of this functor to be an equivalence.
The problem addressed here is to describe that kernel as the derived
category of a concrete exact category.

Recollements organise such localisations into decompositions of
triangulated categories. Introduced by Beilinson, Bernstein and
Deligne \cite{BBD}, they have become useful in tilting theory and
representation theory; see \cite{HKL,Miyachi,JOR,Yang}.
In the module setting, Chen and Xi proved that a good tilting module
of projective dimension at most one induces a recollement whose left
term is the derived category of a homological universal localisation
\cite[Theorem~1.1]{CX1}. For higher projective dimension, this
description by an ordinary ring can fail
\cite[Theorem~1.1 and Proposition~7.2]{CX1'}.
Chen and Xi subsequently showed that the kernel is always realised
by the derived category of a symmetric exact subcategory
\cite[Theorem~1.2]{CX3}. This identifies exact categories as the
appropriate setting for the general kernel problem.

Xu introduced big tilting complexes and established the corresponding
universal-localisation recollement for two-term complexes
\cite{XuTwoTerm}. The aim of this paper is to extend that result to
arbitrary finite-term big tilting complexes. The new issue is to
realise the entire unbounded tensor kernel by an exact category,
even when the stalk modules in the kernel are not closed under
ambient kernels and cokernels.

\begin{Def}\rm\label{Big TC}\cite{XuTwoTerm}
A complex ${\bf T}\in \C{A}$ is called a \emph{big tilting complex}
if it is isomorphic in $\D{A}$ to a bounded complex $P^\bullet$ of
projective $A$-modules satisfying
\begin{enumerate}
\item $\Hom_{\D{A}}(P^\bullet,(P^\bullet)^{(I)}[r])=0$
      for every set $I$ and every $r\neq0$;
\item $A\in\thick_{\D{A}}(P^\bullet)$.
\end{enumerate}
It is \emph{$n$-term} if $P^\bullet$ can be chosen with nonzero terms
only in degrees $-(n-1),\ldots,0$.
\end{Def}

Fix such a representative $\cpx{P}$ of ${\bf T}$ and put
\[
 B=\End_{\D{A}}(\cpx{P}),\qquad
 \Lambda=\dotEnd_A(\cpx{P}),\qquad
 \Delta=\tau_{\leq0}\Lambda.
\]
A chosen isomorphism ${\bf T}\simeq\cpx{P}$ in $\D{A}$ identifies
$B$ with $\End_{\D{A}}({\bf T})$.
Self-orthogonality gives quasi-isomorphisms
$\Delta\to\Lambda$ and $\Delta\to B$.
They transport the right dg endomorphism action to a derived
right $B$-action. Write
$T^\bullet=P^\bullet\otimesL_\Delta B$
for the resulting derived $A$-$B$-bimodule; its underlying left
$A$-object is $P^\bullet$. Over arbitrary rings the bimodule
construction is taken over the derived enveloping algebra, as
specified in Section~\ref{Subsection 2.3}.
Consider the adjunction
\[
 {\bf R}=T^\bullet\otimesL_B-:\D{B}\rightleftarrows\D{A}:
 {\bf H}=\rHom_A(T^\bullet,-),
\]
and define
\[
 \mathscr Y_B=\Ker({\bf R}),\qquad
 \mathscr E=\mathscr Y_B\cap B\Modcat.
\]
The category $\mathscr E$ is extension-closed and carries the exact
structure inherited from $B\Modcat$.

\begin{Theo}\label{main theorem}
Let $A$ be an associative ring with identity, let ${\bf T}$ be an
$n$-term big tilting complex over $A$, and choose a bounded projective
representative $\cpx{P}$ of ${\bf T}$. Put $B=\End_{\D{A}}(\cpx{P})$.
The exact inclusion $\lambda:\mathscr E\hookrightarrow B\Modcat$
induces a triangle equivalence
\[
 \D{\lambda_*}:\D{\mathscr E}\xrightarrow{\ \simeq\ }
 \mathscr Y_B=\Ker(T^\bullet\otimesL_B-).
\]
In particular, there is a recollement
\[
\xymatrix{
 \D{\mathscr E}\ar[r]^-{\D{\lambda_*}}
 &\D{B}\ar[r]^-{{\bf R}}
   \ar@/^1.4pc/[l]\ar@/_1.4pc/[l]
 &\D{A}\ar@/^1.4pc/[l]^-{{\bf H}}\ar@/_1.4pc/[l]
}
\]
of unbounded derived categories.
\end{Theo}

The proof separates the localisation from its exact realisation.
First, the condition $A\in\thick(P^\bullet)$ makes the natural
double-dual evaluation an isomorphism on $A$. It follows that
$P^\bullet$ is a perfect right dg $\Lambda$-module and that the
double-centralizer morphism
$A\to\dotEnd_{\Lambda^{\opp}}(P^\bullet)^{\opp}$
is a quasi-isomorphism. The recollement construction for a compact
dg module \cite[Sections~2--3]{JOR} then gives the localisation,
which is transported to $\D{B}$ along the two quasi-isomorphisms above.

The second step is the exact realisation of the kernel. For each
projective $B$-module $Q$, the adjunction unit gives a natural exact
sequence
\[
 0\longrightarrow Q\longrightarrow\Psi(Q)
 \longrightarrow C(Q)\longrightarrow0,
 \qquad \Psi(Q)=H^0({\bf HR}(Q)),\quad C(Q)\in\mathscr E.
\]
We compare the termwise functor $\Psi$ with ${\bf HR}$ on complexes
of projectives. A coherent comparison on bounded complexes extends
to bounded-above complexes by a uniform cohomological estimate and
then to arbitrary complexes by homotopy limits. The shifted cokernel
construction gives a quasi-inverse to the derived inclusion; a
functorial pushout supplies the inverse comparison inside
$\D{\mathscr E}$. In particular, this argument addresses both
full faithfulness and essential surjectivity in the unbounded case.

The underlying right $B$-object $T^\bullet_B$ is perfect.
If a bounded projective representative is supported in $[a,b]$,
then $\mathscr E$ is $d$-symmetric for every $d\geq b-a$
(Theorem~\ref{thm:sym-new}). Moreover, exactness of
$V^\bullet\otimes_BM$ is expressed by finitely many pp-pairs,
so $\mathscr E$ is definable (Proposition~\ref{prop:definability}).
These properties provide uniform closure and dimension-shifting
arguments without assuming that $\mathscr E$ is abelian.

We also characterise the ordinary-ring case
(Theorem~\ref{thm:chen-xi-criterion}). The following are equivalent:
$\mathscr E$ is an abelian subcategory of $B\Modcat$; the standard
$t$-structure on $\D{B}$ restricts to $\mathscr Y_B$; and the given
embedding of $\mathscr Y_B$ is induced by a homological ring
epimorphism $B\to C$. In right projective amplitude at most one,
these conditions hold. The two-term case recovers Xu's homological
universal localisation (Corollary~\ref{coro:two-term}). If $B$ is
semiperfect, reduction modulo $J(B)$ determines the optimal right
projective amplitude and yields a finite sufficient test for this
ordinary-ring realisation (Theorem~\ref{thm:radical-support-criterion}).

Thus the exact-category interpretation of Chen--Xi extends from good
tilting modules to finite-term big tilting complexes. The additional
work concerns the derived endomorphism action and the coherent
assembly of unbounded complexes. The examples show that the extension
also has content beyond the module case. Infinite amplification of
idempotent-chain tilting complexes produces explicit algebraic
Calkin quotients (Theorem~\ref{thm:algebraic-calkin-recollement}).
A second family, constructed from singular local hypersurface
domains, consists of connected examples with nonzero cohomology in
all degrees of their minimal projective support and nonzero,
non-abelian exact kernels (Example~\ref{ex:singular-chain}).
Proposition~\ref{prop:singular-chain-no-transport} shows that these
objects cannot be obtained from modules by any derived equivalence
between ordinary rings.

The paper is structured as follows. In Section~\ref{sect2} we fix conventions and treat minimal projective
amplitude over semiperfect rings. In Section~\ref{derecoll} we prove the dg recollement,
the symmetry and definability results, the exact-realisation theorem,
and the criteria for realisation by homological ring epimorphisms.
In Section~\ref{sec:examples} we give the two families of examples and
the obstruction to transport from modules.

\section*{Acknowledgments}
The research of Shengyong Pan is supported by Beijing Natural Science Foundation (1262017,1252011).
The second author is supported by
the BUCEA Postgraduate Education and Teaching Quality Improvement Project
(Grant No.J2026016).

\section{Preliminaries\label{sect2}}
In this section, we briefly recall some definitions, basic facts and conventions used in this paper.

\subsection{Notation}

Let $\mathcal C$ be an additive category.
Throughout the paper, a full subcategory $\mathcal {X}$ of $\mathcal C$ is always assumed to be closed under isomorphisms, that is, if
$X$ and $Y$ are objects in $\cal C$, then $Y\in{\mathcal X}$
whenever $Y\simeq X$ with $X\in  {\mathcal X}$.
By $\Ker(\Hom_{\mathcal{C}}(\mathcal{X},-))$ we denote the \emph{right orthogonal subcategory} with respect to $\mathcal{X}$,
that is, the full subcategory of $\mathcal{C}$ consisting of the objects $Y$ such
that $\Hom_{\mathcal{C}}(X,Y)=0$ for all objects $X$ in
$\mathcal{X}$.

Given two morphisms $f: X\to Y$ and $g: Y\to Z$ in $\mathcal C$, we
denote the composition of $f$ and $g$ by $fg$ which is a morphism
from $X$ to $Z$, while we denote the composition of a functor
$F:\mathcal {C}\to \mathcal{D}$ between categories $\mathcal C$ and
$\mathcal D$ with a functor $G: \mathcal{D}\to \mathcal{E}$ between
categories $\mathcal D$ and $\mathcal E$ by $GF$ which is a functor
from $\mathcal C$ to $\mathcal E$. The essential image of the functor $F$ is denoted by Im$(F)$ which is a full subcategory of $\mathcal D$.

By a complex $\cpx{X}$ over $\mathcal{C}$ we mean a sequence  $\cdots\to X^i\xrightarrow{d_X^i } X^{i+1}\epa{d_X^{i+1}} X^{i+2}\to\cdots$ of morphisms $d_X^i$ between objects $X^i$ in $\mathcal{C}$ satisfying that $d_X^id_X^{i+1}=0$ for all
$i\in\mathbb{Z}$. As usual, $\cpx{X}$ is denoted by $
(X^i, d_X^i)_{i\in\mathbb{Z}}$ and $d_X^i$ is called an $i$-th
differential of $\cpx{X}$, $Z^i(\cpx{X})=\Ker(d^i)$ is called the $i$-cycles of $\cpx{X}$ and $H^i(X)$ is called the $i$-th cohomology. Sometimes, we shall write $(X^i)_{i\in\mathbb{Z}}$ for $\cpx{X}$ without mentioning $d^i_X$. For a fixed integer $n$, we denote by
$\cpx{X}[n]$ the complex obtained from $\cpx{X}$ by shifting $n$ degrees, that is, $(\cpx{X}[n])^i=X^{n+i}$,  and by $H^n(\cpx{X})$ the cohomology of $\cpx{X}$ in degree $n$. The \emph{soft truncated complex} $\tau_{\leq n} \cpx{X}$ of $\cpx{X}$ in degree $n$ is a subcomplex of $\cpx{X}$ satisfying that $(\tau_{\leq n} \cpx{X})^i$ equals $X^i$ for all $i<n$, $\Ker(d^i) $ for $i=n$ and zero otherwise.
The \emph{brutal complex} $\cpx{X}_{\leq n}$ of $\cpx{X}$ in degree $n$ is a subcomplex of $\cpx{X}$ satisfying $(\cpx{X}_{\leq n})^i=X^i$ for all $i\leq n$, and zero otherwise.
Recall that the cohomological support $\supp(\cpx{X}) = \{ i \in \mathbb{Z} \mid H^i(\cpx{X}) \neq 0 \}$.
If this support is nonempty and bounded, its cohomological amplitude is $\amp(\cpx{X})=\sup\supp(\cpx{X})-\inf\supp(\cpx{X})$. Projective amplitude is defined separately in Proposition~\ref{thm:minimal-amplitude}.

Let $\C{\mathcal{C}}$ be the category of all complexes over
$\mathcal{C}$ with chain maps, and $\K{\mathcal{C}}$ the homotopy
category of $\C{\mathcal{C}}$. We denote by $\Cb{\mathcal{C}}$ and
$\Kb{\mathcal{C}}$ the full subcategories of $\C{\mathcal{C}}$ and
$\K{\mathcal{C}}$ consisting of bounded complexes over
$\mathcal{C}$, respectively. When $\mathcal{C}$ is abelian, the derived category of $\mathcal{C}$ is denoted by $\D{\mathcal{C}}$, which is the localization of $\K{\mathcal C}$ by inverting quasi-isomorphisms.

Let $\mathcal{T}$ be a triangulated category with (small) coproducts, that is, coproducts indexed over sets exist in
$\mathcal{T}$. An object $U\in \mathcal{T}$ is said to be \emph {compact} if $\Hom_\mathcal{T}(U,-)$ commutes with coproducts in $\mathcal{T}$.
The full subcategory of $\mathcal{T}$ consisting of all compact objects is denoted by $\mathcal{T}^c$.
For any non-empty class $\mathscr{S}$ of objects in $\mathcal{T}$, we denote by $\Tria_\mathcal{T}(\mathscr{S})$ (resp., $\thick_\mathcal{T}(\mathscr{S})$) the smallest full triangulated subcategory of $\mathcal{T}$ containing $\mathscr{S}$ and being closed under coproducts (resp., direct summands). When $\mathcal{T}$ is clear from the context, we also denote $\Tria_\mathcal{T}(\mathscr{S})$ by $\Tria(\mathscr{S})$ without referring to $\mathcal{T}$. If $\mathscr{S}$ consists of a single object $S$, then we simply write $\Tria_\mathcal{T}(S)$ and $\thick_\mathcal{T}(S)$ for $\Tria_\mathcal{T}(\{S\})$ and $\thick_\mathcal{T}(\{S\})$, respectively.


Let $A$ be a ring. We denote by $A\Modcat$ the category of all unitary left $A$-modules. For an $A$-module $M$, we denote by $\add(M)$ (resp., $\Add(M)$) the full subcategory of $A\Modcat$ consisting of all direct summands of finite (resp., arbitrary) direct sums of copies of $M$. In many circumstances, we shall write $A\pmodcat$ and $A\Pmodcat$ for
$\add(_AA)$ and $\Add(_AA)$, respectively. If $\alpha$ is an index set,
we denote by $M^{(\alpha)}$ the direct sum of $\alpha$ copies of $M$.

If $f: M\ra N$ is a homomorphism of $A$-modules, then the image of
$x\in M$ under $f$ is denoted by $(x)f$ instead of $f(x)$. Also, for
any $A$-module $X$, the induced morphisms $\Hom_A(X,f):
\Hom_A(X,M)\ra \Hom_A(X,N)$ and $\Hom_A(f,X): \Hom_A(N, X)\ra
\Hom_A(M, X)$ is denoted by $f^*$ and $f_*$, respectively.

For simplicity,  we write $\C{A}$, $\K{A}$ and $\D{A}$ for $\C{A\Modcat}$, $\K{A\Modcat}$ and $\D{A\Modcat}$, respectively.  As usual, $A\Modcat$ is always identified with the full subcategory of $\D{A}$ consisting of all stalk complexes
concentrated in degree zero. The mapping cone of a chain map $\cpx{f}$ in $\C{A}$ is denoted by $\con(\cpx{f})$. Particularly, for a homomorphism $f:X \to Y$ of $A$-modules, $\con(f)$ is a two-term complex $0\to X\lraf{f} Y\to 0$ in which $X$ and $Y$ are of degree $-1$ and $0$, respectively.

It is known that $\D{A}$ is a compactly generated triangulated category and $\Dc{A}$ consists of all complexes which are  quasi-isomorphic to bounded complexes of finitely generated projective $A$-modules.
 We denote by $J(A)$ the Jocabson radical of $A$.
For any two integers $r$ and $s$ in $\mathbb{Z}$ with $r\leq s$, the notation $[r,s]$ denotes the set $\{r,r+1,\cdots,s\}$. 
We denote by 
$E_{r,s}$ the matrix units of the $n$ by $n$ matrix ring  over $A$, where $1\leq r,s\leq n$.  

\subsection{Recollements and exact categories}\label{Section 2.2}

In this subsection, we first recall the definition of $t$-structures, recollements of triangulated categories (see \cite{BBD, CPS1}), and the definition of exact categories.

\begin{Def}\rm
Let $\mathcal{D}$ be a triangulated category. A {\it $t$-structure} on $\mathcal{D}$ is a pair $(\mathcal{D}^{\leqslant 0}, \mathcal{D}^{\geqslant 0})$ of strictly full subcategories such that
\smallskip

(1) $\Hom_{\mathcal{D}}(\mathcal{D}^{\leqslant 0}, \mathcal{D}^{\geqslant 0}[-1])=0$;

(2) $\mathcal{D}^{\leqslant 0}[1]\subseteq \mathcal{D}^{\leqslant 0}, \mathcal{D}^{\geqslant 0}[-1]\subseteq \mathcal{D}^{\geqslant 0}$;

(3) for each $M\in \mathcal{D}$, there is a triangle $D_M\ra M\ra D^M\ra D_M[1]$ with $D_M\in \mathcal{D}^{\leqslant 0}, D^M\in \mathcal{D}^{\geqslant 0}[-1]$.
\end{Def}
Given a triangulated category $\mathcal{D}$, there is a standard $t$-structure $(\mathcal{D}^{\leqslant 0},\mathcal{D}^{\geqslant 0})$ defined by
$$
\mathcal{D}^{\leqslant 0}=\{X\in\mathcal{D}|H^i(X)=0\hspace{1mm}\mbox{for all} \hspace{1mm}i>0\},\hspace{2mm}
\mathcal{D}^{\geqslant 0}=\{X\in\mathcal{D}|H^i(X)=0\hspace{1mm}\mbox{for all} \hspace{1mm}i<0\}.
$$
If $(\mathcal{D}^{\leqslant 0}, \mathcal{D}^{\geqslant 0})$ is a $t$-structure on $\mathcal{D}$,
then $(\mathcal{D}^{\leqslant 0},\mathcal{D}^{\geqslant 0}][-1]$ is a torsion pair in $\mathcal{D}$.
The notion of torsion pairs is closely associated with recollements of triangulated categories \cite{CX1}.
\begin{Def}\rm
Let $\mathcal{D}, \mathcal{D'}$ and $\mathcal{D''}$ be triangulated categories.
We say that \emph{$\mathcal{D}$ is a recollement of $\mathcal{D'}$ and
$\mathcal{D''}$} (or there is a recollement among $\mathcal{D'}$, $\mathcal{D}$ and $\mathcal{D}''$) if there are six triangle functors displayed in the
diagram
$$\xymatrix{\mathcal{D''}\ar^-{i_*=i_!}[r]&\mathcal{D}\ar^-{j^!=j^*}[r]
\ar^-{i^!}@/^1.4pc/[l]\ar_-{i^*}@/_1.6pc/[l]
&\mathcal{D'}\ar^-{j_*}@/^1.4pc/[l]\ar_-{j_!}@/_1.6pc/[l]}$$
satisfying the following four conditions:

$(1)$ $(i^*,i_*),(i_!,i^!),(j_!,j^!)$ and $(j^*,j_*)$ are adjoint
pairs, of which the units and counits are denoted by $(\eta^i,\epsilon^i), (^i\eta, {^i}\epsilon,), (\eta^j, \epsilon^j)$ and
$(^j\eta, {^j}\epsilon)$, respectively.

$(2)$ $i_*,j_*$ and $j_!$ are fully faithful.

$(3)$ $i^!j_*=0$ (and thus also $j^! i_!=0$ and $i^*j_!=0$).

$(4)$ Each object $X\in\mathcal{D}$ is endowed with the following triangles in $\mathcal D$:
$$
j_!j^!(X)\lraf{\epsilon^{j}_X} X\lraf{\eta^{i}_X} i_*i^*(X)\lra j_!j^!(X)[1]
\quad\mbox{and}\quad
i_!i^!(X)\lraf{^i\epsilon_{X}} X\lraf{^j\eta_{X}} j_*j^*(X)\lra i_!i^!(X)[1].
$$
\end{Def}

An \emph{exact category}, in the sense of Quillen \cite{Quillen}, is an additive category equipped with a distinguished class of kernel--cokernel pairs, called \emph{conflations}, which is closed under isomorphisms and satisfies Quillen's axioms. By the embedding theorem for exact categories (see \cite{Keller90}), every small exact category admits a fully faithful exact embedding into a module category whose essential image is extension-closed. Thus, up to exact equivalence, an exact category may be viewed as an extension-closed full additive subcategory of a module category.

Let $\mathcal S\subseteq\mathcal A$ be an exact subcategory of the module category $\mathcal{A}$.  A complex $X^\bullet\in\C{\mathcal S}$ is called \emph{strictly exact} if it is exact as a complex in $\mathcal A$ and all its cycles $Z^n(X^\bullet)$ belong to $\mathcal S$. Denoted by $\Kac{\mathcal S}$ the full triangulated subcategory of $\K{\mathcal S}$ consisting of strictly exact complexes. The derived category $\D{\mathcal S}$ of the exact category $\mathcal S$ is defined to be the Verdier quotient of $\K{\mathcal S}$ with respect to $\Kac{\mathcal S}$.

\subsection{Differential graded algebras}\label{Subsection 2.3}
In this subsection, we recall the definition of derived categories of differential graded algebras.

Let $k$ be a commutative ring.  A \emph{differential graded algebra} over $k$ (or \emph{dg $k$-algebra} for short) is a $\IZ$-graded associative and unitary $k$-algebra $\Lambda=\bigoplus_{n\in \IZ}\Lambda^{n}$ endowed with a differential
$d$ of degree one, such that $(\Lambda^n,d^n)_{n\in \IZ}$ is a chain complex of $k$-modules and
the equality $(xy)d^{m+n} = x(yd^{n})+(-1)^n(xd^m)y$ holds
for any $x\in \Lambda^{m}$ and $y\in \Lambda^{n}$.
A \emph{left dg $\Lambda$-module} $\cpx{X}$ is a $\IZ$-graded left module $\cpx{X}= \bigoplus_{n\in \IZ}X^n$
over the $\IZ$-graded $k$-algebra $\Lambda$, with
a differential $d$ of
degree one such that $(X^n,d^n)_{n\in \IZ}$ is a complex of $k$-modules, and for any $a \in \Lambda^m, x\in X^n$, the equality
$(ax)d^{m+n} = a(xd^n)+(-1)^n(ad^m)x$ holds. Note that an ordinary ring $A$ can be regarded as a dg $\IZ$-algebra concentrated in degree $0$, and  a dg $A$-module is exactly a complex of $A$-modules.

A morphism of dg $\Lambda$-modules is a homogeneous morphism of degree $0$ of the underlying graded $\Lambda$-modules commuting with the differentials.
The category of dg $\Lambda$-modules with morphisms is denoted by $\C{\Lambda}$.  We say that a morphism $f:\cpx{X}\ra \cpx{Y}$ in $\C{\Lambda}$ is null-homotopic if $f=dr+rd$ for some homogeneous morphism $r:\cpx{X}\ra \cpx{Y}$ of degree $-1$ of the underlying graded $\Lambda$-modules; a quasi-isomorphism if it is a quasi-isomorphism as a chain map of complexes over $k$,
that is, the map $H^i(f): H^i(\cpx{X})\to H^i(\cpx{Y})$ induced from $f$ is an isomorphism for each $i \in \IZ$. The \emph{dg homotopy category} $\K{\Lambda}$ of $\Lambda$ is the quotient category of $\C{\Lambda}$ with the same objects, but with the  morphisms being the homotopy classes of morphisms of dg $\Lambda$-modules, while the \emph{dg derived category} $\D{\Lambda}$ of $\Lambda$ is the localization of $\K{\Lambda}$ with respect to quasi-isomorphisms.

A dg $\Lambda$-module  $\cpx{X}$ is said to be homotopically projective if  $\Hom_{\K{\Lambda}}(\cpx{X},\cpx{C})=0$ for each acyclic dg $\Lambda$-module
 $\cpx{C}$. In general, $\cpx{X}$ may not be homotopically projective, but there is a quasi-isomorphism $_p\cpx{X}\ra \cpx{X}$ such that  $_p\cpx{X}$ is homotopically projective.
It is known that the localization functor $\K{\Lambda}\to\D{\Lambda}$ restricts to a triangle equivalence from the full subcategory of $\K{\Lambda}$ consisting of homotopically projective dg modules  to $\D{\Lambda}$. This fact is very useful and can be applied to define derived functors.

Let $\cpx{Y}$ be a left dg $\Lambda$-module.
The $\Hom$-complex of $\cpx{X}$ and $\cpx{Y}$ over $\Lambda$ is defined to be
the following complex $\dotHom_\Lambda(\cpx{X},\cpx{Y}):=
(\Hom_\Lambda^n(\cpx{X},\cpx{Y}),d^n_{\cpx{X},\cpx{Y}})_{p\in \IZ}$
over $k$. As a $k$-module, the $n$-th component $\Hom_\Lambda^n(\cpx{X},\cpx{Y})$ is the set of homogeneous morphisms $h : \cpx{X}\ra \cpx{Y}$ of degree $n$ of graded
$\Lambda$-modules, that is, $h$ is a homomorphism of $\Lambda$-modules such that $h=(h^p)_{p\in\IZ}$ with $h^p \in \Hom_k(X^p,Y^{p+n})$. The differential
$d^n_{\cpx{X},\cpx{Y}}$ of degree $n$ is given by
$$(h^p)_{p\in\IZ}\mapsto (h^pd_{\cpx{Y}}^{p+n}-(-1)^nd_{\cpx{X}}^{p}h^{p+1})_{p\in \IZ}.$$
Let $\cpx{Z}$ be another dg $\Lambda$-module and define the following operation
$$
\circ: \;\cpx{\Hom}_\Lambda(\cpx{X},\cpx{Y})\times
\cpx{\Hom}_\Lambda(\cpx{Y},\cpx{Z})\lra
\cpx{\Hom}_\Lambda(\cpx{X},\cpx{Z}),\;\;(f, g)\mapsto
(f^pg^{p+m})_{p\in\mathbb Z}
$$
for $f:=(f^p)_{p\in\mathbb{Z}}\in\Hom_\Lambda^m(\cpx{X},\cpx{Y})$ and
$g:=(g^p)_{p\in\mathbb{Z}}\in\Hom_\Lambda^n(\cpx{Y},\cpx{Z})$ with $m,
n\in\mathbb{Z}$. It can be checked that $\circ$ is associative and
distributive. In particular, $(\cpx{\Hom}_\Lambda(\cpx{X},\cpx{X}),\circ)$ is a
$\mathbb Z$-graded algebra over $k$. Moreover, the
following equality holds:
$$
(f\circ g)\,d^{\,m+n}_{\cpx{X},\cpx{Z}}=f\circ
(g)d^{\,n}_{\cpx{Y},\cpx{Z}}+
(-1)^n(f)d^{\,m}_{\cpx{X},\cpx{Y}}\circ g.
$$
Thus $\cpx{\Hom}_\Lambda(\cpx{X},\cpx{X})$, together with the differential $d_{\cpx{X},\cpx{X}}$ as a complex over $k$, is a dg algebra, called the \emph{dg endomorphism algebra} of $\cpx{X}$ and denoted simply by $\cpx{\End}_\Lambda(\cpx{X})$. Also,
the complex $\cpx{\Hom}_{\Lambda}(\cpx{X},\cpx{Y})$ becomes actually a left dg
$\cpx{\End}_\Lambda(\cpx{X})$- and right dg $\cpx{\End}_\Lambda(\cpx{Y})$-bimodule.
So, $\cpx{X}$ is a dg $\Lambda$-$\dotEnd_\Lambda(\cpx{X})$-bimodule.

Let $\cpx{W}$ be a right dg $\Lambda$-module. The tensor complex of $\cpx{W}$ and $\cpx{X}$ over $\Lambda$
is defined to be the following complex $\cpx{W}\otimes_\Lambda^\bullet \cpx{X}:=(\cpx{W}\otimes_\Lambda^n \cpx{X}, \partial^n_{\cpx{W},\cpx{X}})_{n\in \IZ},$ where $\cpx{W}\otimes_\Lambda^n \cpx{X}$ is the quotient module of $\bigoplus _{p\in \IZ} W^p\otimes _k X^{n-p}$ modulo the
$k$-submodule generated by all elements $wa\otimes x-w\otimes ax$ for $w\in W^r, a\in \Lambda^s$ and $x\in X^t$ with $r,s,t\in \IZ$ and $r+s+t=n.$
The differential $\partial^n_{\cpx{W},\cpx{X}}$ of degree $n$ is given by
$$(w\otimes x)\partial^n_{\cpx{W},\cpx{X}}=(w)d^p_{\cpx{W}}\otimes x+(-1)^pw\otimes (x)d^{n-p}_{\cpx{X}}$$ for $w\in W^p, x\in X^{n-p}.$

Now, the \emph{total right-derived functor} $\rHom_{\Lambda}(-,\cpx{Y}):\D{\Lambda}\ra \D{k}$ of the functor
$\dotHom_{\Lambda}(-,\cpx{Y})$ is
given by $\cpx{X}\mapsto \dotHom_{\Lambda}(_p\cpx{X},\cpx{Y})$, while the \emph{total left-derived functor} $\cpx{W}\otimesL_{\Lambda}-: \D{\Lambda}\ra \D{k}$ of the functor $\cpx{W}\otimes^{\bullet}_{\Lambda}-$ is given by $\cpx{X}\mapsto \cpx{W}\otimes^{\bullet}_{\Lambda}{(_p\cpx{X})}$. In particular, if $\cpx{X}$ is homotopically projective, then $\rHom_{\Lambda}(\cpx{X},\cpx{Y})=\dotHom_\Lambda(\cpx{X},\cpx{Y})$ and $\cpx{W}\otimesL_{\Lambda}\cpx{X} =\cpx{W}\otimes^{\bullet}_{\Lambda}\cpx{X}$.

In the following, we mention three properties about derived functors. For more details on derived functors over dg algebras, we refer the reader to \cite{keller, KellerDGCategories}.

$(1)$ If $\cpx{X}$ is homotopically projective, then the canonical localization functor $\K {\Lambda}\to \D{\Lambda}$ induces an isomorphism: $\Hom_{\K{\Lambda}}(\cpx{X},\cpx{Y}) \simeq\Hom_{\D{\Lambda}}(\cpx{X},\cpx{Y})$.

$(2)$ Let  $\Gamma$ and $\Delta$ be two dg algebras, $\cpx{U}$ a dg $\Lambda$-$\Gamma$ bimodule and $\cpx{W}$ a dg $\Delta$-$\Lambda$-bimodule.
If ${_\Lambda}\cpx{U}$ is homotopically projective, then
$${_\Delta}\cpx{W}\otimesL_\Lambda (\cpx{U}\otimesL_\Gamma-)
\lraf{\simeq} (\cpx{W}\otimesL_\Lambda \cpx{U})\otimesL_\Gamma-= ({_\Delta}\cpx{W}\otimes_\Lambda \cpx{U})\otimesL_\Gamma-
:\D{\Gamma}\ra \D{\Delta}.$$

 $(3)$ If $\theta:\Lambda\ra \Gamma$ is a morphism of
 dg algebras which is a quasi-isomorphism, then the restriction functor $\D{\theta_\ast}:\D{\Gamma}\ra \D{\Lambda}$ is a triangle equivalence (see \cite[Proposition 5.8.3]{keller}).

\medskip
In the paper, we are particularly interested in the case that $\dotEnd_\Lambda(\cpx{X})$ has cohomology concentrated in degree zero. This motivates the following definition.

\begin{Def}\rm
A dg $\Lambda$-module $\cpx{X}$ is  said to be \emph{perfect} if it is homotopically projective and compact
in $\D{\Lambda}$;  \emph{self-orthogonal} in $\D{\Lambda}$ if  $\Hom_{\D{\Lambda}}(\cpx{X},\cpx{X}[n])=0$ for  any $n\neq 0$; \emph{partial tilting} (see {\rm  \cite[Definition 4.1]{BP}}) if it is perfect and self-orthogonal in $\D{\Lambda}$.
\end{Def}

If $\cpx{X}$ is  homotopically projective and self-orthogonal (particularly, partial tilting)  in $\D{\Lambda}$, then  $\dotEnd_\Lambda(\cpx{X})$ has cohomology concentrated in degree zero by the property $(1)$.

For arbitrary rings, a derived $A$-$B$-bimodule is understood as an
object of $\D{A\otimesL_{\mathbb Z}B^{\opp}}$.
Here the derived enveloping algebra is formed using cofibrant dg
$\mathbb Z$-algebra replacements, and tensor--Hom constructions are
performed in the resulting derived module categories; see
\cite[Sections~4.2--4.3]{KellerDGCategories}.
In particular, the quasi-isomorphism $\Delta\to B$ induces a
quasi-isomorphism of derived enveloping algebras
\[
 A\otimesL_{\mathbb Z}\Delta^{\opp}
 \longrightarrow A\otimesL_{\mathbb Z}B^{\opp}.
\]
Extension of scalars therefore transports the dg $A$-$\Delta$-bimodule
$P^\bullet$ to $T^\bullet=P^\bullet\otimesL_\Delta B$, without
changing its underlying left $A$-object. When a flatness hypothesis
allows the ordinary enveloping algebra to be used, this agrees with
the usual strict bimodule notation. All assertions about the underlying
left or right objects, and the transported adjunctions, are independent
of these replacements.

\subsection{The optimal amplitudes of perfect complexes over semiperfect rings}
In this subsection assume that $R$ is a semiperfect ring,  we give a description of amplitude  of any perfect object in $\Kb{\pmodcat{R\opp}}$.

We recall that a ring 
$R$
is called semiperfect if the quotient ring 
$R/J(R)$ is semisimple and every idempotent element of 
$R/J(R)$ can be lifted to an idempotent element of 
$R$. This definition is equivalent that every finitely generated projective right 
$R$-module is a finite direct sum of indecomposable projective modules, and the endomorphism ring of each indecomposable summand is a local ring.

\begin{Lem}\label{lem:radical-projectives}
Let $R$ be a semiperfect ring and let $J=J(R)$. For finitely generated
projective right $R$-modules $P,Q$,
\[
\rad(\Hom_R(P,Q))=
   \{f:P\to Q\mid f(P)\subseteq QJ\}.
\]
Equivalently, $f$ lies in the categorical radical if and only if
$f\otimes_RR/J: P/PJ\longrightarrow  Q/QJ$
is zero.
\end{Lem}

\begin{proof}
Decompose $P$ and $Q$ into finite direct sums of indecomposable
projectives. Their endomorphism rings are local, and their tops are
simple $R/J$-modules. A morphism between indecomposable projectives
which induces a nonzero map between their tops can occur only when
the two tops are isomorphic. The corresponding projectives are then
projective covers of the same simple module and hence are isomorphic;
a morphism inducing an isomorphism on their tops is itself an
isomorphism by Nakayama's lemma. The assertion follows componentwise.
\end{proof}

\begin{Lem}\label{lem:lifting-isomorphism}
Let $R$ be semiperfect and let $f:P\to Q$ be a morphism between
finitely generated projective right $R$-modules. If
\[
   \overline f:P/PJ\longrightarrow Q/QJ
\]
is an isomorphism, then $f$ is an isomorphism.
\end{Lem}

{\it Proof}.
There is an exact sequence 
$$
0\lra \Ker{f}\lra P\lraf{f}Q\lra \Coker(f)\lra 0.
$$
After applying the functor $-\otimes_R R/J$, we have the following diagram
\[
\xymatrix@C=3.8em{
P\otimes_R R/J \ar[r]^{f\otimes R/J}\ar[d]_-{\simeq}
  &Q\otimes_R R/J\ar[r]\ar[d]_-{\simeq}
  &\Coker(f)\otimes_R R/J\ar[r]\ar[d]_-{\simeq}
  &0\\
P/PJ\ar[r]^{\overline{f}}
  &Q/QJ\ar[r]
  &\Coker(f)/\Coker(f)J\ar[r]
  &0.
}
\]
Since $\overline{f}$ is an isomorphism, $ \Coker(f)/\Coker(f)J=0$, and $\Coker(f)$ is finitely generated. Hence
Nakayama's lemma gives $\Coker(f)=0$. Thus $f$ is
surjective. Since $Q$ is projective, it splits, so
$ P\simeq \Ker(f)\oplus Q$.
Under this decomposition, $\overline f$ is the projection
$ \Ker(f)/\Ker(f)J\oplus Q/QJ\lra Q/QJ$.
Its injectivity implies
$\Ker(f)/\Ker(f)J=0$.
The module $\Ker(f)$ is finitely generated projective, so another
application of Nakayama gives $\Ker(f)=0$.
\overpr

\begin{Prop}\label{thm:minimal-amplitude}
Let $R$ be a semiperfect ring, let $J=J(R)$ and
$\overline R=R/J$. Let $0\neq T_R^\bullet$ be a perfect
complex of right $R$-modules. Define
\[
\operatorname{amp}_R(T_R^\bullet)
 :=
 \inf
 \left\{
 d-c
 \ \middle|\
 \begin{array}{l}
 Z^\bullet\simeq T_R^\bullet
     \text{ in }\D{R\opp},\\
 Z^\bullet\in \Kb{\pmodcat{R\opp}},\\
 Z^i=0\text{ for every }i\notin[c,d]
 \end{array}
 \right\}.
\]
Then the set
$\deg_R(T_R^\bullet)
 :=
 \supp\bigl(T_R^\bullet\otimesL_R\overline R\bigr)
 =
 \left\{
 i\in\mathbb Z
 \ \middle|\
 H^i\bigl(T_R^\bullet\otimesL_R\overline R\bigr)\neq0
 \right\}$
is finite and nonempty, and
\begin{equation}\label{eq:optimal-projective-amplitude}
\amp_R(T_R^\bullet)
 =
 \max\deg_R(T_R^\bullet)
 -
 \min\deg_R(T_R^\bullet).
\end{equation}
Equivalently, $\amp_R(T_R^\bullet)
 =
 \max
 \left\{
 i\in\mathbb Z
 \ \middle|\
 H^i\bigl(T_R^\bullet\otimesL_R\overline R\bigr)\neq0
 \right\}
 -
 \min
 \left\{
 i\in\mathbb Z
 \ \middle|\
 H^i\bigl(T_R^\bullet\otimesL_R\overline R\bigr)\neq0
 \right\}$.
 \end{Prop}

{\it Proof}.
We first recall the construction and uniqueness of minimal bounded
projective representatives over a semiperfect ring.
Choose a bounded complex $\cpx{Z}\in\Kb{\pmodcat{R\opp}}$
representing $\cpx{T}_R$. Suppose that, for some $i$, the
differential $d_Z^i:Z^i\lra Z^{i+1}$
does not belong to the categorical radical. Since $R$ is semiperfect,
the modules $Z^i$ and $Z^{i+1}$ are finite direct sums of
indecomposable projective modules with local endomorphism rings.
After decomposing these two terms, the matrix of $d_Z^i$ therefore
contains an invertible component
\[
u:P\xrightarrow{\sim}Q
\]
between indecomposable projective modules. Gaussian elimination in the
additive category $\pmodcat{R\opp}$ yields an isomorphism of
complexes
\[
Z^\bullet
 \cong
Z_1^\bullet
 \oplus
\left(
 0\longrightarrow
 P\xrightarrow{1_P}P
 \longrightarrow0
\right),
\]
where the second summand is contractible.

Set
$\ell(Z^\bullet)
 :=
 \sum_{i\in\mathbb Z}
 \ell_{\overline R}(Z^i/Z^iJ)$.
This number is finite because $R/J$ is semisimple artinian and each $Z^i$
is finitely generated. Every cancellation strictly decreases
$\ell(Z^\bullet)$. Hence the procedure terminates and yields
$\cpx{Z}\simeq\cpx{W}\oplus\cpx{C}$,
where $\cpx{C}$ is contractible and every differential of
$W^\bullet$ belongs to the categorical radical. By
Lemma~\ref{lem:radical-projectives}, this is equivalent to
\begin{equation}\label{eq:minimal-radical-differential}
d_W^i(W^i)\subseteq W^{i+1}J
\qquad\text{for every }i\in\mathbb Z.
\end{equation}
Thus $W^\bullet$ is a minimal bounded projective representative of
$T_R^\bullet$.

We next prove that this minimal representative is unique up to
isomorphism of complexes. Let $W^\bullet$ and $W'^\bullet$ be two
minimal bounded projective representatives of $T_R^\bullet$.
Since $\Kb{\pmodcat{R\opp}}\lra\D{R\opp}$
is fully faithful, they are homotopy equivalent. Choose homotopy inverse
chain maps
\[
f:W^\bullet\longrightarrow W'^\bullet,
\qquad
g:W'^\bullet\longrightarrow W^\bullet.
\]
After tensoring with $R/J$, all differentials vanish by
minimality. The homotopy identities therefore reduce degreewise to
$(\overline g^{\,i})(\overline f^{\,i})
 =1_{W^i/W^iJ},
(\overline f^{\,i})(\overline g^{\,i})
 =1_{W'^i/W'^iJ}$.
Hence
$\overline f^{\,i}:
W^i/W^iJ\lraf{\sim}W'^i/W'^iJ$
is an isomorphism for every $i$. By
Lemma~\ref{lem:lifting-isomorphism}, each
$f^i:W^i\lraf{\sim}W'^i$
is an isomorphism. Thus $f$ is an isomorphism of complexes.

We now compute the cohomological support after reduction
modulo $J$. Since $W^\bullet$ is a bounded complex of projective right
$R$-modules representing $T_R^\bullet$, there is an isomorphism
$T_R^\bullet\otimesL_R\overline R
 \simeq
W^\bullet\otimes_R\overline R$.
Condition \eqref{eq:minimal-radical-differential} implies that the
differential of $W^\bullet\otimes_R\overline R$ is zero. Consequently,
\begin{equation}\label{eq:minimal-reduction-cohomology}
H^i\bigl(T_R^\bullet\otimesL_R\overline R\bigr)
 \cong
W^i\otimes_R\overline R
 \cong
W^i/W^iJ.
\end{equation}
Since $W^i$ is finitely generated, Nakayama's lemma gives
$W^i/W^iJ=0 \Longleftrightarrow
W^i=0$.
It follows from \eqref{eq:minimal-reduction-cohomology} that
\begin{equation}\label{eq:minimal-term-support}
\deg_R(T_R^\bullet)
 =
\{i\in\mathbb Z\mid W^i\neq0\}.
\end{equation}
In particular, this set is finite and, since
$T_R^\bullet\neq0$, nonempty.

It remains to verify optimality.
Let $Z^\bullet$ be any bounded complex of finitely generated
projective right $R$-modules representing $T_R^\bullet$, with
$Z^i=0$ for every $i\notin[c,d]$.
Cancelling elementary contractible summands does not introduce terms
outside $[c,d]$. Hence the minimal part $W_Z^\bullet$ obtained from
$Z^\bullet$ also satisfies
$W_Z^i=0 (i\notin[c,d])$.
By uniqueness of minimal representatives,
$W_Z^\bullet\cong W^\bullet$, so
$W^i=0 (i\notin[c,d])$.
By \eqref{eq:minimal-term-support}, we obtain
$\deg_R(T_R^\bullet)\subseteq[c,d]$.
Thus $\max\deg_R(T_R^\bullet)-\min\deg_R(T_R^\bullet)
 \leq d-c$.
Taking the infimum over all such representatives gives
$\max\deg_R(T_R^\bullet)-\min\deg_R(T_R^\bullet)
 \leq\operatorname{amp}_R(T_R^\bullet)$.
Conversely, \eqref{eq:minimal-term-support} shows that $W^\bullet$
is concentrated in
$\bigl[\min\deg_R(T_R^\bullet),\,
      \max\deg_R(T_R^\bullet)\bigr]$.
Since $W^\bullet$ itself represents $T_R^\bullet$, it follows that
$\amp_R(T_R^\bullet)
 \leq
\max\deg_R(T_R^\bullet)-\min\deg_R(T_R^\bullet)$.
Together these inequalities prove
\eqref{eq:optimal-projective-amplitude}.
\overpr

\section{Derived recollements for big tilting complexes}\label{derecoll}

Throughout this section, let $A$ be a ring with identity.
We begin by recalling the definition of big tilting complexes over rings, and subsequently derive several basic properties of their 
$n$-term analogues.

\subsection{The symmetric exact subcategory of $B\Modcat$}

\begin{Def}\rm\label{Big tilting}
A complex $\bf T\in \C{A}$ is called a \emph{big tilting complex} over $A$ if it is isomorphic in $\D{A}$ to a
bounded complex $\cpx{P}$ of projective $A$-modules such that

(1) $\Hom_{\D{A}}(\cpx{P},{\cpx{P}}^{(\alpha)}[n])=0$ for all $0\neq n\in \mathbb{Z}$ and index sets $\alpha$;

(2) ${_A}A$ belongs to $\thick_{\D A}( \cpx{P})$, the smallest full triangulated subcategory of $\D{A}$ containing $\cpx{P}$ and being closed under direct summands. Thus $\thick_{\D{A}}(A)\subseteq \thick_{\D A}( \cpx{P})$.
\end{Def}
It is called $n$-term if $P^\bullet$ may be chosen with nonzero terms only in
degrees $-(n-1),\ldots,0$.
Notice that both tilting complexes (see \cite{ keller, Rickard}) and good tilting modules of finite projective dimension (see \cite{Bz2,CX1,CX1',ct}) are big tilting complexes. In particular, for a good tilting $A$-module $T$, we can take $\cpx{P}$ in Definition \ref{Big tilting} to be a delete projective resolution of $T$.

%
%
%
%

Let
$$
\cpx{P}:=\cdots\lra0\lra P^{-(n-1)}\lra\cdots\lra P^{-1}\lraf{} P^{0}\lra 0\lra \cdots
$$
be an  $n$-term big tilting complex over $A$, where $P^{-i}$ are projective for all $0\leq i\leq n-1$.  By the definition of big tilting complexes,  the complex $\cpx{P}$ is homotopically projective and self-orthogonal in $\D{A}$. For simplicity, the Hom-complex $\dotHom_A(\cpx{X}, \cpx{Y})$ of the complexes $\cpx{X}, \cpx{Y}\in \C{A}$ is sometimes denoted by
${}_A^\bullet(\cpx{X}, \cpx{Y})$.

Define
 $$B:=\End_{\D{A}}(\cpx{P}), \;\; \Lambda:=\dotEnd_A(\cpx{P}), \;\;\Delta:=\tau_{\leq 0}\Lambda\;\;\mbox{and}\;\; \Gamma:=\dotEnd_{\Lambda^{\opp}}(\cpx{P})^{\opp}.$$
Then $B$ is an ordinary ring,  while $\Lambda$, $\Delta$ and $\Gamma$ are dg algebras.  There are three canonical homomorphisms of dg algebras: the inclusion $\sigma:\Delta\ra \Lambda$,  the surjection  $\pi=H^0: \Delta\ra B$ and the homomorphism $\mu: A\to \Gamma$ given by left multiplication.

Observe that $\cpx{P}$ can be endowed with three dg bimodule structures:  dg $\Gamma$-$\Lambda$-bimodule, dg $A$-$\Lambda$-bimodule and dg $A$-$\Delta$-bimodule, of which the latter two coincide with the restriction of the first bimodule structure via $\mu$ and $\sigma$.
Since  ${_A}\cpx{P}$ is homotopically projective and self-orthogonal in $\D{A}$, both $\sigma$ and $\pi$ are quasi-isomorphisms.  Further, we define
$$F:=(B\otimesL_\Delta-) \D{\sigma_\ast}:\D{\Lambda}\lra \D{B}.$$

\begin{Lem}\rm\label{inverse of F}
 $F$ is a triangle equivalence ,
and a quasi-inverse of $F$ is given by
$$ 
F^{-1}:=\Lambda\otimesL_\Delta\D{\pi_\ast}:\D{B}\lra\D{\Lambda}.
$$
\end{Lem}
{\it Proof}. Since $\sigma$ is a quasi-isomorphism of dg algebras, we know that
$\Lambda$ is equal to $\Delta$ in $\D{\Delta}$, and
$\D{\sigma_\ast}$ is an isomorphism of dg algebras due to property (3).
This implies that $F(\Lambda)=(B\otimesL_{\Delta}-)\D{\sigma_\ast}(\Lambda)=B\otimesL_{\Delta}\Delta=B$,
and it is enough to show that
$B\otimesL_\Delta-:\D{\Delta}\ra \D{B}$ is an equivalence.
If $n\leq 0$, then there are isomorphisms
$$\Hom_{\D{\Delta}}(\Delta,\Delta[n])\simeq H^n(\Delta)\simeq H^n(\tau_{\leq 0}\Lambda)
\simeq H^n(\Lambda)\simeq H^n(\dotEnd_A(\cpx{P}))
\simeq\Hom_{\D{A}}(\cpx{P},\cpx{P}[n]).$$
Thus
\begin{center}
$\aligned\Hom_{\D{\Delta}}(\Delta,\Delta[n])
=&\left\{
   \begin{array}{ll}
     B, & \hbox{if $n=0$;} \\
      & \hbox{} \\
     0, & \hbox{if $n\neq0$.}
   \end{array}
 \right.\\
=&H^n(B)\\
=&\Hom_{\D{B}}(B,B[n]).\endaligned$
\end{center}
If $n>0$, then  $\Hom_{\D{\Delta}}(\Delta,\Delta[n])\simeq H^n(\Delta)\simeq H^n(\tau_{\leq 0}\Lambda)=0
=H^n(B)
=\Hom_{\D{B}}(B,B[n])$.
Therefore, for any $n\in\IZ$,  we get that $\Hom_{\D{\Delta}}(\Delta,\Delta[n])\simeq \Hom_{\D{B}}(B,B[n])$.
Since $B$ is a compact generator of $\D{B}$,
by \cite[Lemma 4.2]{keller}, we get that $B\otimesL_\Delta-:\D{\Delta}\ra \D{B}$ is an equivalence.

Since $\pi:\Delta\ra B$ is a quasi-isomorphism, so $B=\Delta$ in $\D{\Delta}$. That is, $\D{\pi_{\ast}}(B)=\Delta$. Similarly, we can prove that $\Lambda\otimesL_\Delta\D{\pi_\ast}$ is an equivalence with $(\Lambda\otimesL_\Delta\D{\pi_\ast})F=\Id_{\D{\Lambda}}$ and $F(\Lambda\otimesL_\Delta\D{\pi_\ast})=\Id_{\D{B}}$.
Thus $F^{-1}:=\Lambda\otimesL_\Delta\D{\pi_\ast}$ is the quasi-inverse of $F$.
\overpr

\medskip

Let
$$\mathcal{C}^m_\Lambda:=\{\cpx{M}\in \D{\Lambda}\mid H^n(\cpx{M})=0\ \mbox{for all}\ n\neq m\},$$
$$\mathcal{C}^m_B:=\{\cpx{N}\in \D{B}\mid H^n(\cpx{N})=0\ \mbox{for all}\ n\neq m\}.$$
For simplicity, put $\mathcal{C}_\Lambda:=\mathcal{C}^0_\Lambda $ and $\mathcal{C}_B:=\mathcal{C}^0_B$.
Clearly, $\mathcal{C}_\Lambda$ and $\mathcal{C}_B$ are closed under extensions.

\begin{Lem}\label{Equivalence}
$(1)$ The functor $F$ induces an equivalence
$\mathcal{C}^m_\Lambda\lraf{\simeq}\mathcal{C}^m_B$ for all $m\in\IZ$.

$(2)$ The functors $F:\mathcal{C}_\Lambda\ra \D{B}$ and $H^0:\mathcal{C}_\Lambda\ra \D{B}$
are naturally isomorphic.

$(3)$ The composition of $\dotHom_A (\cpx{P},-): \D{A}\to\D{\Lambda}$ with $F$ induces a triangle equivalence
$$\thick_{\D{A}}( \cpx{P})\lraf{\simeq}\Dc{B}.$$
\end{Lem}

{\it Proof}.
$(1)$  Let  $\cpx{M}\in \D{\Lambda}$ and $m\in\mathbb{Z}$. Then there is a series of natural isomorphisms
$$
\begin{aligned}
H^m(\cpx{M})&\simeq \Hom_{\D{\Lambda}}(\Lambda,\cpx{M}[m])
\simeq \Hom_{\D{B}}(F(\Lambda),F(\cpx{M})[m])\\
&\simeq \Hom_{\D{B}}(B,F(\cpx{M})[m])
\simeq H^m(F(\cpx{M})).\qquad(\clubsuit)
\end{aligned}
$$
Consequently, $\cpx{M}\in\mathcal{C}^m_\Lambda$ if and only if $F(\cpx{M})\in\mathcal{C}^m_B$. Since $F$ is an equivalence, $(1)$ holds.

$(2)$ When $n=0$, it follows from $(\clubsuit)$ that $H^0(\cpx{M})\simeq H^0(F(\cpx{M}))$ for all
$\cpx{M}\in\D{\Lambda}$. This implies that
the functors
$H^0F:\D{\Lambda} \ra B\Modcat$ and $H^0:\D{\Lambda} \ra  B\Modcat$ are naturally isomorphic.
Suppose $\cpx{M}\in \mathcal{C}_\Lambda$.
Then $F(\cpx{M})\in \mathcal{C}_B$ by $(1)$. In other words, $H^n(F(\cpx{M}))=0$ for any $n\neq 0$. By the canonical truncation of complexes,  we obtain a series of natural isomorphisms in $\D{B}$
$$ F(\cpx{M})\simeq \tau_{\leq 0}F(\cpx{M})\simeq H^0F(\cpx{M})\simeq H^0(\cpx{M}).$$
Thus $(2)$ holds.

$(3)$ Recall that $\cpx{P}$ is a dg $A$-$\Lambda$-bimodule and homotopically projective as a dg $A$-module. So,
there exists an adjoint pair $(\cpx{P}\otimesL
_{\Lambda}-, \dotHom_A (\cpx{P},-))$ between $\D{\Lambda}$ and $\D{A}.$  This is also an adjoint pair between
$\thick_{\D{\Lambda}}( \Lambda)$ and $\thick_{\D{A}}( \cpx{P})$ because $\cpx{P}\otimesL_{\Lambda}\Lambda\simeq \cpx{P}$ in $\D{A}$ and $\dotHom_A (\cpx{P}, \cpx{P})\simeq \Lambda$ in $\D{\Lambda}$. 

We \textbf{claim} that $\dotHom_A (\cpx{P},-): \D{A}\to\D{\Lambda}$ restricts to a triangle equivalence from $\thick_{\D{A}}( \cpx{P})$ to $\thick_{\D{\Lambda}}( \Lambda) $.

Let  $\eta: \mbox{id}_\D{\Lambda}\ra \dotHom_A (\cpx{P},\cpx{P}\otimesL_{\Lambda}-)$ be the unit adjunction, and let
$\mathcal{X}$ be the full subcategory of $\D{\Lambda}$ consisting of $\cpx{X}$ such that $\eta_\cpx{X}$ is an isomorphism.
Then the restriction of
$\cpx{P}\otimesL_\Lambda-$ to $\mathcal{X}$ is fully faithful.
It is easy to check that $\mathcal{X}$ is a thick triangulated category of $\D{\Lambda}$.
Since $\Lambda[n]\in\mathcal{X}$ for any $n\in\mathbb{Z}$,
we have $\thick_{\D{\Lambda}}( \Lambda)\subseteq \mathcal{X}$. It follows that the restriction of  $\cpx{P}\otimesL_\Lambda-$ to $\thick_{\D{\Lambda}}( \Lambda)$ is also fully faithful. Similarly, we can show that the restriction of
$\dotHom_A (\cpx{P},-)$ to $\thick_{\D{A}}( \cpx{P})$ is fully
faithful.
Thus $\thick_{\D{\Lambda}}(\Lambda)$ is triangle equivalent to $\thick_{\D{A}}(\cpx{P})$. This shows the claim.

Since $F$ is an equivalence,  it restricts to a triangle equivalence $\Dc{\Lambda}\lraf{\simeq}\Dc{B}$.
Now, $(3)$ follows from the equality $\Dc{\Lambda}=\thick_{\D{\Lambda}}( \Lambda)$. \overpr

\medskip

It follows from Lemma \ref{Equivalence} (1) that $F$ preserves cohomology and restricts to an equivalence between
the standard hearts (induced by standard $t$-structures).
Since ${_A}\cpx{P}$ is homotopically projective,  the  localization functor $\K{A}\to \D{A}$ induces an equivalence
 $\thick_{\K A}(\cpx{P})\lraf{\simeq} \thick_{\D A}(\cpx{P})$.
Let $G$ be the composition of $\dotHom_A (\cpx{P},-): \D{A}\to\D{\Lambda}$ with $F:\D{\Lambda}\lraf{\simeq} \D{B}$.
By Lemma \ref{Equivalence} (3),  the functor  $G$ induces an equivalence $G_0: \thick_{\D A}(\cpx{P})\lraf{\simeq}\Dc{B}$.  Thus there are triangle equivalences $\thick_{\K{A}}(\cpx{P})\stackrel{\simeq}{\lra} \thick_{\D{A}}(\cpx{P})\lraf{\simeq}\Dc{B}$.
The following result collects some basic properties of $n$-term big tilting complexes.

\begin{Lem}\label{Basic properties}
$(1)$  There exists a triangle in
$\thick_{\K A}(\cpx{P})$:

\begin{center}
$(\sharp)\quad\quad  A\lra \cpx{Q}_1\lra \cpx{P}_0\lra A[1]$,
\end{center}
where $\cpx{P}_0, \cpx{Q}_1\in \thick_{\D{A}}(\cpx{P})$.

$(2)$\ The map $\mu:A\ra \Gamma$ is a quasi-isomorphism of dg algebras.

$(3)$ $\cpx{P}$ is a partial tilting dg $\Lambda^{\opp}$- module and there is a ring isomorphism
$\End_{\D{\Lambda^{\opp}}}(\cpx{P})\opp\simeq A$.

$(4)$ There is a natural isomorphism of triangle functors
$$\dotHom_A(\cpx{P},A)\otimesL_A- \lraf{\simeq}\dotHom_{\Lambda^{\opp}}(\cpx{P},\Lambda)\otimesL_A-:\; \D{A}\lra \D{\Lambda}.$$
\end{Lem}

{\it Proof}. $(1)$\  Let
$\cpx{U}:=G_0(A)$. Then
$\cpx{U}\in \Dc{B}$ because of $A\in \thick_{\K A}( \cpx{P})$. So, $\cpx{U}$ is isomorphic in $\D{B}$ to a bounded complex $\cpx{V}$ of finitely generated projective $B$-module.
We may assume that
$$\cpx{V}:=\cdots\lra0\lra V^{-r}\lra V^{-(r-1)}\lra\cdots\lra V^{s-1}\lra V^{s}\lra0\lra\cdots$$
with  $V^{i}\in B\pmodcat$
for all $-r\leq i\leq s$.

\textbf{Claim that:} $\cpx{V}$ is in $\D{B}$ isomorphic to a complex of the form $$ \cdots\lra0\lra V^{0}\lra V^{1}\lra\cdots\lra V^{n-1}\lra0\lra\cdots$$
with $V^{i}\in \add({_B}B)$  for $i=0,1,\cdots,n-1$.
Recall that $\Lambda=\dotHom_A (\cpx{P}, \cpx{P})$ and $F(\Lambda)\simeq B$ in $\D{B}$. This implies
$G(\cpx{P})\simeq B$. Since $A\in \thick_{\K A}( \cpx{P})$ and $G$ is an equivalence, for each $m\in\mathbb{Z}$, there are a series of isomorphisms
$$\Hom_{\D{B}}(G(A), B[m])\simeq \Hom_{\D{B}}(G(A), G(\cpx{P})[m])\simeq\Hom_{\D A}(A, \cpx{P}[m])=H^m(\cpx{P}).
$$
On the other hand, we have that
$$\Hom_{\D{B}}(G(A), B[m])\simeq \Hom_{\D{B}}(\cpx{U}, B[m])\simeq\Hom_{\D{B}}(\cpx{V}, B[m]).$$
Since $\cpx{V}$ is homotopically projective, there
exits an isomorphism  $\Hom_{\K B}(\cpx{V}, B[m])\simeq \Hom_{\D B}(\cpx{V}, B[m])$. Thus $H^m(\cpx{P})\simeq \Hom_{\K B}(\cpx{V}, B[m])$.
By the form of $\cpx{P}$, $\Hom_{\K B}(\cpx{V}, B[m])=0$ for $m\geq 1$ or $m\leq -n$.
Since $\cpx{V}\in \Dc{B}$,
so, for any  $X\in \add({_BB})=\pmodcat{B}$,  $\Hom_{\K B}(\cpx{V}, X[m])=0$ whenever $m\geq 1$ or $m\leq -n$.
Thus,
if $r\geq 1$, then $\Hom_{\K B}(\cpx{V}, V^{-r}[r])=0$. This implies that
$\cpx{V}$ is isomorphic in $\D{B}$ to the complex
$$\cdots\lra0\lra \widetilde{V}{^{-(r-1)}}\lra\cdots\lra V^{s-1}\lra V^{s}\lra0\lra\cdots.$$
Finally, we have that $\cpx{V}=\cdots\ra0\ra V^{0}\ra\cdots\ra V^{s-1}\ra V^{s}\ra0\ra\cdots.$
On the other hand, there are a series of isomorphisms
$$\aligned H^m(G(A))&\simeq \Hom_{\D{B}}(B, G(A)[m])\simeq \Hom_{\D{B}}(G(\cpx{P}),  G(A)[m])\\
&\simeq
\Hom_{\D{A}}(\cpx{P}, A[m]) \simeq \Hom_{\K{A}}(\cpx{P}, A[m]).
\endaligned$$
This implies that $H^m(G(A))=0$ if $m\geq n$.
So $H^m(\cpx{V})=0$ if $m\geq n$
due to $G(A)=\cpx{U}\simeq \cpx{V}$.
Thus $H^{\geq n}(\cpx{V})=0$. Note that $\cpx{V}$ is a bounded complex of finitely generated projective $B$-module. 
Therefore, the complex $\cpx{V}$ has the following form
$$ \cdots\lra0\lra V^{0}\lra V^{1}\lra\cdots\lra V^{n-1}\lra0\lra\cdots$$
with  $V^{i}\in \add({_B}B)$  for $i=0,1,\cdots,n-1$. This finishes the claim.
\medskip

Clearly, there exists a  triangle
$\cpx{V}_{\geq 1}\ra\cpx{V}\ra V^0\ra \cpx{V}_{\geq 1}[1]$ in $\Kb{\pmodcat B}$,
where $\cpx{V}_{\geq 1}$ is a brutal truncation of
$\cpx{V}$ with the form $\cdots\lra0\lra 0\lra V^{1}\lra\cdots\lra V^{n-1}\lra0\lra\cdots$.
Recall that $G_0$ is an equivalence $\thick_{\D A}(\cpx{P})\lraf{\simeq}\Dc{B}$.
Let $G_0^{-1}$ be a quasi-inverse of $G_0$.
Now, we apply $G_0^{-1}$ to the triangle above,
and obtain a triangle
$\cpx{P}_0[-1]\to A\to \cpx{Q}_1\to \cpx{P}_0$ in  $\thick_{\D A}(\cpx{P})$,
where $\cpx{P}_0=G_0^{-1}(\cpx{V}_{\geq 1})[1]$ and $\cpx{Q}_1=G_0^{-1}(V^0)$.

$(2)$  There are some triangles in $\K{\pmodcat B}$
$$\aligned
  \cpx{V}\lra V^0&\lra \cpx{V}_{\geq 1}[1]\lra \cpx{V}[1],\\
 \cpx{V}_{\geq 1}[1]\lra V^1[2] &\lra \cpx{V}_{\geq 2}[2]\lra \cpx{V}_{\geq 1}[2],\\
 \cpx{V}_{\geq 2}[2]\lra V^2[4] &\lra \cpx{V}_{\geq 3}[3]\lra \cpx{V}_{\geq 2}[3],\\
\hspace{1.2cm}\vdots\hspace{1.2cm}\vdots&\hspace{1.2cm}\vdots\\
  \cpx{V}_{\geq n-3}[n-3]\lra V^{n-3}[2\times(n-3)]&\lra \cpx{V}_{\geq n-2}[n-2]\lra \cpx{V}_{\geq n-3}[n-2],\\
 \cpx{V}_{\geq n-2}[n-2]\lra V^{n-2}[2\times(n-2)] &\lra \cpx{V}_{\geq n-1}[n-1]=V^{n-1}[2n-2]\lra \cpx{V}_{\geq n-2}[n-1].
\endaligned$$
Apply $G_0^{-1}$ to the triangles above, and obtain that triangles
$$(\ddag)\quad\quad\aligned
A&\lra\cpx{Q}_1 \lra \cpx{P}_0\lra A[1],\\
\cpx{P}_0&\lra\cpx{Q}_2\lra \cpx{P}_1\lra \cpx{P}_0[1],\\
\cpx{P}_1&\lra\cpx{Q}_3 \lra \cpx{P}_2\lra \cpx{P}_1[1],\\
\hspace{1.2cm}&\vdots\hspace{1.2cm}\vdots\hspace{1.2cm}\vdots\\
\cpx{P}_{n-4} &\lra\cpx{Q}_{n-2}\lra \cpx{P}_{n-3}\lra \cpx{P}_{n-4}[1],\\
\cpx{P}_{n-3}&\lra\cpx{Q}_{n-1} \lra \cpx{P}_{n-2}\lra \cpx{P}_{n-3}[1] 
\endaligned$$
with
 $\cpx{Q}_{r}=G^{-1}_{0}({\cpx{V}}^{r+1}[2\times (r+1)])\in \thick_{\D{A}}(\cpx{P})$
and $\cpx{P}_{s}={G_{0}}^{-1}({\cpx{V}}_{\geq i+1}[i+1])\in \thick_{\D{A}}(\cpx{P})$, where  $1\leq r\leq n-1$ and $ 0\leq s\leq n-2$.
According to the last triangle, we obtain a triangle
$${}_{\Lambda^{\opp}}^\bullet(_{A}^\bullet(\cpx{P}_{n-3},\cpx{P}),\cpx{P})
\ra {}_{\Lambda^{\opp}}^\bullet(_{A}^\bullet(\cpx{Q}_{n-1},\cpx{P}),\cpx{P})
\ra {}_{\Lambda^{\opp}}^\bullet(_{A}^\bullet(\cpx{P}_{n-2},\cpx{P}),\cpx{P})
\ra {}_{\Lambda^{\opp}}^\bullet(_{A}^\bullet(\cpx{P}_{n-3},\cpx{P}),\cpx{P})[1],
$$
and a commutative diagram in $\K{A}$
$$
\xymatrix{\cpx{P}_{n-3}\ar[r]^-{}\ar[d]^-{\rho_{\cpx{P}_{n-3}}}
&\cpx{Q}_{n-1}\ar[r]^-{\varphi}\ar[d]^-{\rho_{\cpx{Q}_{n-1}}}
&\cpx{P}_{n-2}\ar[r]^-{}\ar[d]^-{\rho_{\cpx{P}_{n-2}}}
&\cpx{P}_{n-3}[1]\ar[d]^-{\rho_{\cpx{P}_{n-3}[1]}}\\
{}_{\Lambda^{\opp}}^\bullet(_{A}^\bullet(\cpx{P}_{n-3},\cpx{P}),\cpx{P})
\ar[r]^-{} 
&{}_{\Lambda^{\opp}}^\bullet(_{A}^\bullet(\cpx{Q}_{n-1},\cpx{P}),\cpx{P})
\ar[r]^-{} 
&{}_{\Lambda^{\opp}}^\bullet(_{A}^\bullet(\cpx{P}_{n-2},\cpx{P}),\cpx{P})
\ar[r]^-{}
&{}_{\Lambda^{\opp}}^\bullet(_{A}^\bullet(\cpx{P}_{n-3},\cpx{P}),\cpx{P})[1].}
$$
where $\rho_{\cpx{P}_{n-2}}$ and $\rho_{\cpx{Q}_{n-1}}$ are induced from $\dotHom_A(-, \cpx{P})$.  Since $\cpx{Q}_{n-1}, \cpx{P}_{n-2}\in\add(\cpx{P})$, both $\rho_{\cpx{P}_{n-2}}$ and $\rho_{\cpx{Q}_{n-1}}$ are isomorphisms in $\K{A}$.
Thus
$\rho_{\cpx{P}_{n-3}}$ is also an isomorphism in $\K{A}$.
By iteration, we get that
$\mu:A\ra {}_{\Lambda^{\opp}}^\bullet(_{A}^\bullet(A,\cpx{P}),\cpx{P})=\Gamma$
is a quasi-isomorphism.

$(3)$ Since $\mu:A\ra \Gamma$
is a quasi-isomorphism. This implies that  $ A\simeq H^0(\Gamma)$ and $H^n(\Gamma)=0$ for any $n\neq 0$.
Thus $\cpx{P}_\Lambda$ is self-orthogonal in $\D{\Lambda^{\opp}}$.
Applying $\Hom^\bullet_A(-,\cpx{P})$ to the last triangle in $(\ddag)$ yields
a triangle in $\K{\Lambda^{\opp}}$:
$$
{}_A^\bullet (\cpx{P}_{n-2},\cpx{P})\lra {}_A^\bullet (\cpx{Q}_{n-1},\cpx{P})\lra
{}_A^\bullet (\cpx{P}_{n-3},\cpx{P})\lra {}_A^\bullet (\cpx{P}_{n-2},\cpx{P})[1].
$$
Notice that ${}_A^\bullet (\cpx{X},\cpx{P})$ is compact and homotopically projective in
$\D{\Lambda^{\opp}}$ for any $\cpx{X}\in \add(\cpx{{_A}P})$, and so is
${}_A^\bullet (\cpx{P}_{n-2},\cpx{P})$.
By iteration, we get that ${}_A^\bullet (A,\cpx{P})\simeq
\cpx{P}_\Lambda$ is compact and homotopically projective in
$\D{\Lambda^{\opp}}$. Thus the dg module $\cpx{P}_\Lambda$ is partial tilting and
$A\simeq H^0(\Gamma)=\End_{\K{\Lambda\opp}}(\cpx{P})\opp\simeq\End_{\D{\Lambda\opp}}(\cpx{P})\opp$.

$(4)$ By $(\ddag)$, we have a triangle in
$\K{\Lambda^{\opp}}$
$${}_A^\bullet(\cpx{P}_{n-2},\cpx{P})
\lra{}_A^\bullet(\cpx{Q}_{n-1},\cpx{P})
\lra {}_A^\bullet(\cpx{P}_{n-3},\cpx{P})
\lra{}_A^\bullet(\cpx{P}_{n-2},\cpx{P})[1],$$
and a triangle in $\K{\Lambda}$
$${}_A^\bullet(\cpx{P},\cpx{P}_{n-3})\lra {}_A^\bullet(\cpx{P},\cpx{Q}_{n-1})
\lra {}_A^\bullet(\cpx{P},\cpx{P}_{n-2})\lra {}_A^\bullet(\cpx{P},\cpx{P}_{n-3})[1].$$
This leads to the following commutative diagram in
$\K{\Lambda}$
$${\footnotesize
\xymatrix{{}_A^\bullet(\cpx{P},\cpx{P}_{n-3})\ar[r]^-{}\ar[d]_-{f}
&{}_A^\bullet(\cpx{P},\cpx{Q}_{n-1})\ar[r]^-{}\ar[d]_-{g} 
&{}_A^\bullet(\cpx{P},\cpx{P}_{n-2})\ar[r]^-{}\ar[d]_-{h}
&{}_A^\bullet(\cpx{P},\cpx{P}_{n-3})[1]\ar[d]_-{}\\
{}_{\Lambda^{\opp}}^\bullet({}_A^\bullet(\cpx{P}_{n-3},\cpx{P}),\Lambda)
\ar[r]^-{}
&{}_{\Lambda^{\opp}}^\bullet({}_A^\bullet(\cpx{Q}_{n-1},\cpx{P}),\Lambda)\ar[r]^-{}
&{}_{\Lambda^{\opp}}^\bullet({}_A^\bullet(\cpx{P}_{n-2},\cpx{P}),\Lambda)\ar[r]^-{}
&{}_{\Lambda^{\opp}}^\bullet({}_A^\bullet(\cpx{P}_{n-3},\cpx{P}),\Lambda)[1],}
}$$
where $f, g$ and $h$ are induced from the functor
$\dotHom_A(-, \cpx{P})$.  Since $\cpx{Q}_{n-1}, \cpx{P}_{n-2}\in\add(\cpx{P})$, both $g$ and $h$ are isomorphisms.
This implies that $f$ is an isomorphism in $\K{\Lambda}$.
By iteration, there exists a commutative diagram in $\K{\Lambda}$
$$
\xymatrix{{}_A^\bullet(\cpx{P},A)\ar[r]^-{}\ar[d]_-{\alpha}
&{}_A^\bullet(\cpx{P},\cpx{Q_1})\ar[r]^-{}\ar[d]_-{\beta} 
&{}_A^\bullet(\cpx{P},\cpx{P_0})\ar[r]^-{}\ar[d]_-{\gamma}
&{}_A^\bullet(\cpx{P},A)[1]\ar[d]\\
{}_{\Lambda^{\opp}}^\bullet(\cpx{P},\Lambda)
\ar[r]^-{}
&{}_{\Lambda^{\opp}}^\bullet( {}_A^\bullet(\cpx{Q_1},\cpx{P}),\Lambda)\ar[r]^-{}
&{}_{\Lambda^{\opp}}^\bullet( {}_A^\bullet(\cpx{P_0},\cpx{P}),\Lambda)\ar[r]^-{}
&{}_{\Lambda^{\opp}}^\bullet(\cpx{P},\Lambda)[1]}
$$
with $\beta$ and $\gamma$ isomorphisms. This implies that $\alpha$ is also an isomorphism.
In particular, $\alpha$ as a chain map between dg $\Lambda$-modules is a quasi-isomorphism. Note that ${}_A^\bullet(\cpx{P},A)$ and ${}_{\Lambda^{\opp}}^\bullet(\cpx{P},\Lambda)$ are dg $\Lambda$-$A$-bimodules, and that $\alpha$
is a chain map between dg bimodules.  Thus $\alpha$ is a quasi-isomorphism as a chain map between dg bimodules.
This yields a natural isomorphism from $\dotHom_A(\cpx{P},A)\otimesL_A-$ to $\dotHom_\Lambda(\cpx{P},\Lambda)\otimesL_A-$. 
\overpr

\medskip

With the above preparations, we can construct the following
recollement.

\begin{Lem}\label{Recollement}
 There is a recollement of triangulated categories
$$\xymatrix{\mathscr{Y}_{\Lambda}\ar^-{i_*=i_!}[r]
&\D{\Lambda}\ar^-{j^!=j^*}[r]
\ar^-{i^!}@/^1.6pc/[l]\ar_-{i^*}@/_1.6pc/[l]
&\D{A}\ar^-{j_*}
@/^1.6pc/[l]\ar_-{j_!}@/_1.6pc/[l]}$$
where $\mathscr{Y}_{\Lambda}:=\Ker(j^*)$, $i_*$ is the canonical inclusion and
$$j_!=\dotHom_A(\cpx{P},A)\otimesL_A-, \;\; j^!={_A}\cpx{P}\otimesL_\Lambda-, \;\; j_*=\dotHom_A(\cpx{P},-).$$
\end{Lem}

{\it Proof.}  Recall that $\cpx{P}$ is a dg $\Gamma$-$\Lambda$-bimodule. By  Lemma \ref{Basic properties}, $\cpx{P}$ is a partial tilting dg $\Lambda^{\opp}$-module.
It follows from \cite[Proposition 3.2]{JOR} that there is a recollement
of triangulated categories
$$\xymatrix@C1.5cm{\mathscr{Y}_{\Lambda}\ar^-{i_*=i_!}[r]
&\D{\Lambda}\ar^-{{_\Gamma}\cpx{P}\otimesL_\Lambda-}[r]
\ar^-{i^!}@/^1.6pc/[l]\ar_-{i^*}@/_1.6pc/[l]
&\D{\Gamma}\ar^-{\rHom_\Gamma(\cpx{P},-)}
@/^1.6pc/[l]\ar_-{\dotHom_\Lambda(\cpx{P},\Lambda)
\otimesL_\Gamma-}@/_1.6pc/[l]}$$
where $\mathscr{Y}_\Lambda:=\Ker({_\Gamma}\cpx{P}\otimesL_\Lambda-)$ and $i_*$ is the canonical inclusion. Since
$ \mu: A\to \Gamma$ is a quasi-isomorphism, the restriction functor $\D{\mu_*}:\D{\Gamma}\ra\D{A}$ is a triangle equivalence.
Consequently, its left adjoint $\Gamma\otimesL_A-$ and also its right adjoint $\rHom_A(\Gamma, -)$ are equivalences.
Note that
$$(\dotHom_\Lambda(\cpx{P},\Lambda)\otimesL_\Gamma-)(\Gamma\otimesL_A-)
\simeq (\dotHom_\Lambda(\cpx{P},\Lambda)\otimes_\Gamma\Gamma)\otimesL_A-\simeq\dotHom_\Lambda(\cpx{P},\Lambda)\otimesL_A-\simeq \dotHom_A(\cpx{P},A)\otimesL_A-$$
where the last isomorphism follows from Lemma \ref{Basic properties} (4).
Moreover, $\D{\mu_*}({_\Gamma}\cpx{P}\otimesL_\Lambda-)={_A}\cpx{P}\otimesL_\Lambda-$ and $$
\rHom_\Gamma(\cpx{P},\rHom_A(\Gamma, -))\simeq \rHom_A(\Gamma\otimesL_\Gamma\cpx{P},-)\simeq \rHom_A(\cpx{P},-)=\dotHom_A(\cpx{P},-).$$
In the above recollement, we replace $\D{\Gamma}$ with $\D{A}$ up to equivalence, and then obtain the recollement required in Lemma \ref{Recollement}.
$\square$
\medskip

Notice that $\cpx{P}$ is a $A$-$\Delta$ dg bimodule and  $B$ is a $\Delta$-$B$ dg bimodule due to the homomorphisms $\sigma:\Delta\ra \Lambda$ and $\pi: \Delta\ra B$.
This implies that $\cpx{T}:=\cpx{P}\otimesL_{\Delta} B\in \D{A\otimes_{\mathbb{Z}} B^{op}}$.

\begin{Lem}\label{Rec equivalence}
 The functor $F$ induces an equivalence between the following two recollements
$$\xymatrix@C=3.4em@R=2.3em{\mathscr{Y}_{\Lambda}\ar[r]^-{i_*=i_!}\ar[d]_-{\simeq}^-{F}
&\D{\Lambda}\ar[r]^-{j^!=j^*}
\ar@/^1.1pc/[l]_-{i^!}
\ar_-{i^*}@/_1.2pc/[l]
\ar[d]^-{F}_-{\simeq}
&\D{A}\ar_-{j_*}@/^1.1pc/[l]\ar_-{j_!}@/_1.2pc/[l]\ar[d]_-{\simeq}^-{F j_!}\\
\mathscr{Y}_B\ar[r]^-{\bf j}&
\D{B}\ar[r]^-{\bf R}
\ar@/^1.2pc/[l]^-{\bf K}\ar_-{{\bf L}}@/_1.2pc/[l]
&\Tria_{\D{B}}(Fj_!(A))\ar@/^1.2pc/[l]^-{\bf H}\ar_-{\bf i}@/_1.2pc/[l]}
$$
\noindent where ${\bf i}$ and ${\bf j}$ are canonical inclusions,  ${\bf L}:=Fi^*F^{-1}$ is a left adjoint of ${\bf j}$ and ${\bf R}:=j^*F^{-1}\simeq \cpx{T}\otimesL_B-$ is the left adjoint of the functor ${\bf H}:=Fj_*\simeq \rHom_A(\cpx{T},-)$.
\end{Lem}

{\it Proof}. Let $\cpx{M}\in \D{B}$. 
The functor $F^{-1}=\Lambda\otimesL_\Delta\D{\pi_\ast}$ (see Lemma \ref{inverse of F})
implies that 
$$\aligned{\bf R}(\cpx{M})=j^*F^{-1}(\cpx{M})
=&(\cpx{P}\otimesL_{\Lambda}-)(\Lambda\otimesL_{\Delta}\D{\pi_*}(\cpx{M}))
\simeq \cpx{P}\otimesL_{\Delta}\D{\pi_*}(\cpx{M})
\simeq \cpx{P}\otimesL_{\Delta} \cpx{M}\\
&\simeq (\cpx{P}\otimesL_{\Delta} B)\otimesL_{B}\cpx{M}
=\cpx{T}\otimesL_{B}\cpx{M}.
\endaligned
$$
Thus ${\bf R}\simeq  \cpx{T}\otimesL_B-$.
Since $({\bf R},{\bf H})$ and $(\cpx{T}\otimesL_B-,\rHom_A(\cpx{T},-))$ are adjoint pairs, 
there is a natural isomorphism ${\bf H}\simeq \rHom_A(\cpx{T},-)$.
\overpr

\smallskip

Let $\mbox{For}_A:\D{A\otimes_{\mathbb{Z}} B^{op}}\ra\D{A}$ be the functor
which forgets the derived right $B$-action.  Since $ B\simeq \Delta$ in $\D{\Delta}$,
so
$\mbox{For}_A(\cpx{T})=\mbox{For}_A(\cpx{P}\otimesL_{\Delta}B)=\cpx{P}\otimesL_{\Delta}B\simeq \cpx{P}\otimesL_{\Delta}\Delta\simeq \cpx{P}
$ in $\D{A}$.
Hence
the complex $\cpx{T}$ is not a new left $A$-object: it is just equipped with the derived right $B$-action transported from
$\pi:\Lambda\ra B$.
\medskip

Since $\Lambda\simeq \Delta$ in $\D{\Delta^{op}}$ induced by the homomorphism $\sigma$, we get that 
$\Lambda\otimesL_{\Delta}B \simeq \Delta\otimesL_{\Delta}B\simeq B$ in $\D{B^{op}}$. 
By \cite[Lemma 4.2]{keller}, it is easy to know that  
$-\otimesL_{\Delta}B:\D{\Lambda^{op}}\ra\D{B^{op}}$ is an equivalence.
This implies that $\cpx{T}=\cpx{P}\otimesL_{\Delta}B$ is perfect in $\D{B^{op}}$ because of the compactness of $\cpx{P}$ in $\D{\Lambda^{op}}$ by the Lemma \ref{Basic properties} (3).
This implies that there exists a complex $\cpx{V}=(V^i)$ in $\K{\pmodcat {B^{op}}}$ such that $\cpx{V}\simeq \cpx{T}$ in $\D{B^{op}}$, where $V^i=0$ if $i\notin [a,b]$ for some $a,b\in\IZ$. Put $m=b-a$.

The distinction between $\cpx{T}$ and $V^\bullet$ is important.  The object
$\cpx{T}\in\D{A\otimes_{\mathbb{Z}} B^{\opp}}$ retains the derived left $A$-action and defines
the $A$-valued functor $\bf R$.  The complex $V^\bullet$ records only the
underlying right $B$-object and is not asserted to carry a strict left
$A$-action.  Nevertheless, if
$\mbox{For}_{\mathbb Z}:\D{A}\to\D{\mathbb Z}$ forgets the
$A$-action, then naturally for $\cpx{X}\in\D{B}$. Moreover, there are isomorphisms
\[
 \mbox{For}_{\mathbb Z}({\bf R} (\cpx{X}))
 \simeq T_B\otimesL_B \cpx{X}
 \simeq \cpx{V}\otimesL_B \cpx{X}\simeq \cpx{V}\otimes^{\bullet}_B \cpx{X}.
\]

\begin{Lem} \label{lem:amplitude-new}
 If $i\notin [a,b]$, then
$H^i({\bf R}(M))=0$ for all $M\in B\Modcat$.
\end{Lem}
{\it Proof}.
Let $M\in B\Modcat$, regarded as a stalk complex in degree zero.
Since $\cpx{V}=(V^i)\in \K{\pmodcat {B^{op}}}$ is the homotopically projective resolution of $\cpx{T}_B$, there exist isomorphisms 
\begin{equation}\label{TV}
\cpx{V}\otimes^{\bullet}_BM
\simeq \cpx{V}\otimesL_B M\lraf{\simeq}\cpx{T}\otimesL_B M,
\end{equation}
 where $(\cpx{V}\otimes^{\bullet}_BM)^i=V^i\otimes_BM$
for all $i\in\IZ$. Thus ${\bf R}(M)\simeq \cpx{V}\otimes^{\bullet}_BM$. It follows from $V^i=0$ if $i\notin [a,b]$
that $H^i({\bf R}(M))=0$ when $i\notin [a,b]$.
\overpr

\begin{Def}\rm\cite{CX3}\label{symmetric subcat}
Let $n\in \mathbb{N}$, and let $\mathcal{A}$ be a bicomplete abelian category. An additive subcategory $\mathcal{B}$ of $\mathcal{A}$ is said to be {\it $n$-symmetric} if

(1) $\mathcal{B}$ is closed under extensions, products and coproducts.

(2) For any exact sequence $0\ra X\ra M_n\ra\cdots\ra M_1\ra M_0\ra Y\ra 0$ in $\mathcal{A}$ with all $M_i\in \mathcal{B}$, we have $X,Y\in \mathcal{B}$.
\end{Def}
Define a subcategory of $B\Modcat$ as follow:
$$
 \mathscr{E}=\mathscr{Y}_B\cap B\Modcat
     =\{M\in B\Modcat\mid{\bf R}(M)=0\}.
$$

\begin{Theo}\label{thm:sym-new}
The subcategory $\mathscr{E}$ is an $d$-symmetric subcategory of $B\Modcat$ with $d\geq m$.
\end{Theo}

{\it Proof}.
By the recollements in Lemma \ref{Rec equivalence}, $(\Tria_{\D{B}}(Fj_!(A)),\mathscr{Y}_B,\Img({\bf H}))$ is a TTF-triple in $\D{B}$. This implies that $\mathscr{Y}_B$ is closed under extensions, products and coproducts.
So is $\mathscr{E}$. 
Let
$$
 0\lra X\lra M_m\lra\cdots\lra M_0\lra Y\lra0
$$
be an exact sequence in $B\Modcat$ with  $M_i\in\mathscr{E}$ for all $0\leq i\leq m$. 
By the definition of symmetric subcategories, it is enough to show that $X,Y\in\mathscr{E}$.
Breaking the sequence into short exact
sequences and applying ${\bf R}$ yields an isomorphism
${\bf R}(Y)\simeq{\bf R}(X)[m+1]$ in $\D{A}$.
Both sides before shifting have cohomology in $[a,b]$.  If $a\leq i\leq 0$, then
$H^i({\bf R}(Y))\simeq H^{i+m+1}({\bf R}(X))=0$
due to $i+m+1>b$ by Lemma \ref{lem:amplitude-new}.  Hence ${\bf R}(Y)=0$ in $\D{A}$, that is, $Y\in\Ker({\bf R})$.  Conversely,
$H^i({\bf R}(X))\simeq H^{r-m-1}({\bf R}(Y))=0$
due to $i-m-1<a$ by Lemma \ref{lem:amplitude-new}.  Thus $X,Y\in\mathscr{E}$. The same argument works for every $d\geq m$.
\overpr

\begin{Rem}\rm
The proof uses the fact that $\cpx{T}$ is a perfect object in $\D{B^{\opp}}$.  It does not compute
$\cpx{P}\otimesL_\Lambda M$ by tensoring $P^\bullet$ with a
stalk dg module: an arbitrary dg module may require an unbounded semi-free
resolution, so that naive argument does not control amplitude.
\end{Rem}

The following terminology is standard in the model theory of modules; see
Prest \cite{PrestModelTheory} and \cite[Chapters~1 and~3]{PrestPurity}.  The terminology of definable subcategories in
the setting of locally finitely presented additive categories goes back to
Crawley--Boevey~\cite{CrawleyBoevey}.

\begin{Def}
Let $R$ be a ring.  A \emph{positive primitive formula}, denoted by
\emph{pp formula}, for left $R$-modules is a formula equivalent to
\[
              \varphi(x):\qquad \exists\; y\,(Ux+Vy=0),
\]
where $x=(x_1,\ldots,x_s)^{\mathsf T}$ is a finite tuple of free
variables, $y=(y_1,\ldots,y_t)^{\mathsf T}$ is a finite tuple of
existentially quantified variables, and $U,V$ are finite matrices over
$R$ of compatible sizes.  Thus $Ux+Vy=0$ denotes a finite homogeneous
system of $R$-linear equations.  For a left $R$-module $M$, evaluation
of the formula gives the subgroup
\[
 \varphi(M)=
 \{x\in M^s\mid
   \text{there exists }y\in M^t\text{ with }Ux+Vy=0\}
 \subseteq M^s.
\]
\end{Def}

Let $\varphi$ and $\psi$ be pp formulas with the same tuple of free
variables.  We write $\psi\leq\varphi$ if
$\psi(M)\subseteq\varphi(M)$
for every left $R$-module $M$.  In this case
$\varphi/\psi$ is called a \emph{pp-pair}.  It is evaluated on $M$ as
the quotient group
$(\varphi/\psi)(M)=\varphi(M)/\psi(M)$,
and its \emph{zero class} is
\[
       \{M\in B\Modcat\mid \varphi(M)=\psi(M)\}.
\]
A full subcategory $\mathcal D\subseteq R\Modcat$ is
\emph{definable} if it is the intersection of the zero classes of a family
of pp-pairs.  Equivalently, $\mathcal D$ is closed under arbitrary
products, filtered colimits, and pure submodules
\cite[Theorem~3.4.7]{PrestPurity}.  Recall that $N\subseteq M$ is pure if
every finite system of linear equations with constants in $N$ which has
a solution in $M$ already has a solution in $N$.

The order $\psi\leq\varphi$ already expresses the implication
$\psi(x)\Rightarrow\varphi(x)$ in every module.  Therefore, for a fixed
module $M$, the reverse inclusion, and hence the equality
$\varphi(M)=\psi(M)$, is expressed by the pp-implication
\[
             M\models
             \forall x\,(\varphi(x)\Longrightarrow\psi(x)).
\]

For example, over $R=\mathbb Z$, let
\[
 \varphi(x):(x=x),
 \qquad
 \psi(x):\exists y\,(x=2y).
\]
Then $\varphi(M)=M$, $\psi(M)=2M$, and
$(\varphi/\psi)(M)=M/2M$.
Thus the zero class of this pp-pair consists precisely of the abelian
groups satisfying $M=2M$.

\begin{Rem}\label{rem:cycles-boundaries-pp}
The symbols $\varphi$ and $\psi$ have no independent module-theoretic
meaning; they simply denote two pp formulas, with $\psi\leq\varphi$.
In the proof of Proposition~\ref{prop:definability}, however, they have a
specific homological interpretation:
\[
 \varphi_r(M)=Z^r(V^\bullet\otimes_BM),
 \qquad
 \psi_r(M)=B^r(V^\bullet\otimes_BM).
\]
Thus $\varphi_r$ is the cycle condition, $\psi_r$ is the boundary
condition, and their pp-pair computes cohomology:
\[
                  (\varphi_r/\psi_r)(M)
                  \simeq H^r(V^\bullet\otimes_BM).
\]
This is why the vanishing of the cohomology groups defining $\mathscr{E}$ is a
family of pp-conditions.
\end{Rem}

\begin{Prop}
\label{prop:definability}
The subcategory $\mathscr{E}$ is definable in $B\Modcat$.  In particular, it is
closed under filtered colimits, pure submodules, and direct summands, in
addition to the products, coproducts, and extensions of
Theorem~\ref{thm:sym-new}.
\end{Prop}
{\it Proof}.
Let
$ \mbox{For}_{\mathbb Z}:\D{A}\to \D{\mathbb Z}$
be the forgetful functor. It is conservative: an $B$-complex is
acyclic if and only if its underlying complex of abelian groups is
acyclic. By the natural identification established above,
\[
   \mbox{For}_{\mathbb Z}{\bf R}(X)
      \simeq \cpx{V} \otimesL_B X
      \simeq \cpx{V}\otimes_BX
\]
for every $X\in \D{B}$, since
$\cpx{V}\in \Kb{\Pmodcat B}$ is homotopically projective.
Therefore, for a stalk module $M$,
\[
   M\in\mathscr{E} 
   \iff
   H^r(\cpx{V}\otimes_BM)=0
   \quad\text{for every }r\in\mathbb Z.
\]

We now express each of these vanishing conditions by a pp-pair.
We use column vectors. For every $r$ choose an integer $n_r$ and an
idempotent
$e_r\in M_{n_r}(B)$
such that
$ V^r\simeq e_rB^{n_r}$
as right $B$-modules. Under these identifications, the differential
$d_V^r:V^r\to V^{r+1}$ is represented by a matrix
$ D^r\in M_{n_{r+1}\times n_r}(B)$
satisfying
$D^r=e_{r+1}D^re_r$
and $ D^{r+1}D^r=0$.
For every left $B$-module $M$ there is a natural identification
$V^r\otimes_BM\simeq e_rM^{n_r}$,
under which the differential is induced by left multiplication by
$D^r$.

Define pp-formulas, with $x$ an $n_r$-tuple and $y$ an
$n_{r-1}$-tuple, by
\[
   \phi_r(x)
   :\quad e_rx=x\ \wedge\ D^rx=0,\hspace{2mm}
   \psi_r(x)
   :\quad
   \exists y\,
   \bigl(e_{r-1}y=y\ \wedge\ D^{r-1}y=x\bigr).
\]
These are pp-formulas because all the displayed conditions are finite
homogeneous systems of $B$-linear equations. Moreover,
\[
   D^rD^{r-1}=0,
   \qquad
   e_rD^{r-1}=D^{r-1},
\]
show that $\psi_r(M)\subseteq\phi_r(M)$ for every $M$, hence
$\psi_r\leq\phi_r$.

Under the identifications above,
$\phi_r(M)
      =Z^r(V^\bullet\otimes_BM),
   \psi_r(M)
      =B^r(V^\bullet\otimes_BM)$.
Consequently, we have $(\phi_r/\psi_r)(M)
   =
   \frac{\phi_r(M)}{\psi_r(M)}
   \simeq
   H^r(V^\bullet\otimes_BM)$.
It follows that
$\mathscr{E}
   =
   \bigcap_{r\in\mathbb Z}
   \{M\mid(\phi_r/\psi_r)(M)=0\}$.
Only finitely many nontrivial pp-pairs occur because $\cpx{V}$ is
bounded. Thus $\mathscr{E}$ is definable.

The standard closure theorem for definable subcategories now gives
closure under products, filtered colimits and pure submodules.
Closure under direct summands follows either from that theorem or
from the fact that every split submodule is pure.
\overpr

\subsection{Projective assembly and exact realisation}

This section identifies the triangulated kernel
$\mathscr Y_B=\Ker({\bf R})$ with the derived category of the exact
subcategory $\mathscr E$.  The argument has three stages.  First, the
adjunction unit is analysed on projective stalks.  Second, the resulting
short exact sequences are assembled functorially on bounded-above and then
unbounded complexes.  Finally, the cokernel construction is used to build a
quasi-inverse to the exact inclusion
$\D{\mathscr E}\to\mathscr Y_B$.
\smallskip

For a projective $B$-module $Q$, define
$\Psi(Q):=H^0{\bf HR}( Q)$.

\begin{Lem}\rm
\label{lem:projective-unit-final}
Let $Q\in B\Pmodcat$. Then

(1) $H^i({\bf HR}(Q))=0$ if $i\ne0$;

(2) the homomorphism
$\eta_Q^0:=H^0(\eta_Q):Q\to \Psi(Q)$ is a monomorphism with  $C(Q):=\Coker(H^0(\eta_Q))\in \mathscr{E}$.
Moreover, $\Psi$ and $C$ are additive functors on $B\Pmodcat$, and there is an exact sequence
\begin{equation}\label{eq:projective-unit-final}
 0\longrightarrow Q\lraf{\eta_Q^0} \Psi(Q)
   \longrightarrow C(Q)\longrightarrow0.
\end{equation}
\end{Lem} 

{\it Proof}.
(1) Define $\mathscr{X}=\{X\in\Pmodcat B|H^i{\bf HR}(X)=0\hspace{1mm}\mbox{ if}\hspace{1mm} i\neq 0\}$.
Then we have the following isomophisms
$$\aligned H^i{\bf HR}(B)&=H^i\rHom_A(\cpx{T},\cpx{T}\otimesL_BB)\\
&\simeq H^i\rHom_A(\cpx{T},\cpx{T})\\
&\simeq \Hom_{\D{A}}(\cpx{T}, \cpx{T}[i])\\
&= \Hom_{\D{A}}(\cpx{P}\otimesL_{\Delta} B, \cpx{P}\otimesL_{\Delta} B[i])\\
&\simeq \Hom_{\D{A}}(\cpx{P}, \cpx{P}[i]).
\endaligned$$
Since $\cpx{P}$ is a big tilting complex in $\D{A}$, we get that $B\in\mathscr{X}$.
For any index set $\alpha$, 
$$\aligned H^i{\bf HR}(B^{(\alpha)})&=H^i\rHom_A(\cpx{T},\cpx{T}\otimesL_BB^{(\alpha)})\\
&\simeq H^i\rHom_A(\cpx{T},{\cpx{T}}^{(\alpha)})\\
&\simeq \Hom_{\D{A}}(\cpx{T}, {\cpx{T}}^{(\alpha)}[i])\\
&= \Hom_{\D{A}}(\cpx{P}\otimesL_{\Delta} B, \cpx{P}\otimesL_{\Delta} B^{(\alpha)}[i])\\
&\simeq \Hom_{\D{A}}(\cpx{P}, {\cpx{P}}^{(\alpha)}[i]).
\endaligned$$
This implies that $B^{(\alpha)}\in \mathscr{X}$. Clearly, $\mathscr{X}$ is closed under direct summands.
Thus $\Pmodcat B=\mathscr{X}$.

\bigskip

(2)   If $Q=B$, then $H^0(\eta_Q)$ is equal to $\eta_Q:B=\Hom_{\D{A}}(\cpx{P},\cpx{P})\lraf{\simeq} \Psi(B)$.
This is an isomorphism of rings. Suppose that $Q=B^{(\alpha)}$ for some index set $\alpha$. There is a commutative diagram 
$$
\xymatrix{B^{(\alpha)}\ar"1,2"^(0.45){H^0(\eta_{B^{(\alpha)}})}\ar"2,1"^(0.5){\simeq}
&\Psi(B^{(\alpha)})\ar"2,2"^(0.5){\simeq}\\
\Hom_{\D{A}}(\cpx{P},\cpx{P})^{(\alpha)}\ar"2,2"^(0.5){\theta}& \Hom_{\D{A}}(\cpx{P},{\cpx{P}}^{(\alpha)} ).
}
$$
Notice that $\theta$ is always a monomorphism of abelian groups. This implies that $H^0(\eta_{B^{(\alpha)}})$ is also a monomorphism.
Let $Q_1\oplus Q_2\simeq B^{(\alpha)}$ for some index set $\alpha$. There is a commutative diagram
$$\xymatrix@R1.2cm@C1.5cm{
Q_1\oplus Q_2\ar@{>}"1,2"^(0.52){\simeq} \ar@{=}"2,1"^(0.52){}
&B^{(\alpha)}\ar@{>}"1,3"^(0.47){H^0(\eta_{B^{(\alpha)}})}
&\Psi(B^{(\alpha)})\ar@{>}"1,4"^(0.35){\simeq}
&\Psi(Q_1)\bigoplus \Psi(Q_2)
\ar@{=}"2,4"^(0.52){}\\
Q_1\oplus Q_2\ar@{>}"2,4"^(0.45){\left(
                                         \begin{array}{cc}
                                           H^0(\eta_{Q_1}) & 0 \\
                                           0 & H^0(\eta_{Q_2}) \\
                                         \end{array}
                                       \right)
}
&
&
&\Psi(Q_1)\bigoplus \Psi(Q_2).
}
$$
Since $H^0(\eta_{B^{(\alpha)}})$ is a monomorphism, so
$\left(
\begin{array}{cc}
H^0(\eta_{Q_1}) & 0 \\
 0 & H^0(\eta_{Q_2}) \\
 \end{array}
\right)$ is also a monomorphism. This implies that $H^0(\eta_{Q_1})$ and $H^0(\eta_{Q_2})$ are monomorphisms.

For each $Q\in\Pmodcat {B^{op}}$, the injection $H^0(\eta_{Q}):Q\stackrel{}{\ra}\Psi(Q)$ induces an exact sequence in $B\Modcat$
$$0\lra Q\xrightarrow{H^0(\eta_Q)}\Psi(Q)\lra C(Q)\lra 0.$$
Thus, there is a triangle $Q\stackrel{}{\ra}\Psi(Q)\ra C(Q)\ra Q[1]$ in $\D{B}$.
By Lemma \ref{Rec equivalence}, there is in $\D{B}$ another triangle $\textbf{jK}(Q)\ra Q\ra {\bf HR}(Q)\ra {\bf jK}(Q)[1]$
with $\textbf{jK}(Q)\in\mathscr{Y}_B$.

By part~(1), ${\bf HR}(Q)$ lies in the heart of the standard
$t$-structure of $\D{B}$.  Hence there is a canonical isomorphism
$\epsilon_Q:{\bf HR}(Q)\xrightarrow{\simeq}\Psi(Q)[0]$ with
$H^0(\epsilon_Q)=1_{\Psi(Q)}$.  Since $H^0$ is fully faithful on the heart,
the square
\[
\xymatrix@C=4.8em{
Q[0]\ar[r]^-{\eta_Q}\ar@{=}[d]
  &{\bf HR}(Q)\ar[d]^-{\epsilon_Q}\\
Q[0]\ar[r]_-{\eta_Q^0}
  &\Psi(Q)[0]
}
\]
commutes.  Axiom~\textup{(TR3)} now gives a morphism between the two
triangles, and the two-out-of-three property gives
$C(Q)[0]\simeq({\bf j}{\bf K})(Q)[1]$.
Therefore, $C(Q)\in\mathscr{E}$.
Finally, for a morphism $f:Q\to Q'$ of projective modules, naturality gives
\[
 H^0({\bf HR}(f))\eta_Q^0=\eta_{Q'}^0f.
\]
Thus $H^0({\bf HR}(f))$ induces a unique map
$C(f):C(Q)\to C(Q')$.  Identities, compositions, and sums are preserved
because the unit and ${\bf HR}(f)$ are natural and additive.  We therefore
obtain additive functors $\Psi,C:\Pmodcat B\to B\Modcat$, with $C$ taking
values in $\mathscr{E}$, and the sequences
\[
 0\longrightarrow Q\lraf{\eta_Q^0}\Psi(Q)
   \longrightarrow C(Q)\longrightarrow0
\]
form the asserted natural exact sequence.
\overpr

\begin{Rem}\rm
(1) For a morphism $f:Q\to Q'$ of projective modules, naturality is
expressed by the commutative square
\[
\xymatrix@C=4.2em{
Q\ar[r]^-{\eta_Q^0}\ar[d]_-f
  &\Psi(Q)\ar[d]^-{\Psi(f)}\\
Q'\ar[r]_-{\eta_{Q'}^0}
  &\Psi(Q').
}
\]
It follows that $C(f)$ is the induced morphism on cokernels and that
$C:\Pmodcat B\to B\Modcat$ is additive.

For a complex
$Q^\bullet$ of projective modules, write
\[
 \Psi(Q^\bullet)^i=\Psi(Q^i),\qquad
 C(Q^\bullet)^i=C(Q^i),
\]
with differentials induced by functoriality.  Thus $\Psi$ and $C$ denote termwise functors; they must not be confused with the
ordinary cohomology module $H^0({\bf HR}(Q^\bullet))$.
Accordingly, $\Psi$ and $C$ extend termwise to complexes of projective modules.

(2) By Lemma \ref{lem:projective-unit-final}(1),
$\Psi(Q)[0]=H^0{\bf HR}(Q)\simeq{\bf HR}(Q[0])$ in $\D{B}$ for every $Q\in\Pmodcat B$.

\end{Rem}

To extend the isomorphism $H^0{\bf HR}(-)\simeq{\bf HR}(-[0])$ on projective modules  to homotopically projective complexes, we must appeal to the derived $\infty$-category of dg algebras.

Let $\mathbb{A}$ be a dg algebra.  Denote by
$\mathsf{hProj}_{\dg}(\mathbb{A})$ the dg category of homotopically projective dg
$\mathbb{A}$-modules.  For $P,Q\in\mathsf{hProj}_{\dg}(\mathbb{A})$, its morphism
complex is the internal Hom complex
$ \underline{\Hom}_\mathbb{A}(P,Q)$.
Recall that a dg $\mathbb{A}$-module $P$ is homotopically projective if $\underline{\Hom}_\mathbb{A}(P,N)$
is acyclic for every acyclic dg $\mathbb{A}$-module $N$.

\begin{Def}\rm\cite{LurieHA}.\label{Dg nerve}
The \emph{derived $\infty$-category} of $\mathbb{A}$ is defined by
\[
 \Di_\infty(\mathbb{A})
 :=
 \mathrm N_{\dg}\bigl(\mathsf{hProj}_{\dg}(\mathbb{A})\bigr),
\]
where \(\mathrm N_{\dg}\) denotes the dg nerve.
\end{Def}

The objects of \(\Di_\infty(\mathbb{A})\) are the  homotopically projective dg
$\mathbb{A}$-modules.  For \(P,Q\in\mathsf{hProj}_{\dg}(\mathbb{A})\), the mapping space
is naturally equivalent to
\[
 \operatorname{Map}_{\Di_\infty(\mathbb{A})}(P,Q)
 \simeq
 \operatorname{DK}\!\left(
   \tau_{\leq 0}\underline{\Hom}_\mathbb{A}(P,Q)
 \right),
\]
where the non-positively graded cochain complex on the right is first
reindexed as a non-negatively graded chain complex, and
\(\operatorname{DK}\) denotes the Dold--Kan construction.  Composition
is induced by composition in the dg category
\(\mathsf{hProj}_{\dg}(\mathbb{A})\).
Since \(\mathsf{hProj}_{\dg}(\mathbb{A})\) is strongly pretriangulated, its dg
nerve is a stable \(\infty\)-category.  Moreover, its homotopy category
is naturally equivalent to the derived category of $\mathbb{A}$:
\[
 \mathrm h\Di_\infty(\mathbb{A})
 \simeq
 H^0\bigl(\mathsf{hProj}_{\dg}(\mathbb{A})\bigr)
 \simeq
 \D{\mathbb{A}}.
\]
Thus \(\Di_\infty(\mathbb{A})\) is a stable \(\infty\)-categorical enhancement of
\(\D{\mathbb{A}}\).  Here ``derived \(\infty\)-category'' should not be confused
with the stable derived category occurring in singularity theory.

\begin{Lem}\label{lem:bounded-assembly-final}
For every bounded-above complex
$Q^\bullet\in\Kf{\Pmodcat B}$, there is a natural isomorphism
\[
 \theta_{\cpx{Q}}:\Psi(Q^\bullet)\lraf{\simeq}{\bf HR}(Q^\bullet)
\]
in $\D{B}$ such that the triangle
\[
\xymatrix@C=4.5em{
Q^\bullet\ar[r]^-{u_{\cpx{Q}}}\ar[dr]_-{\eta_{\cpx{Q}}}
  &\Psi(Q^\bullet)\ar[d]^-{\theta_{\cpx{Q}}}\\
  &{\bf HR}(Q^\bullet)
}
\]
commutes, where
$u_{\cpx{Q}}=(\eta_{Q^i}^0)_{i\in\mathbb Z}$ is the termwise
projective-unit map and $\eta_{\cpx{Q}}$ is the derived adjunction unit.
\end{Lem}

{\it Proof}.
We first construct the comparison on projective stalk complexes.
We then use the universal property of the pretriangulated hull to
extend it, together with its compatibility with the adjunction unit,
to bounded complexes. Finally, we pass to bounded-above complexes by
functorial mapping telescopes and a way-out estimate.

\textbf{Step 1: projective stalks.}
The complex of $A$-$B$-bimodule $\cpx{T}:=\cpx{P}\otimesL_{\Delta}B
 \in\D{A\otimes_{\mathbb Z}B^{op}}$
induces an adjunction of stable $\infty$-categories
\[
\xymatrix{
 \cpx{T}\otimesL_B-:
 \Di_\infty(B)\ar@<.5ex>[r]
 &
 \Di_\infty(A):
 \rHom_A(\cpx{T},-)\ar@<.5ex>[l]
}
\]
with its canonical unit. Passing to homotopy categories gives the
adjoint pair
\[
\xymatrix{
 {\bf R}:\D{B}\ar@<.5ex>[r]
 &
 \D{A}:{\bf H}\ar@<.5ex>[l]
}
\]
used above, together with the adjunction unit $\eta:\Id_{\D{B}}\longrightarrow{\bf HR}$.
This is the usual enhanced tensor--Hom adjunction for dg modules (see \cite[\S6]{keller}).

We use the cohomological convention for the standard $t$-structure
on $\Di_\infty(B)$:
\[
 \Di_\infty(B)^{\leq0}
 =
 \left\{
 X\in\Di_\infty(B)
 \ \middle|\
 H^i(X)=0\text{ for every }i>0
 \right\}
\]
and
\[
 \Di_\infty(B)^{\geq0}
 =
 \left\{
 X\in\Di_\infty(B)
 \ \middle|\
 H^i(X)=0\text{ for every }i<0
 \right\}.
\]
Its heart is $\Di_\infty(B)^\heartsuit
 =
 \Di_\infty(B)^{\leq0}
 \cap
 \Di_\infty(B)^{\geq0}$ by \cite[Definition~1.3.5.16 and
Proposition~1.3.5.21]{LurieHA}.
The cohomology functor induces a canonical equivalence
\[
 H^0:\Di_\infty(B)^\heartsuit
 \lraf{\simeq}
 B\Modcat,
\]
whose quasi-inverse sends a $B$-module $M$ to the image of the stalk
complex $M[0]$ in the derived $\infty$-category, equivalently, one may
use any homotopically projective replacement of $M[0]$ (See \cite[Remark~1.3.5.19]{LurieHA}).

Let $Q$ be a projective $B$-module and put $X_Q:={\bf HR}(Q[0])$.
By Lemma~\ref{lem:projective-unit-final}(1), $H^r(X_Q)=0\quad(r\neq0), H^0(X_Q)=\Psi(Q)$.
Thus $ X_Q\in\Di_\infty(B)^{\leq0}\cap\Di_\infty(B)^{\geq0}
       =\Di_\infty(B)^\heartsuit.$
The canonical equivalence between this heart and $B\Modcat$ therefore
gives a natural equivalence
\begin{equation}\label{eq:stalk-enhanced-comparison}
 \epsilon_Q:
 {\bf HR}(Q[0])
 \lraf{\simeq}
 \Psi(Q)[0].
\end{equation}

For completeness, this equivalence can be described directly using
canonical truncations. For every $X\in\Di_\infty(B)$ there is a
natural diagram
\[
 X\xleftarrow{\ \rho_X\ }\tau_{\leq0}X
  \xrightarrow{\ \lambda_X\ }
  \tau_{\geq0}\tau_{\leq0}X.
\]
If $H^r(X)=0$ for $r\neq0$, then both arrows are equivalences and $\tau_{\geq0}\tau_{\leq0}X\simeq H^0(X)[0]$.
Consequently, $\epsilon_Q=\lambda_{X_Q}\circ\rho_{X_Q}^{-1}$.

The truncation maps are natural, and the inverse of a natural
equivalence is determined through a contractible space of choices.
It follows that $\epsilon_Q$ is natural in $Q$, explicitly, for every homomorphism
$f:Q\to Q'$ of projective modules, the square
\[
\xymatrix@C=4.8em{
{\bf HR}(Q[0])\ar[r]^-{\epsilon_Q}\ar[d]_-{{\bf HR}(f)}
  &\Psi(Q)[0]\ar[d]^-{\Psi(f)[0]}\\
{\bf HR}(Q'[0])\ar[r]_-{\epsilon_{Q'}}
  &\Psi(Q')[0]
}
\]
commutes in $\Di_\infty(B)$.

We next check compatibility with the adjunction unit. Both $Q[0]$
and $\Psi(Q)[0]$ belong to the heart, and $H^0$ is fully faithful on
the heart. Moreover, $H^0(\epsilon_Q)=1_{\Psi(Q)}$.
Therefore $H^0(\epsilon_Q\circ\eta_Q)
 =
 H^0(\epsilon_Q)\circ H^0(\eta_Q)
 =
 \eta_Q^0$.
By the full faithfulness of $H^0$ on the heart, we obtain
\begin{equation}\label{eq:stalk-enhanced-unit}
 \epsilon_Q\circ\eta_Q
 =
 \eta_Q^0:
 Q[0]\longrightarrow\Psi(Q)[0].
\end{equation}

\textbf{Step 2: bounded complexes.}
Working in a fixed larger universe, let $\mathcal P=\Pmodcat{B}$
be regarded as a dg category concentrated in degree zero, and put $\Kbb{b}{\mathcal P}
 :=
 \mathrm N_{\dg}(\pretr(\mathcal P))$,
where $\pretr(\mathcal P)$ is the dg category of finite
twisted complexes over $\mathcal P$. Its homotopy category is
naturally equivalent to $\mathrm h\Kbb{b}{\mathcal P}
 \simeq
 \Kb{\Pmodcat{B}}$.
Since every bounded complex of projective $B$-modules is
homotopically projective, totalisation gives an exact functor $\iota:
 \Kbb{b}{\mathcal P}
 \longrightarrow
 \Di_\infty(B)$.

Consider the additive enhanced functors
\[
\begin{aligned}
 i_0:\mathcal P&\longrightarrow\Di_\infty(B),
 &Q&\longmapsto Q[0],\\
 j_0:\mathcal P&\longrightarrow\Di_\infty(B),
 &Q&\longmapsto\Psi(Q)[0],\\
 t_0:\mathcal P&\longrightarrow\Di_\infty(B),
 &Q&\longmapsto{\bf HR}(Q[0]).
\end{aligned}
\]
The functor $\iota$ is the exact extension of $i_0$, and Step~1 gives
a natural equivalence $\epsilon:t_0\lraf{\simeq}j_0$.

We use the universal property of the pretriangulated hull, or
equivalently of the finite stable envelope (see
\cite{KellerDGCategories} and
\cite[\S1.1.3]{LurieHA}). For every stable $\infty$-category
$\mathcal C$, restriction along $\mathcal P\lra\Kbb{b}{\mathcal P}$
induces an equivalence of $\infty$-categories
\begin{equation}\label{eq:pretr-universal-property}
 \Fun^{\ex}(\Kbb{b}{\mathcal P},\mathcal C)
 \lraf{\simeq}
 \Fun^{\add}(\mathcal P,\mathcal C).
\end{equation}
In particular, this equivalence controls natural transformations and
all their coherences.

Termwise application of $\Psi$ induces a dg functor on finite twisted
complexes. Composing it with passage to the derived
$\infty$-category gives an exact functor
\[
 J_\infty:
 \Kbb{b}{\mathcal P}
 \longrightarrow
 \Di_\infty(B)
\]
whose restriction to $\mathcal P$ is $j_0$. Thus, for every bounded
complex $\cpx{M}=(M^i,d_M^i)_{i\in\mathbb Z}$,
the object $J_\infty(\cpx{M})$ is the image in $\Di_\infty(B)$ of
the termwise complex
\[
 \Psi(\cpx{M})
 =
 \bigl(\Psi(M^i),\Psi(d_M^i)\bigr)_{i\in\mathbb Z}.
\]
In the model $\mathsf{hProj}_{\dg}(B)$, this object is represented by
any homotopically projective replacement of the displayed complex.

On the other hand, ${\bf HR}\iota$ is an exact extension of $t_0$.
By \eqref{eq:pretr-universal-property}, the natural equivalence $\epsilon^{-1}:j_0\lraf{\simeq}t_0$
extends, through a contractible space of choices, to a natural
transformation
\[
 \theta:
 J_\infty
 \longrightarrow
 {\bf HR}\iota.
\]
Since the restriction functor in
\eqref{eq:pretr-universal-property} is an equivalence of
$\infty$-categories, it reflects equivalences. Its restriction sends
$\theta$ to the natural equivalence $\epsilon^{-1}$; consequently,
$\theta$ is itself a natural equivalence.

Equivalently, the full subcategory of $\Kbb{b}{\mathcal P}$ consisting
of the objects $\cpx{M}$ for which $\theta_{\cpx{M}}$ is an
equivalence contains $\mathcal P$ and is closed under finite direct
sums, shifts, and cofibers. Since $\Kbb{b}{\mathcal P}$ is generated
from $\mathcal P$ by these operations, this subcategory is all of
$\Kbb{b}{\mathcal P}$.

Passing to homotopy categories therefore gives, for every bounded
complex $\cpx{M}\in\Kb{\Pmodcat{B}}$, a natural isomorphism
\begin{equation}\label{eq:bounded-enhanced-comparison}
 \theta_{\cpx{M}}:
 \Psi(\cpx{M})
 \lraf{\simeq}
 {\bf HR}(\cpx{M})
\end{equation}
in $\D{B}$.

It remains to record compatibility with the unit. The maps $\eta_Q^0:Q[0]\longrightarrow\Psi(Q)[0]$
form a natural transformation $i_0\to j_0$. Its exact extension is
represented on a bounded complex $\cpx{M}$ by the termwise chain map
\[
 u_{\cpx{M}}
 =
 (\eta_{M^i}^0)_{i\in\mathbb Z}:
 \cpx{M}\longrightarrow\Psi(\cpx{M}).
\]
The enhanced adjunction unit gives a natural transformation $\eta:
 \iota\longrightarrow{\bf HR}\iota$.
On every projective stalk,
\eqref{eq:stalk-enhanced-unit} gives $\epsilon_Q^{-1}\circ\eta_Q^0=\eta_Q$.
Since restriction in
\eqref{eq:pretr-universal-property} is fully faithful on mapping
spaces of natural transformations, this identity extends coherently
to all finite twisted complexes. Hence
\begin{equation}\label{eq:bounded-enhanced-unit}
 \theta_{\cpx{M}}\circ u_{\cpx{M}}
 =
 \eta_{\cpx{M}}
 \qquad
 \text{in }\D{B}.
\end{equation}

\textbf{Step 3: bounded-above complexes.}
Let the perfect right $B$-model $\cpx{V}$ of $\cpx{T}$ be supported
in $[a,b]$, and let the bounded projective left $A$-model $\cpx{P}$
be supported in $[u,v]$. Then
\[
 {\bf R}\bigl(\Dle{s}{B}\bigr)
 \subseteq
 \Dle{s+b}{A},
 \qquad
 {\bf H}\bigl(\Dle{t}{A}\bigr)
 \subseteq
 \Dle{t-u}{B}.
\]
The first inclusion follows by totalising
$V^\bullet\otimes_B-$, and the second follows by totalising
$\Hom_A^\bullet(\cpx{P},-)$. Thus, with $c_+=b-u$,
we have the upper way-out estimate
\begin{equation}\label{eq:upper-way-out-detail}
 {\bf HR}\bigl(\Dle{s}{B}\bigr)
 \subseteq
 \Dle{s+c_+}{B}
 \qquad(s\in\mathbb Z).
\end{equation}
On the other hand, termwise application of $\Psi$ does not change the
degrees in which a complex is supported:
\begin{equation}\label{eq:J-upper-bound-detail}
 \Psi\bigl(\Kle{s}{\Pmodcat{B}}\bigr)
 \subseteq
 \Kle{s}{B\Modcat}.
\end{equation}

Let $\cpx{Q}$ be a bounded-above complex of projective $B$-modules,
say $Q^i=0$ for $i>N$ for some integer $N$.  For $m\le N$ with $m\in\mathbb Z$,  let
$\cpx{Q_{\geq m}}$ and $\cpx{Q_{<m}}=\cpx{Q_{\leq m-1}}$
denote the brutal truncations. The first complex is
bounded, the second is supported in degrees at most $m-1$, let
\[
 Q_{\geq m}^\bullet
 \qquad\text{and}\qquad
 Q_{<m}^\bullet=Q_{\leq m-1}^\bullet
\]
denote the brutal truncations. Since $Q^\bullet$ is bounded above,
$Q_{\geq m}^\bullet$ is bounded. Moreover, the degreewise split
short exact sequence
\[
 0\longrightarrow Q_{\geq m}^\bullet
 \longrightarrow Q^\bullet
 \longrightarrow Q_{<m}^\bullet
 \longrightarrow0
\]
induces a cofiber sequence in $\Di_\infty(B)$.

We now perform the telescope construction in the stable
$\infty$-category $\Di_\infty(B)$. Only after completing the
construction do we pass to its homotopy category $\mathrm h\Di_\infty(B)\simeq\D{B}$.
This ensures that all homotopy colimits and all morphisms induced from
them are functorial.

As $m$ tends to $-\infty$, the canonical inclusions
$Q_{\geq m}^\bullet
 \longrightarrow
 Q_{\geq m-1}^\bullet$
form a functorial direct system. Since filtered colimits in
$B\Modcat$ are exact and are computed degreewise, the homotopy
colimit of this degreewise eventually constant system is represented
by its ordinary degreewise colimit. Consequently, the canonical
morphism
\begin{equation}\label{eq:Q-lower-hocolim-detail}
 \beta_{\cpx{Q}}:
 \hocolim_{m\to-\infty}
 Q_{\geq m}^\bullet
 \lraf{\simeq}
 Q^\bullet
\end{equation}
is a natural equivalence in $\Di_\infty(B)$.

After applying $\Psi$ termwise, the resulting system is again
degreewise eventually constant. Its ordinary degreewise colimit is
the complex $\Psi(Q^\bullet)$. Since filtered colimits in
$B\Modcat$ are exact, the canonical morphism
\begin{equation}\label{eq:J-lower-hocolim-detail}
 \alpha_{\cpx{Q}}:
 \hocolim_{m\to-\infty}
 \Psi(Q_{\geq m}^\bullet)
 \lraf{\simeq}
 \Psi(Q^\bullet)
\end{equation}
is a natural equivalence. This does not require $\Psi$ to preserve
arbitrary filtered colimits.

As $Q^\bullet$ varies, the morphisms $\alpha_{\cpx{Q}}$ form a natural
equivalence of functors. We take its inverse in the corresponding
functor $\infty$-category. The inverse $\alpha^{-1}$ is determined
through a contractible space of choices and is therefore coherent
and natural in $Q^\bullet$.

For every $m$, Step~2 gives a natural equivalence $\theta_{\cpx{Q}_{\geq m}}:
 \Psi(Q_{\geq m}^\bullet)
 \lraf{\simeq}
 {\bf HR}(Q_{\geq m}^\bullet)$.
Naturality with respect to the inclusions
$Q_{\geq m}^\bullet\to Q_{\geq m-1}^\bullet$ makes these equivalences
an equivalence of direct systems. Taking homotopy colimits gives a
natural equivalence
\begin{equation}\label{eq:truncated-comparison-hocolim}
 \operatorname*{hocolim}_{m\to-\infty}
 \Psi(Q_{\geq m}^\bullet)
 \lraf{\simeq}
 \operatorname*{hocolim}_{m\to-\infty}
 {\bf HR}(Q_{\geq m}^\bullet).
\end{equation}
The inclusions $Q_{\geq m}^\bullet\to Q^\bullet$ induce a compatible
family of morphisms ${\bf HR}(Q_{\geq m}^\bullet)
 \lra
 {\bf HR}(Q^\bullet)$.
By the universal property of the homotopy colimit, they determine a
natural morphism
\begin{equation}\label{eq:bounded-above-enhanced-map}
 \gamma_{\cpx{Q}}:
 \operatorname*{hocolim}_{m\to-\infty}
 {\bf HR}(Q_{\geq m}^\bullet)
 \longrightarrow
 {\bf HR}(Q^\bullet).
\end{equation}
We do not claim that ${\bf HR}$ preserves this homotopy colimit.

Combining the preceding morphisms, we obtain the natural composite
\begin{equation}\label{eq:bounded-above-comparison-zigzag}
\Psi(Q^\bullet)
\lraf{\simeq}{\alpha_{\cpx{Q}}^{-1}}
\hocolim_{m\to-\infty}
\Psi(Q_{\geq m}^\bullet)
\lraf{\simeq}{
 \hocolim_{m\to-\infty}
 (\theta_{{\cpx{Q}}_{\geq m}})
}
\hocolim_{m\to-\infty}
{\bf HR}(Q_{\geq m}^\bullet)
\lraf{\gamma_{\cpx{Q}}}
{\bf HR}(Q^\bullet).
\end{equation}
We define $\theta_{\cpx{Q}}$ to be the composite in
\eqref{eq:bounded-above-comparison-zigzag}; equivalently,
\begin{equation}\label{eq:bounded-above-theta-definition}
 \theta_{\cpx{Q}}
 =
 \gamma_{\cpx{Q}}
 \circ
 \hocolim_{m\to-\infty}(\theta_{{\cpx{Q}}_{\geq m}}) \circ\alpha_{\cpx{Q}}^{-1}.
\end{equation}
Every construction in
\eqref{eq:bounded-above-comparison-zigzag} is functorial in
$Q^\bullet$. Hence $\theta_{\cpx{Q}}$ is natural in bounded-above complexes.

We now prove that $\theta_{\cpx{Q}}$ is an isomorphism. Fix
$r\in\mathbb Z$. Choose $m$ sufficiently small that
\[
 m-1+c_+<r-1
 \qquad\text{and}\qquad
 m-1<r-1.
\]
Since $Q_{<m}^\bullet\in\Dle{m-1}{B}$,
equations \eqref{eq:upper-way-out-detail} and
\eqref{eq:J-upper-bound-detail} imply
\[
\begin{split}
 H^{r-1}\bigl(\Psi(Q_{<m}^\bullet)\bigr)
 =
 H^r\bigl(\Psi(Q_{<m}^\bullet)\bigr)
 &=0,\\
 H^{r-1}\bigl({\bf HR}(Q_{<m}^\bullet)\bigr)
 =
 H^r\bigl({\bf HR}(Q_{<m}^\bullet)\bigr)
 &=0.
\end{split}
\]

Naturality of the construction gives a morphism of cofiber sequences
in $\Di_\infty(B)$ and hence a morphism of distinguished triangles in
$\D{B}$:
\[
\xymatrix@C=3.8em{
\Psi(\cpx{Q_{\geq m}})\ar[r]\ar[d]_-{\theta_{\cpx{Q_{\geq m}}}}
  &\Psi(\cpx{Q})\ar[r]\ar[d]_-{\theta_{\cpx{Q}}}
  &\Psi(\cpx{Q_{<m}})\ar[r]\ar[d]_-{\theta_{\cpx{Q_{<m}}}}
  &\Psi(\cpx{Q_{\geq m}})[1]
       \ar[d]_-{\theta_{\cpx{Q_{\geq m}}}[1]}\\
{\bf HR}(\cpx{Q_{\geq m}})\ar[r]
  &{\bf HR}(\cpx{Q})\ar[r]
  &{\bf HR}(\cpx{Q_{<m}})\ar[r]
  &{\bf HR}(\cpx{Q_{\geq m}})[1].
}
\]
The
The preceding vanishings imply that the horizontal morphisms induce
isomorphisms
\[
 H^r\bigl(\Psi(Q_{\geq m}^\bullet)\bigr)
 \xrightarrow{\simeq}
 H^r\bigl(\Psi(Q^\bullet)\bigr)\qquad
\text{and}\qquad
 H^r\bigl({\bf HR}(Q_{\geq m}^\bullet)\bigr)
 \xrightarrow{\simeq}
 H^r\bigl({\bf HR}(Q^\bullet)\bigr).
\]
Under these identifications,
$H^r(\theta_{\cpx{Q}})$ agrees with
$H^r(\theta_{{\cpx{Q}}_{\geq m}})$. The latter is an isomorphism by Step~2,
because $Q_{\geq m}^\bullet$ is bounded. Therefore
\[
 H^r(\theta_{\cpx{Q}}):
 H^r\bigl(\Psi(Q^\bullet)\bigr)
 \lraf{\simeq}
 H^r\bigl({\bf HR}(Q^\bullet)\bigr).
\]
Since this holds for every $r\in\mathbb Z$, $\theta_{\cpx{Q}}$ is an
isomorphism in $\D{B}$.

Finally, we prove compatibility with the adjunction unit. For every
$m$, equation \eqref{eq:bounded-enhanced-unit} gives $\theta_{{\cpx{Q}}_{\geq m}}\circ u_{{\cpx{Q}}_{\geq m}}
 =
 \eta_{{\cpx{Q}}_{\geq m}}$.
These identities are compatible with the transition maps of the
truncation systems. Since the maps $u_{{\cpx{Q}}_{\geq m}}$ are obtained by
termwise restriction of $u_{\cpx{Q}}$, we have
\begin{equation}\label{eq:hocolim-termwise-unit}
 \alpha_{\cpx{Q}}\circ\hocolim_{m\to-\infty}(u_{Q_{\geq m}^{\bullet}})
 =u_{\cpx{Q}}\circ\beta_{\cpx{Q}}.
\end{equation}
Moreover, naturality of the enhanced adjunction unit gives, for each
canonical inclusion
$\iota_m:Q_{\geq m}^\bullet\to Q^\bullet$,
\[
 {\bf HR}(\iota_m)\circ\eta_{\cpx{Q}_{\geq m}}
 =
 \eta_{\cpx{Q}}\circ\iota_m.
\]
Taking homotopy colimits, and using the defining property of
$\gamma_{\cpx{Q}}$, yields
\[
 \gamma_{\cpx{Q}}
 \circ
 \hocolim_{m\to-\infty}
       (\theta_{\cpx{Q}_{\geq m}})
 \circ
 \hocolim_{m\to-\infty}
       (u_{\cpx{Q}_{\geq m}})
 =
 \eta_{\cpx{Q}}\circ\beta_{\cpx{Q}}.
\]
Combining this equality with
\eqref{eq:bounded-above-theta-definition} and
\eqref{eq:hocolim-termwise-unit}, we obtain
$\theta_{\cpx{Q}}\circ u_{\cpx{Q}}\circ\beta_{\cpx{Q}}
 =
 \eta_{\cpx{Q}}\circ\beta_{\cpx{Q}}$.

Since $\beta_{\cpx{Q}}$ is an equivalence, it follows that
$\theta_{\cpx{Q}}\circ u_{\cpx{Q}}=\eta_{\cpx{Q}}$.
Thus $\theta_{\cpx{Q}}$ is a natural isomorphism satisfying the required
compatibility with the derived adjunction unit.
\overpr

\medskip

We now extend the comparison of
Lemma~\ref{lem:bounded-assembly-final} to arbitrary complexes of
projective modules. The passage to unbounded complexes uses
homotopy inverse limits and the fact that ${\bf HR}$ preserves
products.

\begin{Lem}\label{lem:unbounded-assembly-final}
Let $\cpx{Q}\in\K{\Pmodcat B}$. There is a natural isomorphism
\[
 \alpha_{\cpx{Q}}:
 \Psi(\cpx{Q})\xrightarrow{\simeq}{\bf HR}(\cpx{Q})
\]
in $\D{B}$ such that
$\alpha_{\cpx{Q}}\circ u_{\cpx{Q}}=\eta_{\cpx{Q}}$, where
$u_{\cpx{Q}}=(\eta_{Q^i}^{0})_{i\in\mathbb Z}$ is the termwise
projective-unit map. Consequently, there is a natural
distinguished triangle
\begin{equation}\label{eq:assembly-triangle-final}
 C(\cpx{Q})[-1]
 \xrightarrow{\partial_{\cpx{Q}}}\cpx{Q}
 \xrightarrow{\eta_{\cpx{Q}}}{\bf HR}(\cpx{Q})
 \longrightarrow C(\cpx{Q}),
\end{equation}
and every term of $C(\cpx{Q})$ belongs to $\mathscr E$.
\end{Lem}

{\it Proof}.
We retain the enhanced tensor--Hom adjunction and the coherent
comparison constructed in
Lemma~\ref{lem:bounded-assembly-final}.
All homotopy inverse limits and comparisons between towers are
formed in the corresponding stable $\infty$-categories before
passing to $\D{B}$. Complexes are understood via their canonical
images in the derived $\infty$-category.

The recollement of Lemma~\ref{Rec equivalence} shows that
${\bf R}$ has a left adjoint. Since ${\bf H}$ is right adjoint
to ${\bf R}$, both functors preserve products. Hence their
composite ${\bf HR}$ preserves products as well.
A countable homotopy inverse limit is the fiber of the map
$1-s:\prod_n X_n\to\prod_n X_n$ induced by the transition
maps. Since ${\bf HR}$ is exact and preserves products, the
canonical comparison
\[
 {\bf HR}(\holim_n X_n)
 \lra
 \holim_n{\bf HR}(X_n)
\]
is an equivalence, naturally in the tower $(X_n)_n$.

For $n\geq0$, let $\cpx{Q_n}=\cpx{Q}_{\leq n}$ be the
upper brutal truncation, with canonical projection
$\pi_n:\cpx{Q}\to\cpx{Q_n}$.
The projections $p_n:\cpx{Q_{n+1}}\to\cpx{Q_n}$ form an
inverse tower and satisfy $\pi_n=p_n\pi_{n+1}$.
Each $\cpx{Q_n}$ is a bounded-above complex of projective
modules.

The truncation tower gives a short exact sequence of complexes
\[
 0\lra\cpx{Q}
 \lraf{(\pi_n)_n}
 \prod_{n\geq0}\cpx{Q_n}
 \lraf{1-s}
 \prod_{n\geq0}\cpx{Q_n}
 \lra 0,
\]
where $s((x_n)_n)=(p_n(x_{n+1}))_n$.
Indeed, in every degree the tower is eventually constant with
identity transition maps. Its inverse limit is the corresponding
term of $\cpx{Q}$, and surjectivity of $1-s$ follows by solving
the equations recursively on the constant tail.
Since products of $B$-modules are exact, termwise products
represent products in $\D{B}$. The preceding short exact
sequence therefore gives a natural equivalence
\begin{equation}\label{eq:Q-holim-detail}
 \delta_{\cpx{Q}}:
 \cpx{Q}\lraf{\simeq}
 \holim_n\cpx{Q_n}.
\end{equation}

Since $\Psi$ is applied termwise, one has
$\Psi(\cpx{Q_n})=\Psi(\cpx{Q})_{\leq n}$.
Applying the same argument to $\Psi(\cpx{Q})$ gives a natural
equivalence
\begin{equation}\label{eq:J-holim-detail}
 \delta_\Psi:
 \Psi(\cpx{Q})\lraf{\simeq}
 \holim_n\Psi(\cpx{Q_n}).
\end{equation}
This argument does not require $\Psi$ to preserve products.
Moreover, preservation of homotopy inverse limits by ${\bf HR}$
and \eqref{eq:Q-holim-detail} give the natural equivalence
\begin{equation}\label{eq:T-holim-detail}
 \beta_{\cpx{Q}}:
 {\bf HR}(\cpx{Q})\lraf{\simeq}
 \holim_n{\bf HR}(\cpx{Q_n}),
\end{equation}
induced by the maps ${\bf HR}(\pi_n)$.

By Lemma~\ref{lem:bounded-assembly-final}, there are equivalences
$\theta_n:\Psi(\cpx{Q_n})\lraf{\simeq}
{\bf HR}(\cpx{Q_n})$.
The enhanced naturality established there makes
$(\theta_n)_n$ an equivalence of inverse towers, including
compatibility with the adjunction units.
Define
\begin{equation}\label{eq:unbounded-alpha-construction}
 \alpha_{\cpx{Q}}
 :=
 \beta_{\cpx{Q}}^{-1}
 \circ(\holim_n\theta_n)
 \circ\delta_\Psi.
\end{equation}
All three factors are equivalences.

To verify compatibility with the unit, put
$u_n=u_{\cpx{Q_n}}$.
The enhanced comparisons of
Lemma~\ref{lem:bounded-assembly-final} satisfy
$\theta_nu_n=\eta_{\cpx{Q_n}}$ coherently in $n$.
The truncation maps and the definition of
$\alpha_{\cpx{Q}}$ give the commutative diagram
\[
\xymatrix@C=4.5em@R=3.5em{
 \cpx{Q}
 \ar[r]^-{u_{\cpx{Q}}}
 \ar[d]_-{\delta_{\cpx{Q}}}
 &
 \Psi(\cpx{Q})
 \ar[r]^-{\alpha_{\cpx{Q}}}
 \ar[d]_-{\delta_\Psi}
 &
 {\bf HR}(\cpx{Q})
 \ar[d]^-{\beta_{\cpx{Q}}}
 \\
 \holim_n\cpx{Q_n}
 \ar[r]_-{\holim_n u_n}
 &
 \holim_n\Psi(\cpx{Q_n})
 \ar[r]_-{\holim_n\theta_n}
 &
 \holim_n{\bf HR}(\cpx{Q_n}).
}
\]
The bottom composite is
$\holim_n\eta_{\cpx{Q_n}}$.
Naturality of the enhanced adjunction unit gives
\[
 (\holim_n\eta_{\cpx{Q_n}})
 \circ\delta_{\cpx{Q}}
 =
 \beta_{\cpx{Q}}\circ\eta_{\cpx{Q}}.
\]
Consequently,
$\beta_{\cpx{Q}}\alpha_{\cpx{Q}}u_{\cpx{Q}}
=\beta_{\cpx{Q}}\eta_{\cpx{Q}}$.
Since $\beta_{\cpx{Q}}$ is an equivalence, this proves
$\alpha_{\cpx{Q}}u_{\cpx{Q}}=\eta_{\cpx{Q}}$.

The construction is natural for chain maps.
Since termwise application of the additive functor $\Psi$
preserves chain homotopies, and ${\bf HR}$ is defined on
$\D{B}$, this naturality descends to $\K{\Pmodcat B}$.

Finally, applying
\eqref{eq:projective-unit-final} degreewise gives a natural
short exact sequence of complexes
\[
 0\lra\cpx{Q}
 \lraf{u_{\cpx{Q}}}\Psi(\cpx{Q})
 \lra C(\cpx{Q})
 \lra 0.
\]
Its associated distinguished triangle, followed by the
identification $\alpha_{\cpx{Q}}$ and rotation, gives
\eqref{eq:assembly-triangle-final}.
Every term of $C(\cpx{Q})$ belongs to $\mathscr E$ by
Lemma~\ref{lem:projective-unit-final}.
No degreewise splitting of $u_{\cpx{Q}}$ is required.
\overpr

\medskip

In subsequent applications, homotopically projective replacements
are chosen to have projective terms.
We next relate the preceding construction to the derived category
of the exact category $\mathscr E$.
Since $\mathscr E$ is an exact subcategory of $B\Modcat$,
the inclusion $\lambda:\mathscr E\hookrightarrow B\Modcat$
is exact and induces a triangle functor
\[
 {\bf\overline{j}}:=\D{\lambda_*}:
 \D{\mathscr E}\longrightarrow\D{B}.
\]
The following lemma shows that complexes with terms in $\mathscr E$
belong to $\mathscr Y_B$ and that ${\bf\overline{j}}$ reflects
zero objects.

\begin{Lem}\label{lem:strict-acyclicity-final}
Let $\cpx{X}$ be a complex whose terms belong to $\mathscr{E}$, that is, $\cpx{X}\in\K{\mathscr{E}}$.

(1) ${\bf \overline{j}}(\cpx{X})\in \mathscr{Y}_B$ in $\D{B}$.

(2) $\cpx{X}$ is exact in $\C{B}$ if and only if $\cpx{X}\in \Kac{\mathscr{E}}$.
     
(3) If ${\bf \overline{j}}(\cpx{X})=0$, then $\cpx{X}\in \Kac{\mathscr{E}}$. 
\end{Lem}
{\it Proof}.
(1) We first justify the reduction to bounded complexes.
As established above, ${\bf R}$ preserves products and
coproducts. Consequently, $\mathscr{Y}_B=\Ker({\bf R})$
is closed under countable homotopy limits and homotopy
colimits.

Suppose first that $\cpx{X}$ is bounded below.
For $n\geq0$, let $\cpx{X}_{\leq n}$ be its brutal
truncation in degrees at most $n$.
The canonical projections form an inverse system of
bounded complexes, all of whose terms belong to
$\mathscr{E}$.
This system is eventually constant in each degree,
with transition maps equal to the identity.
Hence there is a short exact sequence of complexes
\[
 0\longrightarrow\cpx{X}
 \longrightarrow\prod_{n\geq0}\cpx{X}_{\leq n}
 \xrightarrow{\,1-s\,}
 \prod_{n\geq0}\cpx{X}_{\leq n}
 \longrightarrow0,
\]
where $s$ is induced by the transition maps.
Since products of $B$-modules are exact, this identifies
$\cpx{X}$ with
$\operatorname*{holim}_{n\geq0}\cpx{X}_{\leq n}$
in $\D{B}$.
Thus the assertion for bounded complexes implies the
assertion for bounded-below complexes.

For an arbitrary complex $\cpx{X}$, the brutal truncations
$\cpx{X}_{\geq -n}$ form a direct system under the canonical
inclusions. The standard telescope sequence
\[
 0\lra
 \bigoplus_{n\geq0}\cpx{X}_{\geq -n}
 \lraf{\,1-s\,}
 \bigoplus_{n\geq0}\cpx{X}_{\geq -n}
 \lra\cpx{X}
 \lra 0
\]
is exact. Therefore $\cpx{X}\simeq
 \hocolim_{n\geq0}\cpx{X}_{\geq -n}$ in $\D{B}$.
Each $\cpx{X}_{\geq -n}$ is bounded below and has terms
in $\mathscr{E}$.
It follows that it is enough to prove
$\overline{{\bf j}}(\cpx{X})\in\mathscr{Y}_B$
for bounded complexes $\cpx{X}$ with terms in
$\mathscr{E}$.
Let $$\cpx{X}:=0\lra X^0\lra X^{1}\lra \cdots\lra X^{n-1}\lra X^n\lra 0.$$ 
By brutal truncations, there are a series of triangles
$$\aligned
\cpx{X}_{\geq 1}\lra &\cpx{X} \lra X^0\lra \cpx{X}_{\geq 1}[1],\\
\cpx{X}_{\geq 2}\lra &\cpx{X}_{\geq 1} \lra X^1[-1]\lra \cpx{X}_{\geq 2}[1],\\
&\cdots\hspace{0.6cm}\cdots\\
X^{n}[-n]\lra &\cpx{X}_{\geq n-1}\lra X^{n-1}[-(n-1)]\lra X^{n}[-n+1].
\endaligned$$
Since $X^{0}, X^1,\cdots, X^{n}\in \mathscr{E}=\mathscr{Y}_B\cap B\Modcat=\Ker({\bf R})\cap B\Modcat$, we get that 
$$X^{0}, X^1[-1],\cdots, X^{n-1}[-(n-1)], X^{n}[-n]\in\Ker({\bf R}).$$
Applying ${\bf R}$ to the last triangle yields a triangle 
$${\bf R}(X^{n}[-n])\lra {\bf R}(\cpx{X}_{\geq n-1})\lra {\bf R}(X^{n-1}[-(n-1)])\lra {\bf R}(X^{n}[-n])[1]. $$
This implies that 
$\cpx{X}_{\geq n-1}\in \mathscr{Y}_B$. The same process ultimately yields $\cpx{X}\in \mathscr{Y}_B$. 

(2) Assume that $X^\bullet$ is exact in $B\Modcat$.  Exactness gives, for every
$i$, a short exact sequence
$$
 0\lra Z^i(X^\bullet)\lra X^i
   \lraf{d^i}Z^{i+1}(X^\bullet)\lra 0.\hspace{1cm} (\diamondsuit)
$$
Since $X^i\in\mathscr{E}$, one has ${\bf R}(X^i)=0$.  The triangle obtained from
$(\diamondsuit)$ consequently yields an isomorphism
${\bf R}(Z^i(X^\bullet))\simeq{\bf R}(Z^{i+1}(X^\bullet))[-1]$ in $\D{A}$.
Iteration gives
${\bf R} (Z^i(X^\bullet))\simeq {\bf R}(Z^{i+r}(X^\bullet))[-r]$ for all $r\geq 0$.
 
For every $B$-module $M$, Lemma \ref{lem:amplitude-new} gives
${\bf R}(M)\in\mathscr{D}^{[a,b]}(A)$.
Hence 
${\bf R}(Z^{i+r}(X^\bullet))[-r]\in\mathscr{D}^{[a,b]}(A)$, while 
${\bf R} (Z^i(X^\bullet))\in\mathscr{D}^{[a+r,b+r]}(A)$.  Choose $r>b-a$.  These intervals are disjoint, so
the common object belongs to
$\mathscr{D}^{[a,b]}(A)\cap\mathscr{D}^{[a+r,b+r]}(A)=0$.
Thus ${\bf R}(Z^i(X^\bullet))=0$ for every $i$, or equivalently
$Z^i(X^\bullet)\in\mathscr{E}$.  
This implies that $\cpx{X}\in\K{\mathscr{E}}$ is a 
strictly exact complex. Thus $\cpx{X}\in\Kac{\mathscr{E}}$.

By the definition of strictly complexes, it is obvious that $\cpx{X}$ is exact in $\C{B}$ if $\cpx{X}\in\Kac{\mathscr{E}}$. 

(3) Let $\cpx{X}\in\D{\mathscr{E}}$ with $\overline{{\bf j}}(\cpx{X})=0$ in $\D{B}$, i.e.,
$\overline{{\bf j}}(\cpx{X})$ is an exact sequence in $\C{B}$. Since $\overline{{\bf j}}$ 
is induced by the exact, inclusion functor $\mathscr{E}\subseteq B\Modcat$, $\cpx{X}$ is exact in $\C{B}$.
(2) implies that $\cpx{X}\in\K{\mathscr{E}}$ is a strictly exact complex, i.e., $\cpx{X}\in\Kac{\mathscr{E}}$.
Notice that $\D{\mathscr{E}}=\K{\mathscr{E}}/\Kac{\mathscr{E}}$, so $\cpx{X}=0$ in $\D{\mathscr{E}}$.
\overpr
\medskip

For the following constructions, fix a functorial
cofibrant replacement in the projective model structure
on $\C{B}$,
\[
 {}_p:\C{B}\lra\C{B},
 \qquad
 q:{}_p\Longrightarrow\Id_{\C{B}}.
\]
Thus ${}_p(\cpx{Z})$ is homotopically projective with
projective terms, and
$q_{\cpx{Z}}:{}_p(\cpx{Z})\longrightarrow\cpx{Z}$
is a natural quasi-isomorphism.

Write $\lambda_*:\C{\mathscr{E}}\longrightarrow\C{B}$
for the termwise inclusion.
In diagrams of complexes, the notation
$\overline{{\bf j}}$ will also denote this inclusion.
Likewise, complexes with terms in $\mathscr{E}$ occurring
in diagrams in $\D{B}$ are implicitly viewed through
the inclusion.

For a complex $\cpx{Q}$ of projective $B$-modules, put $u_{\cpx{Q}}=(\eta^0_{Q^i})_{i\in\mathbb Z}, C(\cpx{Q})=\Coker(u_{\cpx{Q}})$,
where the cokernel is taken in $\C{B}$.
The natural comparison in
Lemma~\ref{lem:unbounded-assembly-final} satisfies $\alpha_{\cpx{Q}}\circ u_{\cpx{Q}}=\eta_{\cpx{Q}}$.
Here and below, $\circ$ denotes ordinary composition.

Let $\Phi$ be the composition of the following functors
$$\D{B}
 \xrightarrow[\simeq]{{}_p}
 \bigl(\K{B}_{P}\cap\K{\Pmodcat{B}}\bigr)
 \xrightarrow{\,C=\Coker(u_{-})\,}
 \K{\mathscr{E}}
 \xrightarrow{Q}
 \D{\mathscr{E}}.$$
The functor $C$, applied termwise, preserves chain
homotopies and mapping cones, and therefore induces
the triangle functor used in this composition.

\begin{Lem}\label{lem:Phi-final}
For every $\cpx{X}\in \D{B}$, there exists a triangle
$$ \overline{{\bf j}}\Phi(\cpx{X})[-1]\longrightarrow \cpx{X}\longrightarrow{\bf HR} (\cpx{X})
       \longrightarrow \overline{{\bf j}}\Phi(\cpx{X})$$
in $\D{B}$. In particular, if $\cpx{X}\in\mathscr{Y}_B$, then
$\overline{{\bf j}}\Phi(\cpx{X})[-1]$ is isomorphic to $\cpx{X}$ in $\D{B}$.
\end{Lem}
{\it Proof}. Let $\cpx{Q}$ be the  homotopically projective resolution of $\cpx{X}$.
By the definition of $\Phi$, $\overline{{\bf j}}\Phi(\cpx{X})$ is isomorphic to $C(\eta_{\cpx{Q}})$ in
$\D{B}$.
By Lemma \ref{lem:unbounded-assembly-final}, there is a triangle 
$$
C(\eta_{\cpx{Q}})[-1]\lraf{}Q^\bullet
 \xrightarrow{\eta_{\cpx{Q}}=(\eta_{Q^i})}{\bf HR}(Q^\bullet)\longrightarrow C(\eta_{\cpx{Q}}).
$$
Since $\cpx{Q}\lraf{\simeq} \cpx{X}$ and ${\bf HR}(Q^\bullet)\lraf{\simeq} {\bf HR}(\cpx{X})$
in $\D{B}$, there exists a triangle in $\D{B}$
$$
 \overline{{\bf j}}\Phi(\cpx{X})[-1]\longrightarrow \cpx{X}\longrightarrow{\bf HR} (\cpx{X})
       \longrightarrow \overline{{\bf j}}\Phi(\cpx{X}).
$$
In particular, if $\cpx{X}\in\mathscr{Y}_B$, then ${\bf R}(\cpx{X})=0$. Thus $\overline{{\bf j}}\Phi(\cpx{X})[-1]\simeq 
\cpx{X}$ in $\D{B}$. The triangle furnished by
Lemma~\ref{lem:unbounded-assembly-final} is natural
in $\cpx{Q}$ and compatible with the adjunction unit.
Together with the functorial replacement, this shows
that the resulting triangle is natural in $\cpx{X}$
\overpr

We recall the following standard fact about derived categories
of exact categories; see \cite[Section~11, Lemma~11.6 and the
preceding construction]{Keller1}.
For completeness, we include a direct proof valid for
unbounded complexes.

\begin{Lem}
\label{lem:conflation-triangle-final}
Let $\mathcal{A}$ be an exact category.  A degreewise conflation of complexes
$ 0\ra \cpx{X}\lraf{\cpx{u}}Y^\bullet
   \lraf{\cpx{v}} Z^\bullet\ra0$
induces a natural distinguished triangle
$X^\bullet\lraf{}Y^\bullet\lraf{}Z^\bullet
       \ra X^\bullet[1]$
in $\D{\mathcal{A}}$.
\end{Lem}

{\it Proof}.
By \cite[Section~11, Lemma~11.6]{Keller1}, the canonical chain map defined by
$$
 \con(u)\longrightarrow Z^\bullet,\qquad (y,x)\mapsto(y)v.
$$
is a quasi-isomorphism. Thus the mapping-cone triangle of $u$
induces the required distinguished triangle in $\D{\mathcal A}$,
with connecting morphism
\[
 Z^\bullet\xrightarrow{\pi^{-1}}\con(u)
 \longrightarrow X^\bullet[1].
\]
The mapping-cone construction and $\pi$ are natural with
respect to morphisms of degreewise conflations, so the
resulting triangle has the asserted naturality.
\overpr

\begin{Lem}\label{lem:pushout-counit-final}
Let $\cpx{X}\in\C{\mathscr{E}}$, and $q_{\overline{{\bf j}} (X^\bullet)}:{_p(\overline{{\bf j}} (X^\bullet))}\ra 
\overline{{\bf j}} (X^\bullet)$ be the homotopically projective  resolution of $\overline{{\bf j}} (X^\bullet)$.
Then 

(1) there exists a commutative diagram in $\C{B}$
$$\xymatrix@C1.7cm{0\ar@{>}"1,2"^(0.45){}
&{_p(\overline{{\bf j}}(\cpx{X}))}\ar@{>}"1,3"^(0.48){H^0(\eta_{_p(\overline{{\bf j}}(\cpx{X}))})}\ar@{>}"2,2"_(0.45){q_{\overline{{\bf j}}(\cpx{X})}}
&\Psi(_p(\overline{{\bf j}}(\cpx{X}))\ar@{>}"1,4"^(0.45){}\ar@{>}"2,3"_(0.45){}
&C(p(\overline{{\bf j}}(\cpx{X})))\ar@{>}"1,5"^(0.45){}\ar@{=}"2,4"_(0.45){}
&0\\
0\ar@{>}"2,2"^(0.45){}&
\overline{{\bf j}}(\cpx{X})\ar@{>}"2,3"^(0.45){}
&\cpx{W}\ar@{>}"2,4"^(0.45){}
&C(p(\overline{{\bf j}}(\cpx{X})))\ar@{>}"2,5"^(0.45){}
&0
}$$
Moreover, the lower row  is a conflation in $\C{\mathscr{E}}$.

(2) there is a natural transformation
$\delta:\Phi(\overline{{\bf j}}(-))[-1]\longrightarrow Id_{\D{\mathscr{E}}}$
on $\D{\mathscr{E}}$ such that
the following  diagram is commutative in $\D{B}$ 
$$\xymatrix{
\Phi(\overline{{\bf j}}(\cpx{X}))[-1]\ar@{>}"1,2"^(0.65){\delta_{\cpx{X}}}\ar@{>}"2,1"_(0.53){\partial_{\overline{{\bf j}}(\cpx{X})}}
&\cpx{X}\ar@{=}"2,2"_(0.53){}\\
{_p(\overline{{\bf j}}(\cpx{X}))}\ar@{>}"2,2"_(0.53){q_{\overline{{\bf j}}(\cpx{X})}}&\overline{{\bf j}}(\cpx{X})
}
$$
\end{Lem}
{\it Proof}.
(1) 
Push the natural short exact sequence
\eqref{eq:projective-unit-final}, applied termwise to
${}_p(\overline{{\bf j}}(\cpx{X}))$, out along the
quasi-isomorphism
$q_{\overline{{\bf j}}(\cpx{X})}$. This yields the commutative diagram in (1).
For brevity set
\[
 \cpx{Q}_{X^\bullet}={_p(\overline{{\bf j}}(\cpx{X}))},\qquad
 \cpx{J}_{X^\bullet}=\Psi(_p(\overline{{\bf j}}(\cpx{X})),\qquad
 \cpx{C}_{X^\bullet}=C(p(\overline{{\bf j}}(\cpx{X}))),
\]
and write
$\eta_{\cpx{X}}=u_{\cpx{Q}_{\cpx{X}}}:
\cpx{Q}_{\cpx{X}}\longrightarrow\cpx{J}_{\cpx{X}}$
for the termwise projective-unit monomorphism.
Degreewise, the pushout can be written explicitly as
\begin{equation}\label{eq:PX-explicit-final}
 W^i=
 \bigl(J_{\cpx{X}}^i\oplus X^i\bigr)
 \big/
 \bigl\{
   \bigl((a)\eta_{\cpx{X}}^i,
         -(a)q_{\overline{{\bf j}}(\cpx{X})}^i\bigr)
   \mid a\in Q_{\cpx{X}}^i
 \bigr\}.
\end{equation}
Because $\eta_{\cpx{X}}$ and $q_{\overline{{\bf j}}(\cpx{X})}$ are chain maps, the direct-sum
differential on $\cpx{J}_{X^\bullet}\oplus X^\bullet$ preserves the submodules in
\eqref{eq:PX-explicit-final}; it therefore induces a differential on
$\cpx{W}$.  For any $i\in\IZ$, notice that the maps of the lower row in (1) are 
$$
 \lambda^i_{\cpx{X}}:X^i\longrightarrow W_{\cpx{X}}^i,\quad x\mapsto(0,x),
 \qquad
 \pi^i_{\cpx{X}}:W_{\cpx{X}}^i\longrightarrow C_{\cpx{X}}^i,\quad(j,x)\mapsto(j)\pi_{\cpx{X}}^i,
$$
Clearly, $\pi^i_{\cpx{X}}:W_{\cpx{X}}^i\ra C_{\cpx{X}}^i$ is surjective.
Claim that 
$\lambda^i_{\cpx{X}}$ is injective. Indeed,
if $(0,x)=0$, then
$(0,x)=((a)\eta_{\cpx{X}}^i,-(a)q_{\overline{{\bf j}}(\cpx{X})}^i)$ for some $a$. The injectivity
of $\eta_{\cpx{X}}^i$ implies $a=0$, and $x=0$.

If $\pi_{\cpx{X}}^i(j)=0$, write $j=(a)\eta_{\cpx{X}}^i$, then
$(j,x)=(0,x+(a)q_{\overline{{\bf j}}(\cpx{X})}^i)$,
This implies that $\Img(\lambda^i_{\cpx{X}})=\Ker(\pi^i_{\cpx{X}})$ for all $i\in\IZ$.
Thus 
$$0\lra
\overline{{\bf j}}(\cpx{X})\lra
\cpx{W}\lra
\cpx{C}_{X^\bullet}\lra
0$$ is a short exact sequence of
$B$-complexes.
Every $X^i$ belongs to $\mathscr{E}$ by assumption, and every
$C_{\cpx{X}}^i$ belongs to $\mathscr{E}$ by
Lemma~\ref{lem:projective-unit-final}.  The degreewise exact sequence
$ 0\ra X^i\longrightarrow W_{\cpx{X}}^i\ra C_{\cpx{X}}^i\ra0
$
and the extension closure of $\mathscr{E}$ imply $W_{\cpx{X}}^i\in\mathscr{E}$.  Hence the lower row is a
degreewise conflation in the exact category $\C{\mathscr{E}}$.

(2)
By Lemma \ref{lem:conflation-triangle-final} and (1), there exists a triangle
$\cpx{X}\ra \cpx{W}_{\cpx{X}}\ra \cpx{C}_{\cpx{X}}\xrightarrow{\delta_{\cpx{X}}[1]} \cpx{X}[1]$ in $\D{\mathscr{E}}$.
Let $\cpx{f}:\cpx{X}\ra \cpx{Y}$ be a chain map. By the functoriality of the functors 
$\overline{{\bf j}},{\bf H},{\bf R}$ and $_p(-)$, there exists a commutative diagram

$$\xymatrix@C1.5cm{
0\ar@{>}"1,2"^(0.45){}
&\overline{{\bf j}}(\cpx{X})\ar@{>}"1,3"^(0.45){}\ar@{.>}@/_3.8pc/"4,2"_(0.5){\overline{{\bf j}}(\cpx{f})}
&\cpx{W}_{\cpx{X}}\ar@{>}"1,4"^(0.45){}\ar@{.>}@/^3.8pc/"4,3"^(0.5){\zeta(\cpx{f})}
&\cpx{C}_{\cpx{X}}\ar@{>}"1,5"^(0.45){}
&0\\
0\ar@{>}"2,2"^(0.45){}
&_p(\overline{{\bf j}}(\cpx{X}))\ar@{>}"2,3"^(0.43){}\ar@{>}"1,2"_(0.45){q_{\overline{{\bf j}}(\cpx{X})}}\ar@{.>}"3,2"^(0.45){_p(\overline{{\bf j}}(\cpx{f}))}
&\cpx{J}_{X^\bullet}\ar@{>}"2,4"^(0.45){}\ar@{>}"1,3"_(0.45){}\ar@{.>}"3,3"^(0.45){\Psi_p(\overline{{\bf j}}(\cpx{f}))}
&\cpx{C}_{\cpx{X}}\ar@{>}"2,5"^(0.45){}\ar@{=}"1,4"_(0.45){}
&0\\
0\ar@{>}"3,2"^(0.45){}
&_p(\overline{{\bf j}}(\cpx{Y}))\ar@{>}"3,3"^(0.43){}\ar@{>}"4,2"_(0.45){q_{\overline{{\bf j}}(\cpx{Y})}}
&\cpx{J}_{Y^\bullet}\ar@{>}"3,4"^(0.45){}\ar@{>}"4,3"_(0.45){}
&\cpx{C}_{\cpx{Y}}\ar@{>}"3,5"^(0.45){}\ar@{=}"4,4"_(0.45){}
&0\\
0\ar@{>}"4,2"^(0.45){}&
\overline{{\bf j}}(\cpx{Y})\ar@{>}"4,3"^(0.45){}
&\cpx{W}_{\cpx{Y}}\ar@{>}"4,4"^(0.45){}
&\cpx{C}_{\cpx{Y}}\ar@{>}"4,5"^(0.45){}
&0
}$$
The universal property of the pushout implies that there is a
unique chain map $\zeta(\cpx{f}):\cpx{W}_{\cpx{X}}\to \cpx{W}_{\cpx{Y}}$ such that the whole diagram is commutative in $\C{B}$.
Thus there is a morphism of conflations 
$$
\xymatrix{0\ar@{>}"1,2"^(0.45){}
&\overline{{\bf j}}(\cpx{X})\ar@{>}"1,3"^(0.45){}\ar@{>}"2,2"^(0.45){\overline{{\bf j}}(\cpx{f})}
&\cpx{W}_{\cpx{X}}\ar@{>}"1,4"^(0.48){\delta_{\cpx{X}}[1]}\ar@{>}"2,3"^(0.45){\zeta(\cpx{f})}
&\cpx{C}_{\cpx{X}}\ar@{>}"1,5"^(0.45){}\ar@{>}"2,4"^(0.45){\rho(\cpx{f})}
&0\\
0\ar@{>}"2,2"^(0.45){}
&\overline{{\bf j}}(\cpx{Y})\ar@{>}"2,3"^(0.45){}
&\cpx{W}_{\cpx{Y}}\ar@{>}"2,4"^(0.48){\delta_{\cpx{Y}}[1]}
&\cpx{C}_{\cpx{Y}}\ar@{>}"2,5"^(0.45){}
&0
}
$$
By the naturality in
Lemma~\ref{lem:conflation-triangle-final}, the
preceding morphism of conflations gives
\[
 \cpx{f}\circ\delta_{\cpx{X}}
 =
 \delta_{\cpx{Y}}\circ\rho(\cpx{f})[-1]
 \qquad\text{in }\D{\mathscr{E}}.
\]

Suppose that $\cpx{f}$ is a quasi-isomorphism of
the underlying $B$-module complexes.
Naturality of the replacement maps implies that
${}_p(\overline{{\bf j}}(\cpx{f}))$ is a
quasi-isomorphism between homotopically projective
complexes, and hence a chain homotopy equivalence.
Since the additive functor $C$, applied termwise,
preserves chain homotopies,
$\rho(\cpx{f})$ is a chain homotopy equivalence
in $\K{\mathscr{E}}$.
The preceding diagram of short exact sequences
also shows that $\zeta(\cpx{f})$ is a
quasi-isomorphism of $B$-module complexes.

By Lemma~\ref{lem:strict-acyclicity-final}(2),
a chain map in $\C{\mathscr{E}}$ is a
quasi-isomorphism of underlying $B$-module complexes
if and only if its cone is strictly acyclic
in $\mathscr{E}$.
Thus the two functors
\[
 \cpx{X}\longmapsto\cpx{C}_{\cpx{X}}[-1],
 \qquad
 \cpx{X}\longmapsto\cpx{X},
\]
viewed as functors from $\C{\mathscr{E}}$ to
$\D{\mathscr{E}}$, invert the morphisms defining
the derived category.
The universal property of localisation therefore
descends the preceding naturality identity to
a natural transformation
\[
 \delta:
 \Phi(\overline{{\bf j}}(-))[-1]
 \longrightarrow\operatorname{Id}_{\D{\mathscr{E}}}.
\] Moreover,
there exists a commutative diagram in $\D{B}$ due to (1), 
$$\xymatrix{
\cpx{C}_{\cpx{X}}[-1]\ar@{>}"1,2"^(0.55){\partial_{\cpx{Q}_{\cpx{X}}}}\ar@{=}"2,1"^(0.45){}
&
\cpx{Q}_{\cpx{X}}\ar@{>}"1,3"^(0.45){}\ar@{>}"2,2"^(0.45){q_{\overline{{\bf j}}(\cpx{X})}}
&
\cpx{J}_{\cpx{X}}\ar@{>}"1,4"^(0.45){}\ar@{>}"2,3"^(0.45){}
&
\cpx{C}_{X^\bullet}\ar@{=}"2,4"^(0.45){}\\
\cpx{C}_{X^\bullet}[-1]\ar@{>}"2,2"^(0.55){\overline{{\bf j}}(\delta_{\cpx{X}})}
&
\overline{{\bf j}}(X^\bullet)\ar@{>}"2,3"^(0.45){}
&
\overline{{\bf j}}(\cpx{W}_{X^\bullet})\ar@{>}"2,4"^(0.45){}
&
\cpx{C}_{X^\bullet} .
}
$$
In particular, the left-hand square gives $\overline{{\bf j}}(\delta_{\cpx{X}})
 =q_{\overline{{\bf j}}(\cpx{X})}
 \circ\partial_{\cpx{Q}_{\cpx{X}}}$,
which proves the required compatibility. This finishes the proof.
\overpr

The following proposition establishes that the functor $\overline{{\bf j}}$ is an equivalence, a fact that is crucial for proving our main theorem. 

\begin{Prop}
\label{prop:ordinary-final}
The exact inclusion $\mathscr{E}\hookrightarrow B\Modcat$ induces a triangle
equivalence
$\overline{{\bf j}}:\D{\mathscr{E}}\xrightarrow{\simeq}\mathscr{Y}_B$,
whose quasi-inverse is the restriction of 
$\overline{{\bf K}}:=\Phi[-1]$ to $\mathscr{Y}_B$.
\end{Prop}
{\it Proof}.
Let $\cpx{Y}\in\mathscr{Y}_B\subseteq \D{B}$. By Lemma \ref{lem:Phi-final}, there exists a triangle in $\D{B}$
$$\overline{{\bf j}}\overline{{\bf K}}(\cpx{Y})\lraf{\overline{\varepsilon}_{\cpx{Y}}}\cpx{Y}
    \longrightarrow{\bf HR}(\cpx{Y})\longrightarrow\overline{{\bf j}}\overline{{\bf K}}(\cpx{Y})[1].
$$
$\cpx{Y}\in\mathscr{Y}_B$ implies that ${\bf R}(\cpx{Y})=0$, and then
$\overline{\varepsilon}_{\cpx{Y}}:\overline{{\bf j}}\Phi(\cpx{Y})[-1]\lraf{\simeq}\cpx{Y}$.
Therefore  
$\overline{\varepsilon}:\overline{{\bf j}} \overline{{\bf K}}
\ra\operatorname{Id}_{\mathscr{Y}_B}$
is a natural isomorphism.

Let $\cpx{X}\in \D{\mathscr{E}}$. It follows from Lemma~\ref{lem:strict-acyclicity-final} (1) that 
${\bf R}(\cpx{X})=0$. Hence $\overline{{\bf j}}(\D{\mathscr{E}})\subseteq \mathscr{Y}_B$.
Thus it makes sense to regard $\overline{{\bf j}}$ as a triangle functor
$\D{\mathscr{E}}\ra\mathscr{Y}_B$.
Denoted by $_p(\overline{{\bf j}}(\cpx{X}))$ the homotopically projective resolution of $\overline{{\bf j}}(\cpx{X})$.
Thus there exists an isomorphism $q_{\overline{{\bf j}}(\cpx{X})}:_p(\overline{{\bf j}}(\cpx{X}))\ra \overline{{\bf j}}(\cpx{X})$ in $\D{B}$. Therefore ${\bf HR}(_p(\overline{{\bf j}}(\cpx{X})))\lraf{\simeq}{\bf HR}(\overline{{\bf j}}(\cpx{X}))=0$ in $\D{B}$.
By the triangle 
$$ C(_p(\overline{{\bf j}}(\cpx{X})))[-1]\xrightarrow{\partial_{_p(\overline{{\bf j}}(\cpx{X}))}}_p(\overline{{\bf j}}(\cpx{X}))
   \lra{\bf HR}(_p(\overline{{\bf j}}(\cpx{X})))\lra C(_p(\overline{{\bf j}}(\cpx{X}))),$$
the homomorphism $\partial_{_p(\overline{{\bf j}}(\cpx{X}))}$ is an isomorphism in
$\D{B}$.  
The morphism $q_{\overline{{\bf j}}(X)}$ is an isomorphism as well.  By
Lemma \ref{lem:pushout-counit-final}, the morphism
$\overline{{\bf j}}(\delta_{\cpx{X}}) =q_{\overline{{\bf j}}(X)}\partial_{_p(\overline{{\bf j}}(\cpx{X}))}
$
is an isomorphism in $\D{B}$. 
Notice that there is an other triangle in Lemma \ref{lem:pushout-counit-final}
$$C(_p(\overline{{\bf j}}(\cpx{X})))[-1]\lraf{\delta_{\cpx{X}}}\cpx{X}
  \longrightarrow \cpx{W}_{\cpx{X}}\longrightarrow C(_p(\overline{{\bf j}}(\cpx{X}))).$$
Applying the functor $\overline{{\bf j}}$  yields a triangle in $\D{B}$. Since $\overline{{\bf j}}(\delta_{\cpx{X}})$
is invertible, $0=\overline{{\bf j}}(\cpx{W}_{\cpx{X}})=\cpx{W}_{\cpx{X}}$ in $\D{B}$. Thus $\cpx{W}_{\cpx{X}}\in\K{\mathscr{E}}$ is exact in $\C{B}$. By Lemma \ref{lem:strict-acyclicity-final} (2), the complex
$\cpx{W}_{\cpx{X}}$ is a strictly exact complex, i.e., $\cpx{W}_{\cpx{X}}\in\Kac{\mathscr{E}}$.
Thus $\cpx{W}_{\cpx{X}}=0$ in $\D{\mathscr{E}}$. This implies that $\delta_{\cpx{X}}$ is an isomorphism in $\D{\mathscr{E}}$.  It follows from Lemma \ref{lem:pushout-counit-final} (2) that $\delta:{\bf K}\overline{{\bf j }}=\Phi(\overline{{\bf j}}(-))[-1]\ra \Id_{\D{\mathscr{E}}}$ is a natural isomorphism. So the functor $\overline{{\bf j}}:\D{\mathscr{E}}\ra \mathscr{Y}_B$ is a triangle equivalence with the quasi-inverse $\overline{{\bf K}}$.
\overpr

\begin{Theo}\label{thm:ordinary-final}\rm
Let $A$ be an associative ring with unit, ${\bf T}$ an $n$-term big tilting complex over $A$ and $B$ the endomorphism ring of ${\bf T}$ in the derived category $\D{A}$ of $A$. Then there is an exact subcategory $\mathscr{E}$ of $B\Modcat$ such that $\D{B}$ is a recollement of $\D{\mathscr{E}}$ and $\D{A}$:
$$
\xymatrix{\D{\mathscr{E}}\ar[r]^-{D(\lambda_*)}&\D{B}\ar[r]\ar@/^1.4pc/[l]\ar@/_1.4pc/[l]
\mathcal{}&\D{A}\ar@/^1.4pc/[l]\ar@/_1.4pc/[l]}
$$

\smallskip
\noindent where $\lambda: \mathscr{E}\to B\Modcat$ is the inclusion of exact categories
and $D(\lambda_*)$ stands for the restriction functor induced by $\lambda$.
\end{Theo}

{\it Proof}. By Lemma \ref{Rec equivalence} and $\D{\mathscr{E}}\simeq \mathscr{Y}_B$, we get the recollement required.
\overpr

\begin{Theo}
\label{thm:chen-xi-criterion}
Let $(\D{B}^{\leq 0},\D{B}^{\geq 0})$ be the standard $t$-structure on $\D{B}$.
The following conditions are equivalent.

(1) $\mathscr{E}$ is an abelian subcategory of $B\Modcat$.

(2) $\bigl(\mathscr{Y}_B\cap\D{B}^{\leq0}, \mathscr{Y}_B\cap\D{B}^{\geq0}\bigr)$ is a $t$-structure on $\mathscr{Y}_B$.

(3) There is a homological ring epimorphism
$\lambda:B\to C$ such that the derived restriction functor $\D{\lambda_*}:\D{C}\ra\mathscr{Y}_B$ is an equivalence.
\smallskip

 When these conditions hold, let $\lambda_*:C\Modcat\to B\Modcat$ be the restriction functor induced by $\lambda$, then
$\mathscr{E}=\Img(\lambda_*)$.
\end{Theo}  
{\it Proof}.
Write \(\mathscr Y=\mathscr Y_B\).

\((1)\Rightarrow(2)\).
By Proposition~\ref{prop:ordinary-final}, the exact inclusion induces an
equivalence
\[
 \overline{\bf j}:\D{\mathscr E}\lraf{\simeq}\mathscr Y.
\]
Let \(X\in\mathscr Y\), and choose a complex \(E^\bullet\) of objects of
\(\mathscr E\) representing \(\overline{\bf j}^{-1}(X)\).  Since
\(\mathscr E\) is an abelian subcategory of \(B\Modcat\), the kernels and
cokernels occurring in the soft truncations
\[
 \tau_{\le0}E^\bullet,\qquad \tau_{\ge1}E^\bullet
\]
again belong to \(\mathscr E\).  Their images under \(\overline{\bf j}\)
are the standard truncations \(\tau_{\le0}X\) and \(\tau_{\ge1}X\) in
\(\D B\).  Hence both truncations lie in \(\mathscr Y\).  The standard
truncation triangle is therefore a triangle in \(\mathscr Y\);
orthogonality and the shift conditions are inherited from \(\D B\).
This proves (2).

\((2)\Rightarrow(1)\).
Let \(f:M\to N\) be a morphism with \(M,N\in\mathscr E\), and put
\(C^\bullet=\con(f)\).  Since \(\mathscr Y\) is triangulated,
\(C^\bullet\in\mathscr Y\).  Apply the shifted restricted \(t\)-structure
to \(C^\bullet\).  After inclusion into \(\D B\), uniqueness of standard
truncations identifies the resulting triangle with
\[
 \Ker(f)[1]\longrightarrow C^\bullet\longrightarrow
 \Coker(f)\longrightarrow\Ker(f)[2].
\]
Thus \(\Ker(f)[1]\) and \(\Coker(f)\) belong to \(\mathscr Y\), and hence
so does \(\Ker(f)\).  Since both are stalk complexes,
\[
 \Ker(f),\Coker(f)\in\mathscr Y\cap B\Modcat=\mathscr E.
\]
Therefore \(\mathscr E\) is closed under kernels and cokernels and is an
abelian subcategory of \(B\Modcat\).

\((1)\Rightarrow(3)\).
Theorem~\ref{thm:sym-new} shows that \(\mathscr E\) is closed under
arbitrary products and coproducts.  By (1), it is also closed under
kernels and cokernels.  Hence \(\mathscr E\) is bireflective in
\(B\Modcat\).  By \cite[Theorem~1.4]{LM12}, there is a ring epimorphism
$
 \lambda:B\longrightarrow C
$
whose restriction functor identifies \(C\Modcat\) exactly with
\(\mathscr E\).  This is an exact equivalence and therefore induces
$\D C\lraf{\simeq}\D{\mathscr E}$.
Under this equivalence the derived restriction functor is the composite
\[
 \D C\xrightarrow{\simeq}\D{\mathscr E}
 \xrightarrow[\simeq]{\overline{\bf j}}\mathscr Y
 \hookrightarrow\D B.
\]
It is fully faithful with essential image \(\mathscr Y\).  Thus \(\lambda\)
is homological and (3) holds.

\((3)\Rightarrow(1)\).
Restriction of scalars is exact, so
$\D{\lambda_*}:\D C\longrightarrow\D B$
is \(t\)-exact for the standard \(t\)-structures.  It also reflects
cohomological vanishing, since \(\lambda_*:C\Modcat\to B\Modcat\) is
faithful.  Consequently the equivalence
\(\D C\xrightarrow{\sim}\mathscr Y\) identifies the standard heart
\(C\Modcat\) with
\[
 \mathscr Y\cap\Dle{0}{B}\cap\Dge{0}{B}
 =\mathscr Y\cap B\Modcat=\mathscr E.
\]
Thus \(\lambda_*:C\Modcat\lraf{\simeq}\mathscr E\) is an exact
equivalence.  In particular, \(\mathscr E\) is an abelian subcategory of $B\Modcat$ and
\(\mathscr E=\Img(\lambda_*)\).
\overpr

\begin{Coro}\label{cor:amplitude-one}
If the perfect right $B$-model $V^\bullet$ can be chosen with
$b-a\leq1$, then $\mathscr{E}$ is abelian and the equivalent conditions of
Theorem~\ref{thm:chen-xi-criterion} hold.
\end{Coro}
{\it Proof}. By Theorem~\ref{thm:sym-new}, it follows that $\mathscr{E}$ 1-symmetric.  For a morphism
$f:M\to N$ in $\mathscr{E}$, apply 1-symmetricity to exact sequence 
\[
 0\longrightarrow\Ker(f)\longrightarrow M\longrightarrow N
 \longrightarrow\Coker(f)\longrightarrow0.
\]
The definition of symmetric subcategories implies that 
both end terms belong to $\mathscr{E}$, so $\mathscr{E}$ is abelian.
\overpr

\begin{Coro}\label{coro:two-term}\cite{XuTwoTerm}
Let $A$ be an associative ring with unit, ${\bf T}$ a two-term big tilting complex over $A$ and $B$ the endomorphism ring of ${\bf T}$ in the derived category $\D{A}$ of $A$. Then there is  a homological ring epimorphism $\lambda:B\ra C$ 
 such that $\D{B}$ is a recollement of $\D{C}$ and $\D{A}$:
$$
\xymatrix{\D{C}\ar[r]^-{D(\lambda_*)}&\D{B}\ar[r]\ar@/^1.4pc/[l]\ar@/_1.4pc/[l]
\mathcal{}&\D{A}\ar@/^1.4pc/[l]\ar@/_1.4pc/[l]}
$$
\end{Coro}
{\it Proof}. 
Since ${\bf i}$ preserves compact objects,
the object $\cpx{U}={\bf i}(A)$ can be represented by
$U^\bullet=0\ra U^0\xrightarrow{\theta}U^1\ra 0\in \Kb{\pmodcat B}$.
Let $\lambda_\theta:B\rightarrow B_\theta$ be the universal localisation of $B$ at $\theta$.
Then $$B_\theta\Modcat\simeq \Img((\lambda_\theta)_*)=\bigl\{M\in B\Modcat{\big|}\Hom_B(U^1,M)\xrightarrow{\Hom_B(\theta, M)} \Hom_B(U^0,M)\hspace{1mm}\mbox{is an isomorphism}\bigr\}.$$
Since $\cpx{U}\in\Kb{\pmodcat B}$ is a homotopically projective complex, 
$$\rHom_B(\cpx{U},M)=\dotHom_B(\cpx{U},M)=
0\lra \Hom_B(U^1,M)\xrightarrow{\Hom_B(\theta, M)} \Hom_B(U^0,M)\lra 0.$$
Thus  $M\in \Img((\lambda_\theta)_*)$ if and only if $\rHom_B(\cpx{U},M)=0$, that is, $\rHom_B(\cpx{U},M)$ is an exact complex of abelian groups.

If $M\in \Img((\lambda_\theta)_*)$, then $0=H^i(\rHom_B(\cpx{U},M))=H^i(\dotHom_B(\cpx{U},M))\simeq \Hom_{\K{B}}(\cpx{U},M[i])$ for all
$i\in\IZ$. This implies that $M\in \Ker(\Hom_{\D{B}}(\cpx{U},-))=\mathscr{Y}_B$, and then $M\in \mathscr{Y}_B\cap B\Modcat=\mathscr{E}$. So $\Img((\lambda_{\theta})_*)\subseteq \mathscr{E}$. On the other hand, the equivalence $\D{\mathscr{E}}\lraf{\simeq}\mathscr{Y}_B$ implies that $\mathscr{E}\subseteq \Img((\lambda_{\theta})_*)$.
Therefore, $B_\theta\Modcat\xrightarrow{\simeq}\Img((\lambda_{\theta})_*)=\mathscr{E}$.
Thus $\mathscr{E}$ is an abelian subcategory of $B\Modcat$, and $\D{B_\theta}\simeq \D{\mathscr{E}}\simeq \mathscr{Y}_B$. By Theorem \ref{thm:chen-xi-criterion},
the universal localisation 
$\lambda_\theta:B\rightarrow B_\theta$
is homological. 
By Theorem \ref{thm:ordinary-final}, $\D{B}$ is a recollement of $\D{A}$ and $\D{C}$.
\overpr
\smallskip

The preceding results become particularly effective when the
endomorphism ring is semiperfect.  In that case, the amplitude relevant
to the symmetry and abelianity of the kernel can be computed after
reduction modulo the Jacobson radical.

\begin{Theo}
\label{thm:radical-support-criterion}
Let ${\bf T}$ be an \(n\)-term big tilting complex over \(A\), let
$B=\End_{\D{A}}({\bf T})$,
and let
$\cpx{T}\in \D{A\otimes_{\mathbb Z}B^{\opp}}$
be the associated derived \(A\)-\(B\)-bimodule.  Assume that \(B\) is
semiperfect and that \(\cpx T_B\neq0\).  Put
$\overline B=B/J(B)$
and define
\[
    r=
    \min\left\{
       i\in\mathbb Z
       \ \middle|\
       H^i\bigl(
          \cpx{T}\otimesL_B\overline B
       \bigr)\neq0
    \right\},
\hspace{3mm}    s=
    \max\left\{
       i\in\mathbb Z
       \ \middle|\
       H^i\bigl(
          \cpx{T}\otimesL_B\overline B
       \bigr)\neq0
    \right\}.
\]
Then the following statements hold.
\begin{enumerate}
\item[(1)] The optimal right \(B\)-amplitude of $\cpx{T}$ is
$ \operatorname{amp}_B(\cpx{T})=s-r$.

\item[(2)] The exact category
\[
    \mathscr E
       =
    \left\{
       M\in B\Modcat
       \ \middle|\
       \cpx{T}\otimesL_BM=0
    \right\}
\]
is \(d\)-symmetric for every \(d\geq s-r\).

\item[(3)] If \(s-r\leq1\), then \(\mathscr E\) is an abelian subcategory
of \(B\Modcat\).  In this case, there exist a ring \(C\) and a
homological ring epimorphism
$\lambda:B\longrightarrow C$
such that
$\mathscr E =
    \operatorname{Im}
       \bigl(
          \lambda_*:C\Modcat\longrightarrow B\Modcat
       \bigr)
$
and restriction of scalars induces a triangle equivalence
$\D{C} \xrightarrow{\ \simeq\ }\Ker\bigl( \cpx{T}\otimesL_B-\bigr)
$.
Consequently, there is a recollement
$$
\xymatrix{\D{C}\ar[r]&\D{B}\ar[r]\ar@/^1.4pc/[l]\ar@/_1.4pc/[l]
&\D{A}\ar@/^1.4pc/[l]\ar@/_1.4pc/[l]}.
$$
\end{enumerate}
\end{Theo}
{\it Proof}.
Let $W^\bullet\in K^b(\Pmodcat{B^{\opp}})$
be the minimal perfect complex representing $\cpx{T}_B$.  Since
\(B\) is semiperfect, \(W^\bullet\) is characterised by
$d_W^i(W^i)\subseteq W^{i+1}J(B)$ for every $i\in\mathbb Z$.
It follows that all differentials of
$ W^\bullet\otimes_B\overline B $
are zero.  Hence
\begin{equation}\label{eq:radical-cohomology-terms}
    H^i\bigl(
       \cpx{T}\otimesL_B\overline B
    \bigr)
       \simeq
    H^i(W^\bullet\otimes_B\overline B)
       \simeq
    W^i/W^iJ(B).
\end{equation}
Since \(W^i\) is finitely generated, Nakayama's lemma gives
\[
    W^i/W^iJ(B)=0
    \quad\Longleftrightarrow\quad
    W^i=0.
\]
Therefore,
$\supp(W^\bullet)
       =
    \left\{
       i\in\mathbb Z
       \ \middle|\
       H^i\bigl(
         \cpx{T}\otimesL_B\overline B
       \bigr)\neq0
    \right\}$.
By Proposition \ref{thm:minimal-amplitude}, we get that
\[ \amp_B(\cpx{T})
       =
    \max\supp(W^\bullet)-\min\supp(W^\bullet)
       =
    s-r.
\]
This proves \textup{(1)}.

By the symmetry theorem, \(\mathscr E\) is \(d\)-symmetric for every
integer
$ d\geq\operatorname{amp}_B(\cpx T)=s-r$,
which proves \textup{(2)}.

Suppose that \(s-r\leq1\).  Then \(\mathscr E\) is \(1\)-symmetric.
Let \(f:M\to N\) be a morphism in \(\mathscr E\).  The exact sequence
\[
    0\longrightarrow\Ker(f)
      \longrightarrow M
      \lraf{f}N
      \longrightarrow\Coker(f)
      \longrightarrow0
\]
and $1$-symmetry imply that
$\Ker(f),\Coker(f)\in\mathscr E$.
Thus \(\mathscr E\) is closed under kernels and cokernels and hence is
an abelian subcategory of \(B\Modcat\).
The Theorem \ref{thm:chen-xi-criterion}
 provides a homological ring
epimorphism
$ \lambda:B\longrightarrow C$
whose restriction functor identifies \(C\Modcat\) with
\(\mathscr E\).  Moreover,
\[
    \D{C}\simeq\D{\mathscr E}
       \simeq
    \Ker\bigl(
       \cpx T\otimesL_B-
    \bigr).
\]
Together with the recollement associated with \(\cpx T\), this
proves \textup{(3)}.
\overpr

\begin{Coro}
\label{cor:finite-radical-test}
Under the assumptions of
Theorem~\ref{thm:radical-support-criterion}, suppose that the graded
\(\overline B\)-module
$H^\ast\bigl(
       \cpx{T}\otimesL_B\overline B
    \bigr)
$
is concentrated in one degree or in two consecutive degrees.  Then
the tensor kernel is induced by a homological ring epimorphism
\(B\to C\), and
\[
\D{C}\simeq\Ker(\cpx{T}\otimesL_B-).
\]
\end{Coro}
{\it Proof}.
The assumption is equivalent to \(s-r\leq1\), so the assertion follows
from Theorem~\ref{thm:radical-support-criterion}.
\overpr

\begin{Rem}
\label{rem:computing-minimal-amplitude}
The criterion is effective.  Starting from any bounded complex of
finitely generated projective right $B$-modules representing
$\cpx{T}_B$, one successively removes every contractible summand
arising from an invertible component of a differential.  Over a
semiperfect ring this process terminates at the minimal complex
\(W^\bullet\).  Equivalently, one reduces the complex modulo \(J(B)\):
the degrees in which
\[
    H^i\bigl(
      \cpx{T}\otimesL_BB/J(B)
    \bigr)
\]
is nonzero are exactly the nonzero degrees of \(W^\bullet\).  Thus the
existence of the homological ring epimorphism in
Corollary~\ref{cor:finite-radical-test} can be checked by a finite
calculation over the semisimple ring \(B/J(B)\).
\end{Rem}

\section{Examples}
\label{sec:examples}

The examples in this section illustrate two different behaviours
of the tensor kernel. We first construct idempotent-chain complexes
whose infinite amplifications yield algebraic Calkin-type
recollements. We then construct connected, non-compact big tilting
complexes with two nonzero cohomology modules and arbitrarily large
minimal projective amplitude. In the latter examples, the associated
exact subcategory is not an abelian subcategory of the ambient module
category, and the tensor kernel cannot be realised by a homological
ring epimorphism. Moreover, these complexes cannot be obtained from
modules by derived equivalences between ordinary rings.

\subsection{Idempotent-chain complexes and an algebraic Calkin quotient}
\label{subsec:matrix-calkin}

We begin with the general criterion underlying the matrix-ring
examples.

\begin{Prop}\label{prop:idempotent-chain-n-term}
Let $A$ be a ring and let
$1=e_1+\cdots+e_n$
be a decomposition of the identity into pairwise orthogonal idempotents.
Put $P_i=Ae_i$.  Suppose that, for every $1\leq i<n$, there is an
injective $A$-homomorphism
$u_i:P_{i+1}\longrightarrow P_i$.
For every $1\leq i<n$, 
define
$L_i=\operatorname{Coker}(u_i)$,
$L_n=P_n$,
and set
$ K_n^\bullet=\bigoplus_{i=1}^{n}L_i[i-1]$.
Equivalently, $K_n^\bullet$ is represented by the $n$-term projective
complex
\begin{equation}\label{eq:idempotent-chain-complex}
 K_n^\bullet\simeq
 \bigoplus_{i=1}^{n-1}
 \bigl(P_{i+1}\xrightarrow{u_i}P_i\bigr)[i-1]
 \oplus P_n[n-1],
\end{equation}
concentrated in degrees $-(n-1),\ldots,0$.

Let $I$ be an infinite set and put
$T_n^\bullet=(K_n^\bullet)^{(I)}$.
Then $T_n^\bullet$ is an $n$-term big tilting complex if and only if
\begin{align}
 \Hom_A(L_i,L_j)&=0
 &&\text{for all }i\ne j,                              \label{eq:chain-hom}\\
 \Ext_A^1(L_i,L_j)&=0
 &&\text{whenever }j\ne i+1.                           \label{eq:chain-ext}
\end{align}
The adjacent groups $\Ext_A^1(L_i,L_{i+1})$ are allowed to be nonzero.
If, in addition, $L_i\ne0$ for every $i$, then $T_n^\bullet$ is genuinely
$n$-term: even up to an overall shift, it is not isomorphic in $\D{A}$
to a complex supported in fewer than $n$ consecutive degrees.
\end{Prop}

{\it Proof}.
For $1\leq i<n$, the defining exact sequence
\[
 0\longrightarrow P_{i+1}\lraf{u_i}P_i
 \longrightarrow L_i\longrightarrow0
\]
is a projective resolution of $L_i$, while $L_n=P_n$ is projective.
Consequently,
\[
 \operatorname{pd}_A L_i\leq1
 \quad\text{and}\quad
 \Ext_A^q(L_i,L_j)=0\quad(q\geq2).
\]
For all $i,j$ and $r\in\mathbb Z$, one has
\begin{equation}\label{eq:chain-shifted-ext}
 \Hom_{\D{A}}
 \bigl(L_i[i-1],L_j[j-1+r]\bigr)
 \simeq
 \Ext_A^{\,j-i+r}(L_i,L_j).
\end{equation}

Assume first that \eqref{eq:chain-hom} and \eqref{eq:chain-ext} hold.
If the right-hand side of \eqref{eq:chain-shifted-ext} is a degree-zero
Hom group, it can be nonzero only when $i=j$, which forces $r=0$.
If it is a degree-one extension group, it can be nonzero only when
$j=i+1$, which again forces $r=0$.  Therefore,
$\Hom_{\D{A}}(K_n^\bullet,K_n^\bullet[r])=0$ for all $r\ne0$.

The complex $K_n^\bullet$ is a bounded complex of finitely generated
projective modules and hence is compact.  For every set $J$ and every
$r\ne0$, compactness gives
\[
 \Hom_{\D{A}}
 \bigl(T_n^\bullet,(T_n^\bullet)^{(J)}[r]\bigr)
\simeq
 \prod_{\lambda\in I}
 \Hom_{\D{A}}
 \bigl(K_n^\bullet,(K_n^\bullet)^{(I\times J)}[r]\bigr)
\simeq
 \prod_{\lambda\in I}
 \bigoplus_{\mu\in I\times J}
 \Hom_{\D{A}}(K_n^\bullet,K_n^\bullet[r])=0.
\]
It remains to verify generation.  Since $L_n=P_n$, one has
$P_n\in\thick(K_n^\bullet)$.  The triangles induced by
\[
 P_{i+1}\longrightarrow P_i\longrightarrow L_i
 \longrightarrow P_{i+1}[1]
\]
show, by descending induction, that
$P_i\in\thick(K_n^\bullet)$ for every $i$.  Hence
$A=\bigoplus_{i=1}^{n}P_i\in\thick(K_n^\bullet)$.
Since $K_n^\bullet$ is a direct summand of
$T_n^\bullet=(K_n^\bullet)^{(I)}$, it follows that
$A\in\thick(T_n^\bullet)$.  Thus $T_n^\bullet$ is big tilting.

Conversely, suppose that $T_n^\bullet$ is big tilting.  Each
$L_i[i-1]$ is a direct summand of $K_n^\bullet$, and $K_n^\bullet$ is a
direct summand of $T_n^\bullet$.  Its self-orthogonality therefore
implies
\[
 \Hom_{\D{A}}
 \bigl(L_i[i-1],L_j[j-1+r]\bigr)=0
 \qquad(r\ne0).
\]
Taking $r=i-j$ gives \eqref{eq:chain-hom}; taking
$r=i-j+1$ gives \eqref{eq:chain-ext}.

Finally,
$H^{-(i-1)}(T_n^\bullet)\simeq L_i^{(I)}$ for all $1\leq i\leq n$.
If every $L_i$ is nonzero, then $T_n^\bullet$ has nonzero cohomology in
all degrees $-(n-1),\ldots,0$ and hence cannot be represented by a
complex supported in fewer than $n$ consecutive degrees.
\overpr

\begin{Ex}
\label{ex:three-term-big-tilting}
Let $R\ne0$ be a commutative ring, let $n\ge2$, and put
\[
 A_n=\operatorname{LT}_n(R)
     =\{(a_{ij})\in M_n(R)\mid a_{ij}=0\text{ for }i<j\}.
\]
Write $e_i=E_{ii}$ and $P_i=A_ne_i$.  For $1\le i<n$, right
multiplication by the matrix unit $E_{i+1,i}$ defines an injective
homomorphism
$u_i:P_{i+1}\longrightarrow P_i$.
Set $L_i=\operatorname{Coker}(u_i)\quad(1\le i<n), L_n=P_n$.
Thus there is an exact sequence
\begin{equation}\label{eq:matrix-Li-resolution}
 0\longrightarrow P_{i+1}\lraf{u_i}P_i
 \longrightarrow L_i\longrightarrow0
 \quad(1\le i<n),
 \qquad L_n=P_n.
\end{equation}
The modules $L_i$ are the diagonal $R$-modules.  Applying
$\Hom_{A_n}(-,L_j)$ to \eqref{eq:matrix-Li-resolution} gives
\begin{equation}\label{eq:matrix-ext-table}
 \Hom_{A_n}(L_i,L_j)\simeq
 \begin{cases}
 R,&i=j,\\
 0,&i\ne j,
 \end{cases}
 \qquad
 \Ext_{A_n}^1(L_i,L_j)\simeq
 \begin{cases}
 R,&j=i+1,\\
 0,&\text{otherwise},
 \end{cases}
\end{equation}
and
$\Ext_{A_n}^q(L_i,L_j)=0$ for all $q\ge2$.

Define
\begin{equation}\label{eq:matrix-Kn}
                  K_n^\bullet=\bigoplus_{i=1}^{n}L_i[i-1]
                  \quad\text{in }\D{A_n}.
\end{equation}
Its canonical projective representative is
\begin{equation}\label{eq:matrix-Kn-projective}
 K_n^\bullet\simeq
 \bigoplus_{i=1}^{n-1}
       (P_{i+1}\xrightarrow{u_i}P_i)[i-1]
 \oplus P_n[n-1],
\end{equation}
and is supported in degrees $-(n-1),\ldots,0$.  For all $i,j$ and $r$,
\[
 \Hom_{\D{A_n}}
 \bigl(L_i[i-1],L_j[j-1+r]\bigr)
 \simeq
 \Ext_{A_n}^{\,j-i+r}(L_i,L_j).
\]
The table \eqref{eq:matrix-ext-table} shows that a nonzero group on the
right forces $r=0$.  Hence
$\Hom_{\D{A_n}}(K_n^\bullet,K_n^\bullet[r])=0
$ if $r\ne0$.
Starting with $P_n=L_n$ and using the triangles induced by
\eqref{eq:matrix-Li-resolution}, descending induction gives
$P_i\in\thick(K_n^\bullet)$ for every $i$.  Therefore
$ A_n=\bigoplus_{i=1}^{n}P_i
                    \in\thick_{\D{A_n}}(K_n^\bullet)
$,
so $K_n^\bullet$ is a compact $n$-term tilting complex.

\smallskip

Now we can construct a non-compact big tilting complex.
Fix an infinite set $I$ and put
$ \widetilde P_n^\bullet=(K_n^\bullet)^{(I)}
$.
For every set $J$ and every $r\ne0$, compactness of $K_n^\bullet$ gives
\[
 \Hom_{\D{A_n}}(\widetilde P_n^\bullet,
       (\widetilde P_n^\bullet)^{(J)}[r])
 \simeq
 \prod_{\lambda\in I}
 \Hom_{\D{A_n}}(K_n^\bullet,(K_n^\bullet)^{(I\times J)}[r])
 \simeq
 \prod_{\lambda\in I}\;
 \bigoplus_{\mu\in I\times J}
 \Hom_{\D{A_n}}(K_n^\bullet,K_n^\bullet[r])=0.
\]
Since $K_n^\bullet$ is a direct summand of
$\widetilde P_n^\bullet$, one also has
$A_n\in\thick(\widetilde P_n^\bullet)$.  Thus
$\widetilde P_n^\bullet$ is an $n$-term big tilting complex.  It is not
compact: its degree-zero cohomology $L_1^{(I)}$ is not finitely generated,
whereas the degree-zero cohomology of a perfect complex is finitely
generated.  Moreover,
\[
 H^{-(i-1)}(\widetilde P_n^\bullet)
 \simeq L_i^{(I)}\ne0
 \qquad(1\le i\le n).
\]
It therefore cannot be represented, even up to an overall shift, by a
complex supported in fewer than $n$ consecutive degrees.

\smallskip
Let $Q_n$ be the linearly oriented quiver
\[
                       1\xrightarrow{\alpha_1}2
                       \xrightarrow{\alpha_2}\cdots
                       \xrightarrow{\alpha_{n-1}}n
\]
and let $J_n$ denote the ideal generated by its arrows.  With the
endomorphism-ring convention fixed above, and after reversing $Q_n$ if
the opposite convention is used, \eqref{eq:matrix-ext-table} gives
$ \Gamma_n:=\End_{\D{A_n}}(K_n^\bullet)
               \simeq RQ_n/J_n^2$.
Indeed, the arrows are the adjacent extension classes, while every
composition of two arrows lies in an $\Ext^2$ group and is therefore zero.
\end{Ex}

\begin{Ex}
\label{ex:explicit-low-length-seeds}
For $n=3$, the compact seed in
\eqref{eq:matrix-Kn-projective} is represented by
\begin{equation}
\label{eq:explicit-three-term-seed}
    0\longrightarrow
    P_3\oplus P_3
    \xrightarrow{
       \left(
       \begin{smallmatrix}
          0&u_2\\
          0&0
       \end{smallmatrix}
       \right)}
    P_2\oplus P_2
    \xrightarrow{
       \left(
       \begin{smallmatrix}
          u_1&0
       \end{smallmatrix}
       \right)}
    P_1
    \longrightarrow0.
\end{equation}
Its cohomology is
$H^{-2}(K_3^\bullet)\simeq L_3,
    H^{-1}(K_3^\bullet)\simeq L_2,
    H^0(K_3^\bullet)\simeq L_1$.
    
For $n=4$, one obtains
\begin{equation}
\label{eq:explicit-four-term-seed}
    0\longrightarrow
    P_4\oplus P_4
    \xrightarrow{\,d^{-3}\,}
    P_3\oplus P_3
    \xrightarrow{\,d^{-2}\,}
    P_2\oplus P_2
    \xrightarrow{\,d^{-1}\,}
    P_1
    \longrightarrow0,
\end{equation}
where
\[
    d^{-3}
    =
    \begin{pmatrix}
       0&u_3\\
       0&0
    \end{pmatrix},
    \qquad
    d^{-2}
    =
    \begin{pmatrix}
       0&u_2\\
       0&0
    \end{pmatrix},
    \qquad
    d^{-1}
    =
    \begin{pmatrix}
       u_1&0
    \end{pmatrix}.
\]
The positions of the nonzero matrix entries make both consecutive
composites zero. Moreover,
\[
    H^{-3}(K_4^\bullet)\simeq L_4,
    \quad
    H^{-2}(K_4^\bullet)\simeq L_3,
    \quad
    H^{-1}(K_4^\bullet)\simeq L_2,
    \quad
    H^0(K_4^\bullet)\simeq L_1.
\]
Thus the amplitudes of the three-term and four-term seeds are genuine.
The same is true for their infinite amplifications.
\end{Ex}

We next compute the kernel associated with the big tilting complex
$\widetilde{\cpx{P}_n}$.

\begin{Theo}
\label{thm:algebraic-calkin-recollement}
Retain the notation of
Example~\ref{ex:three-term-big-tilting}. Put $W_n=\Gamma_n^{(I)}, B_n=\End_{\D{A_n}}(\widetilde{\cpx{P}_n})$.
Under the derived equivalence induced by $\cpx{K}_n$, there is a
ring isomorphism $B_n\simeq\End_{\Gamma_n}(W_n)$.
Choose one copy of $\Gamma_n$ in $W_n$, and let $e_n\in B_n$ be the
idempotent given by the corresponding projection followed by the
inclusion. Set $C_n=B_n/B_ne_nB_n$.
Let $\cpx{X}_n\in \D{A_n\otimes_{\mathbb Z}\Gamma_n^{\opp}}$
be the two-sided tilting complex whose underlying left $A_n$-object is
$K_n^\bullet$. Let $T_n^\bullet
    \simeq
    \cpx{X}_n\otimes_{\Gamma_n}^{\mathbf L}W_n
    \in
    D(A_n\otimes_{\mathbb Z}B_n^{\opp})$
be the complex of $A_n$-$B_n$-bimodule associated with
$\widetilde P_n^\bullet$. Then the following statements hold.

\begin{enumerate}
\item[(1)]
The ideal $B_ne_nB_n$ is precisely the ideal
\[
    \mathfrak F(W_n)
    =
    \left\{
       f\in\End_{\Gamma_n}(W_n)
       \ \middle|\
       f\text{ factors through a finitely generated projective }
       \Gamma_n\text{-module}
    \right\}.
\]
Equivalently, its elements are the endomorphisms of finite projective
rank.

\item[(2)]
The quotient $C_n$ is nonzero.

\item[(3)]
The quotient homomorphism $\lambda_n:B_n\longrightarrow C_n$
is a homological ring epimorphism.

\item[(4)]
The tensor kernel and its exact heart are
\[\mathscr{Y}_{B_n}
    =
    \left\{
       M^\bullet\in \D{B_n}
       \ \middle|\
       e_nM^\bullet\simeq0
    \right\},
\]
and
\[
    \mathscr E_n
    =
    \left\{
       M\in B_n\Modcat
       \ \middle|\
       e_nM=0
    \right\}
    \simeq C_n\Modcat.
\]
Consequently,
\begin{equation}
\label{eq:general-calkin-kernel}
    \D{C_n}
    \simeq
    \D{\mathscr E_n}
    \lraf{\sim}
    \Ker(T_n^\bullet\otimesL_{B_n}-):
       \D{B_n}\lra \D{A_n)},
\end{equation}
and there is a recollement
$$
\xymatrix{\D{C_n}\ar[r]&\D{B_n}\ar[r]\ar@/^1.4pc/[l]\ar@/_1.4pc/[l]
&\D{A_n}\ar@/^1.4pc/[l]\ar@/_1.4pc/[l]}.
$$
\end{enumerate}
\end{Theo}

\begin{proof}
The derived equivalence
$\cpx{X}_n\otimes_{\Gamma_n}^{\mathbf L}-:
    D(\Gamma_n)\xrightarrow{\ \sim\ }D(A_n)$
takes $W_n$ to $\widetilde P_n^\bullet$. Hence $B_n
    \simeq
    \End_{\Gamma_n}(W_n)$.
The idempotent $e_n$ is represented by the projection of $W_n$ onto
one copy of $\Gamma_n$, followed by the corresponding inclusion.
Every element of $B_ne_nB_n$ is therefore a finite sum of
endomorphisms factoring through one copy of $\Gamma_n$, and hence
factors through a finitely generated free $\Gamma_n$-module.

Conversely, an endomorphism factoring through $\Gamma_n^m$ is a
finite sum of endomorphisms factoring through one copy of
$\Gamma_n$. All copies of $\Gamma_n$ in $W_n$ are conjugate under
automorphisms of $W_n$, so every such endomorphism belongs to
$B_ne_nB_n$. Since every finitely generated projective
$\Gamma_n$-module is a direct summand of some $\Gamma_n^m$, this proves
$B_ne_nB_n=\mathfrak F(W_n)$.

If $1_{W_n}$ belonged to $\mathfrak F(W_n)$, then $W_n$ would be a
direct summand of a finitely generated projective module and would
therefore be finitely generated. This contradicts
$W_n=\Gamma_n^{(I)}$ with $I$ infinite. Hence
$C_n=B_n/B_ne_nB_n\neq0$.

As a right $B_n$-module, one has $(W_n)_{B_n}\simeq e_nB_n,e_nB_ne_n\simeq\Gamma_n$.
Therefore, naturally for every $\cpx{M}\in \D{B_n}$,
\begin{equation} \label{eq:calkin-tensor-factorisation}
    T_n^\bullet\otimesL_{B_n}M^\bullet
    \simeq
    \cpx{X}_n\otimesL_{\Gamma_n}(W_n\otimesL_{B_n}M^\bullet)\\
    \simeq
    \cpx{X}_n\otimesL_{\Gamma_n}e_nM^\bullet.
\end{equation}
Equivalently, the tensor functor factors through the idempotent corner in
the commutative diagram
\[
\xymatrix@C=4.8em@R=3.0em{
\D{B_n}\ar[r]^-{e_n(-)}\ar[dr]_-{\cpx{T}_n\otimesL_{B_n}-}
  &\D{\Gamma_n}\ar[d]^-{\cpx{X}_n\otimesL_{\Gamma_n}-}\\
  &\D{A_n}.
}
\]
Since
$\cpx{X}_n\otimes_{\Gamma_n}^{\mathbf L}-$
is an equivalence, $T_n^\bullet\otimes_{B_n}^{\mathbf L}M^\bullet\simeq0
    \quad\Longleftrightarrow\quad
    e_nM^\bullet\simeq0$.
This proves the asserted description of $\mathscr Y_{B_n}$.

For a stalk module $M$, the condition $e_nM=0$ is equivalent to $B_ne_nB_nM=0$.
Restriction of scalars along
$\lambda_n:B_n\to C_n$ therefore identifies
$C_n\Modcat$ with $\mathscr E_n$.

By Theorem~3.24, the inclusion
$\mathscr E_n\hookrightarrow B_n\Modcat$ induces a fully faithful
functor $\D{\mathscr E_n}\lra \D{B_n}$
whose essential image is $\mathscr Y_{B_n}$. Under the exact
identification
$\mathscr E_n\simeq C_n\Modcat$,
this functor is the derived restriction functor associated with
$\lambda_n$. Hence $\lambda_n$ is homological, and
\eqref{eq:general-calkin-kernel} together with the asserted
recollement follows.
\end{proof}

\begin{Rem}
\label{rem:algebraic-calkin-interpretation}
When $I=\mathbb N$, the quotient
\[
    C_n
    \simeq
    \frac{
       \End_{\Gamma_n}(\Gamma_n^{(\mathbb N)})
    }{
       \mathfrak F(\Gamma_n^{(\mathbb N)})
    }
\]
is an algebraic analogue of a Calkin algebra.

For $R=k$ a field and $n=3$, one has
\[
    \Gamma_3\simeq kQ_3/J_3^2,
    \qquad
    Q_3:1\longrightarrow2\longrightarrow3,
\]
and the complex in
\eqref{eq:explicit-three-term-seed} gives a genuinely three-term,
non-compact big tilting complex whose tensor kernel is $\D{C_3}
    \simeq
    \Ker(\cpx{T}_3\otimesL_{B_3}-):\D{B_3}\lra \D{A_3}$.
The infinitude of $I$ is essential for obtaining a nonzero left-hand
term. If $I$ is finite, then $e_n$ is a full idempotent of
$\End_{\Gamma_n}(\Gamma_n^{(I)})$, so $B_ne_nB_n=B_n$ and $C_n=0$.
In particular, using only the compact seed $K_n^\bullet$ produces a
derived equivalence and a zero tensor kernel.
\end{Rem}

\subsection{Connected examples over singular local rings}
\label{subsec:singular-chain}

We construct connected examples over singular Gorenstein
local rings. For every $n\geq3$, the resulting big tilting
complex has nonzero cohomology in all $n$ consecutive
degrees of its minimal projective support. Its tensor-kernel
exact category is nonzero and is not an abelian subcategory
of the ambient module category. We also show that the
complex cannot be obtained from any module by a derived
equivalence between ordinary rings.

Throughout this subsection, endomorphism rings are taken
with the composition convention fixed in Section~2.
Whenever $\circ$ is used, it denotes ordinary composition.

\begin{Ex}\label{ex:singular-chain}
Let $k$ be a field and let $d\geq2$. Put
\[
 R=R_d:=
 \left(
 k[u,v,w,z_1,\ldots,z_{d-2}]/(uv-w^2)
 \right)_{(u,v,w,z_1,\ldots,z_{d-2})},
\]
where the variables $z_1,\ldots,z_{d-2}$ are omitted
when $d=2$, and set $F:=\operatorname{Frac}(R)
   =k(u,w,z_1,\ldots,z_{d-2})$.
Fix a minimal injective resolution of $R$
\begin{equation}\label{eq:singular-injective-resolution}
 0\longrightarrow R\longrightarrow I^0
 \xrightarrow{\partial^0}I^1
 \longrightarrow\cdots
 \xrightarrow{\partial^{d-1}}I^d
 \longrightarrow0,
\end{equation}
and put $U:=\bigoplus_{r=0}^{d}I^r,S:=\End_R(U)$.

Let $D=R\times k^d$, with standard idempotents
$e_0,e_1,\ldots,e_d$, and let $M$ be the $D$-bimodule
whose only nonzero components are $e_1Me_0=F, e_iMe_{i-1}=k$ for $2\leq i\leq d$.
Define the square-zero extension
\[
 A=A_d:=D\ltimes M,
 \qquad
 (a,m)(a',m')=(aa',am'+ma').
\]
Thus the vertex rings are $R,k,\ldots,k$, the consecutive
arrow bimodules are $F,k,\ldots,k$, and every product
of two arrow elements is zero.

Write $\iota_0(N)$ for the $A$-module supported at
vertex $0$ with component $N\in R\Modcat$, and write
$\iota_i(V)$ for the $A$-module supported at vertex $i$
with component $V\in k\Modcat$, for $1\leq i\leq d$.
Define
\begin{equation}\label{eq:singular-big-tilting-object}
 {\bf T}:=
 \iota_0(U)\oplus
 \bigoplus_{i=1}^{d}\iota_i(F)[i]
 \qquad\text{in }\D{A},
\end{equation}
and put $B=\End_{\D{A}}({\bf T})$.
Choose a bounded projective representative $\cpx{P}$
of ${\bf T}$, and let $\cpx{T}$ be the associated
$A$-$B$-bimodule complex constructed above.

The following statements hold.

\textup{(1)}
The ring $A$ is connected, and ${\bf T}$ is a
non-compact big tilting complex. It admits a projective
representative supported in $[-d,0]$, and no projective
representative is supported in fewer than $d+1$
consecutive degrees. Its cohomology is
\[
 H^0({\bf T})=\iota_0(U),
 \qquad
 H^{-i}({\bf T})=\iota_i(F)\neq0
 \quad(1\leq i\leq d).
\]

\textup{(2)}
Put $C=\End_k(F)$.
Localisation induces a ring homomorphism
$\rho:S\longrightarrow F$, and multiplication on $F$
identifies $F$ with a subring of $C$.
With the resulting left $S$-action on $C$, there is
a ring isomorphism
\begin{equation}\label{eq:singular-endomorphism-ring}
 B\simeq
 \begin{pmatrix}
 S&C&C&\cdots&C\\
 0&C&C&\cdots&C\\
 0&0&C&\cdots&C\\
 \vdots&\vdots&\ddots&\ddots&\vdots\\
 0&0&\cdots&0&C
 \end{pmatrix}.
\end{equation}
The first row uses the $S$-action induced by $\rho$;
the remaining products are the usual products in $C$.

\textup{(3)}
Let $\epsilon_r\in S$ be the idempotent corresponding
to $I^r$. Extend $\partial^r$ by zero to an endomorphism
of $U$, and denote the resulting element of
$\epsilon_rS\epsilon_{r+1}$ by $a_r$.
Choose a rank-one idempotent $f\in C$ with image $k\cdot1$.

A left $B$-module is specified by $N=(X,Y_1,\ldots,Y_d;\phi,\psi_2,\ldots,\psi_d)$,
where $X\in S\Modcat$, $Y_i\in C\Modcat$,
$\psi_i:Y_i\longrightarrow Y_{i-1}$ is $C$-linear,
and $\phi:Y_1\longrightarrow X$ is $S$-linear.
Here the $S$-action on $Y_1$ is induced by
$S\xrightarrow{\rho}F\longrightarrow C$.
The actions of the other upper-triangular blocks
are determined by compositions of these maps.

Put $X_r=\epsilon_rX$.
Since $\rho(\epsilon_0)=1$ and
$\rho(\epsilon_r)=0$ for $r>0$, the map $\phi$
has image in $X_0$; denote the resulting map by
$\phi_0:Y_1\longrightarrow X_0$.

Then $N$ belongs to $\mathscr E:=
 \{N\in B\Modcat\mid\cpx{T}\otimesL_BN=0\}$
if and only if $fY_i=0$ for every $1\leq i\leq d$
and the sequence
\begin{equation}\label{eq:singular-kernel-sequence}
 0\longrightarrow X_d
 \xrightarrow{a_{d-1}}\cdots
 \xrightarrow{a_2}X_2
 \xrightarrow{\binom{a_1}{0}}X_1\oplus Y_1
 \xrightarrow{(a_0,\phi_0)}X_0
 \longrightarrow0
\end{equation}
is exact.
For $d=2$, the sequence starts with
$0\longrightarrow X_2$.
The exact structure on $\mathscr E$ is induced by
$B\Modcat$.

\textup{(4)}
The category $\mathscr E$ is nonzero and is not an
abelian subcategory of $B\Modcat$. Consequently, $\D{\mathscr E}
 \simeq\Ker(\cpx{T}\otimesL_B-)\neq0$.
There is no homological ring epimorphism $B\to B'$
whose derived restriction functor has essential image
$\Ker(\cpx{T}\otimesL_B-)$.
\end{Ex}

\begin{proof}
The ring $R$ is a local hypersurface domain of dimension
$d$ and embedding dimension $d+1$.
In particular, it is Gorenstein but not regular.
The assertions about the canonical tilting module $U$
follow from \cite[Section~7]{CX1'}:
$U$ is a good $d$-tilting module with
$\pd_RU=d$.
Moreover,
\[
 I^r\simeq
 \bigoplus_{\substack{\mathfrak p\in\operatorname{Spec}(R)\\
                      \operatorname{ht}\mathfrak p=r}}
 E_R(R/\mathfrak p).
\]
Thus $I^0=F$, and $I^d$ contains the injective hull of
the residue field, so $U$ is an injective cogenerator.
For $r>0$, the module $I^r$ is $R$-torsion.
Consequently,
\begin{equation}\label{eq:singular-localised-U}
 F\otimes_RU\simeq F.
\end{equation}

(1) Commuting with the idempotents $e_0,\ldots,e_d$
forces every central element of $A$ to lie in $D$.
Write such an element as $(r,c_1,\ldots,c_d)$.
Commuting with $e_1Me_0=F$ gives $r=c_1\in k$,
and commuting with the remaining arrow components
gives $c_1=\cdots=c_d$.
Hence $Z(A)=k$, and $A$ is connected.

We next verify the tilting conditions.
For $1\leq i\leq d$ and $V\in k\Modcat$, the arrow
maps give a projective resolution
\begin{equation}\label{eq:singular-vertex-resolution}
 0\longrightarrow Ae_d\otimes_kV
 \longrightarrow\cdots
 \longrightarrow Ae_i\otimes_kV
 \longrightarrow\iota_i(V)
 \longrightarrow0.
\end{equation}
Its exactness follows from the fact that the kernel
of $Ae_j\otimes_kV\to\iota_j(V)$ is
$\iota_{j+1}(V)$ for $j<d$.
The consecutive differentials compose to zero because
$M^2=0$.

Applying $\Hom_A(-,\iota_j(W))$ to
\eqref{eq:singular-vertex-resolution} gives
\begin{equation}\label{eq:singular-vertex-ext}
 \Ext_A^r(\iota_i(V),\iota_j(W))
 \simeq
 \begin{cases}
 \Hom_k(V,W),&j\geq i,\ r=j-i,\\
 0,&\text{otherwise},
 \end{cases}
 \qquad r\geq0.
\end{equation}

Put $L_0(N)=Ae_0\otimes_RN$.
The right $R$-module $Ae_0$ is flat, so $L_0$ is
exact; it also preserves projective modules.
By \eqref{eq:singular-localised-U}, there is a short
exact sequence
\begin{equation}\label{eq:singular-first-extension}
 0\longrightarrow\iota_1(F)
 \longrightarrow L_0(U)
 \longrightarrow\iota_0(U)
 \longrightarrow0.
\end{equation}
The derived adjunction between $L_0$ and $e_0(-)$
gives $\rHom_A(L_0(U),\iota_j(V))=0
 \qquad(j\geq1)$.
Dimension shifting in
\eqref{eq:singular-first-extension}, followed by
\eqref{eq:singular-vertex-ext}, therefore yields
\begin{equation}\label{eq:singular-first-vertex-ext}
 \Ext_A^r(\iota_0(U),\iota_j(V))
 \simeq
 \begin{cases}
 \Hom_k(F,V),&r=j,\\
 0,&r\neq j,
 \end{cases}
 \qquad j\geq1,\ r\geq0.
\end{equation}

For every set $J$, the module $U^{(J)}$ is injective
because $R$ is noetherian.
The functor $\iota_0$ is right adjoint to the exact
functor $e_0(-)$, so $\iota_0(U^{(J)})$ is injective.
Furthermore, $\Hom_A(\iota_i(V),\iota_0(U^{(J)}))=0
 \qquad(i\geq1)$.
These observations and
\eqref{eq:singular-vertex-ext}--%
\eqref{eq:singular-first-vertex-ext} imply $\Hom_{\D{A}}({\bf T},{\bf T}^{(J)}[r])=0$ for $r\neq0$.

Applying $\iota_0$ to
\eqref{eq:singular-injective-resolution} shows that
$\iota_0(R)\in\thick({\bf T})$.
For every $i\geq1$, the module $\iota_i(k)$ is a
direct summand of $\iota_i(F)$, and hence belongs
to $\thick({\bf T})$.
The exact sequences
\[
 0\longrightarrow\iota_1(F)
 \longrightarrow Ae_0
 \longrightarrow\iota_0(R)
 \longrightarrow0
\]
and
\[
 0\longrightarrow\iota_{i+1}(k)
 \longrightarrow Ae_i
 \longrightarrow\iota_i(k)
 \longrightarrow0
 \qquad(1\leq i<d),
\]
together with $Ae_d=\iota_d(k)$, give
$A\in\thick({\bf T})$.

For completeness, a projective representative can
be obtained as follows.
Choose a projective resolution $\cpx{Q}\to U$
supported in $[-d,0]$.
For $1\leq i\leq d$, let $\cpx{P}_i$ be the
projective resolution of $\iota_i(F)$ in
\eqref{eq:singular-vertex-resolution}, supported
in $[-(d-i),0]$.
The inclusion in
\eqref{eq:singular-first-extension} is represented
by a chain map
$\cpx{P}_1\longrightarrow L_0(\cpx{Q})$.
Thus
\[
 \con\bigl(\cpx{P}_1\longrightarrow L_0(\cpx{Q})\bigr)
 \oplus\bigoplus_{i=1}^{d}\cpx{P}_i[i]
\]
is a projective representative of ${\bf T}$ supported
in $[-d,0]$.
The displayed cohomology of ${\bf T}$ shows that
this support length is minimal.

Finally, $e_0(-)$ sends perfect $A$-complexes to
perfect $R$-complexes.
However, $e_0{\bf T}=U$ contains $F$ as a direct
summand and is not finitely generated over $R$.
Since $R$ is noetherian, $U$ is not perfect.
Hence ${\bf T}$ is not compact.

(2) Localising endomorphisms of $U$ and using
\eqref{eq:singular-localised-U} gives
$\rho:S\longrightarrow F$.
The diagonal endomorphism rings of the summands
of ${\bf T}$ are $S,C,\ldots,C$.
The reverse cross-Hom groups vanish, whereas
\[
 \Hom_{\D{A}}(\iota_i(F)[i],\iota_j(F)[j])
 \simeq C
 \qquad(1\leq i<j\leq d)
\]
and
\[
 \Hom_{\D{A}}(\iota_0(U),\iota_j(F)[j])
 \simeq C
 \qquad(1\leq j\leq d).
\]
Choose these identifications using the successive
connecting morphisms in the vertex resolutions.
Their Yoneda products correspond to composition
in $C$.
Precomposition by an element of $S$ depends only
on its action on $F\otimes_RU\simeq F$, and hence
is induced by $\rho$.
This proves \eqref{eq:singular-endomorphism-ring}.

(3) Let $b_0,\ldots,b_d$ be the diagonal block
idempotents of $B$ in
\eqref{eq:singular-endomorphism-ring}.
Let $p_r\in b_0Bb_0$ be the image of $\epsilon_r$,
and let $z\in p_0Bb_1$ have entry $1_C$ in block
$(0,1)$ and zero elsewhere.

Applying $\Hom_R(-,U)$ to
\eqref{eq:singular-injective-resolution} gives a
projective resolution of the natural right
$S$-module $U$:
\[
 0\longrightarrow\epsilon_dS
 \xrightarrow{a_{d-1}}\cdots
 \xrightarrow{a_1}\epsilon_0S
 \longrightarrow U_S\longrightarrow0,
\]
where, more explicitly, the differential from
$\epsilon_{r+1}S$ to $\epsilon_rS$ is left
multiplication by $a_r$.
Equivalently, this is the exact sequence obtained
from $\Hom_R(I^r,U)\simeq\epsilon_rS$.

Let $U_0$ be the right $B$-module obtained from
$U_S$ through the projection $B\to S$.
Since $\rho(\epsilon_r)=0$ for $r>0$, the modules
$p_rB$ have no components outside block $0$
when $r>0$.
The nonzero components of $p_0B$ in blocks
$1,\ldots,d$ are all $C$.
Consequently, the preceding resolution extends to
\begin{equation}\label{eq:singular-U-right-resolution}
 0\longrightarrow p_dB
 \xrightarrow{a_{d-1}}\cdots
 \xrightarrow{a_2}p_2B
 \xrightarrow{\binom{a_1}{0}}p_1B\oplus b_1B
 \xrightarrow{(a_0,z)}p_0B
 \longrightarrow U_0\longrightarrow0.
\end{equation}
Indeed, the block-$0$ component is the resolution
of $U_S$, while every other component contains
the identity map $C\to C$ between the last two
projective terms.

For $1\leq i\leq d$, let $F_i$ be the right
$B$-module supported at block $i$ with its natural
right $C$-module structure.
Write $f_i\in b_iBb_i$ for the image of $f$.
Evaluation at $1$ identifies $fC$ with $F$.
The triangular form of $B$ therefore gives
\begin{equation}\label{eq:singular-F-right-resolution}
 0\longrightarrow f_{i+1}B
 \longrightarrow f_iB
 \longrightarrow F_i
 \longrightarrow0
 \qquad(1\leq i<d),
\end{equation}
and $F_d\simeq f_dB$.
The first map in
\eqref{eq:singular-F-right-resolution} is induced
by the block $(i,i+1)$ element with entry $1_C$.

The left idempotents $e_0,\ldots,e_d$ act on the
$A$-$B$-bimodule complex $\cpx{T}$.
Each $e_i\cpx{T}$ has cohomology in a single degree.
After forgetting the left $A$-action, this gives
\begin{equation}\label{eq:singular-right-decomposition}
 \cpx{T}_B
 \simeq U_0\oplus\bigoplus_{i=1}^{d}F_i[i]
 \qquad\text{in }\D{B^{\mathrm{op}}}.
\end{equation}
The action on $U_0$ factors through $B\to S$,
and the action on $F_i$ factors through the
$i$th diagonal projection $B\to C$.

For a left $B$-module
$N=(X,Y_1,\ldots,Y_d;\phi,\psi_2,\ldots,\psi_d)$,
tensoring
\eqref{eq:singular-F-right-resolution} with $N$
gives the two-term complex
\[
 fY_{i+1}\xrightarrow{\,f\psi_{i+1}\,}fY_i,
\]
whereas $F_d\otimesL_BN=fY_d$.
Starting with $i=d$ and proceeding backwards,
we obtain
\[
 F_i\otimesL_BN=0\ \text{for every }i
 \quad\Longleftrightarrow\quad
 fY_i=0\ \text{for every }i.
\]
Tensoring
\eqref{eq:singular-U-right-resolution} with $N$
gives the complex in
\eqref{eq:singular-kernel-sequence}.
Together with
\eqref{eq:singular-right-decomposition}, this
proves the asserted description of $\mathscr E$.

(4) The ideal $CfC$ consists precisely of the
finite-rank endomorphisms of the $k$-vector space
$F$. Since $F$ is infinite dimensional over $k$,
the quotient $\overline C=C/CfC$ is nonzero.
Take $X=0$, $Y_1=\cdots=Y_{d-1}=0$,
$Y_d=\overline C$, and all structure maps zero.
The criterion in (3) shows that this is a nonzero
object of $\mathscr E$.

Put $\mathscr E_S=
 \{X\in S\Modcat\mid U\otimesL_SX=0\}$.
By \cite[Proposition~7.2]{CX1'} and
Theorem~\ref{thm:chen-xi-criterion}, applied to a
projective resolution of $U$, the category
$\mathscr E_S$ is not an abelian subcategory of
$S\Modcat$.
Part~(3) gives
\[
 (X,0,\ldots,0;0,\ldots,0)\in\mathscr E
 \quad\Longleftrightarrow\quad
 X\in\mathscr E_S.
\]
The embedding $S\Modcat\to B\Modcat$ appearing
here preserves ambient kernels and cokernels.
Thus, if $\mathscr E$ were an abelian subcategory
of $B\Modcat$, then $\mathscr E_S$ would be an
abelian subcategory of $S\Modcat$, a contradiction.

Proposition~\ref{prop:ordinary-final} identifies
$\D{\mathscr E}$ with the tensor kernel, and
Theorem~\ref{thm:chen-xi-criterion} excludes its
realisation by a homological ring epimorphism.
\end{proof}

The following proposition excludes transport from
every module over an ordinary ring. Its proof uses
the minimal projective structure of the connected
ring $A$ and does not depend on the exact realisation
theorem.

\begin{Prop}\label{prop:singular-chain-no-transport}
For the ring $A$ and the object ${\bf T}$ in
Example~\ref{ex:singular-chain}, there is no compact
tilting complex $\cpx{K}\in\D{A}$ satisfying
\[
 \Hom_{\D{A}}(\cpx{K},{\bf T}[r])=0
 \qquad(r\neq0).
\]
Consequently, there are no ordinary ring $\Gamma$,
triangle equivalence
$\mathcal F:\D{\Gamma}\xrightarrow{\simeq}\D{A}$,
and $\Gamma$-module $W$ such that
$\mathcal F(W)\simeq{\bf T}$.
\end{Prop}

\begin{proof}
Suppose that such a compact tilting complex
$\cpx{K}$ exists.
Write $\mathfrak m$ for the maximal ideal of $R$.
The Jacobson radical of $A$ is
$\mathfrak m e_0\oplus M$, and $A/J(A)\simeq k^{d+1}$.
The idempotents $e_0,\ldots,e_d$ lift the standard
idempotents of this quotient. Hence $A$ is
semiperfect.
We may therefore represent $\cpx{K}$ by a bounded
minimal complex of finitely generated projective
$A$-modules.
Every term is a finite direct sum of the modules
$Ae_0,\ldots,Ae_d$.

For $i\geq1$, the module $\iota_i(F)$ is annihilated
by $J(A)$.
Minimality of $\cpx{K}$ therefore implies that
the complex
$\dotHom_A(\cpx{K},\iota_i(F))$
has zero differential.
Moreover,
\[
 \Hom_A(Ae_j,\iota_i(F))
 \simeq
 \begin{cases}
 F,&j=i,\\
 0,&j\neq i.
 \end{cases}
\]
The vanishing $\Hom_{\D{A}}(\cpx{K},\iota_i(F)[i+r])=0$ for $r\neq0$
thus forces $Ae_i$ to occur only in degree $-i$
of $\cpx{K}$.
Let $m_i$ be its multiplicity in that degree.

The exact adjunction with $\iota_0$ gives $\Hom_{\D{A}}(\cpx{K},\iota_0(U)[r])
 \simeq
 \Hom_{\D{R}}(e_0\cpx{K},U[r])$.
Since $U$ is an injective cogenerator, the assumed
vanishing implies $\cpx{Q}:=e_0\cpx{K}\simeq N[0]$,
where $N$ is a finitely generated $R$-module of
finite projective dimension.

Let $\cpx{P}_0$ be the graded summand of
$\cpx{K}$ consisting of all copies of $Ae_0$.
Since $\Hom_A(Ae_0,Ae_i)=0$ for $i\geq1$,
it is a subcomplex, and
$\cpx{P}_0\simeq Ae_0\otimes_R\cpx{Q}$
as complexes.
Put $\cpx{V}=\cpx{K}/\cpx{P}_0$.
Then $V^{-i}=(Ae_i)^{m_i}$ for $1\leq i\leq d$,
and all other terms of $\cpx{V}$ are zero.
The short exact sequence
$0\to\cpx{P}_0\to\cpx{K}\to\cpx{V}\to0$
is split in every degree.

With respect to this graded decomposition, write
\[
 d_K=
 \begin{pmatrix}
 d_{P_0}&g\\
 0&d_V
 \end{pmatrix},
\]
where $g$ is a degree-one cocycle in $\cpx{\mathcal H}:=
 \dotHom_A(\cpx{V},\cpx{P}_0)$.
Again using
$\dotHom_A(\cpx{P}_0,\cpx{V})=0$,
block projection gives a short exact sequence
of complexes
\begin{equation}\label{eq:singular-endomorphism-complexes}
 0\longrightarrow\cpx{\mathcal H}
 \longrightarrow\dotHom_A(\cpx{K},\cpx{K})
 \longrightarrow
 \dotHom_A(\cpx{P}_0,\cpx{P}_0)
 \oplus\dotHom_A(\cpx{V},\cpx{V})
 \longrightarrow0.
\end{equation}

For $i\geq1$,
\[
 \Hom_A(Ae_i,Ae_0)
 \simeq
 \begin{cases}
 F,&i=1,\\
 0,&i>1.
 \end{cases}
\]
Precomposition with the differentials of $\cpx{V}$
contributes zero to $\cpx{\mathcal H}$ because
$M^2=0$.
Since $F_R$ is flat and $\cpx{Q}\simeq N[0]$,
it follows that $H^1(\cpx{\mathcal H})
 \simeq\Hom_k(k^{m_1},F\otimes_RN)$, and $H^r(\cpx{\mathcal H})=0$ for $r\neq1$.
Also, $H^0\dotHom_A(\cpx{P}_0,\cpx{P}_0)
 \simeq\End_R(N)$.
Taking cohomology in
\eqref{eq:singular-endomorphism-complexes} gives
an exact sequence
\begin{equation}\label{eq:singular-self-extension-obstruction}
 \begin{aligned}
 \End_R(N)\oplus\End_{\K{A}}(\cpx{V})
 &\xrightarrow{\Delta}
 \Hom_k(k^{m_1},F\otimes_RN)
 &\longrightarrow
 \Hom_{\D{A}}(\cpx{K},\cpx{K}[1]).
 \end{aligned}
\end{equation}
For degree-zero cocycle representatives $a,b$,
the connecting map is represented by
$g\circ b-a\circ g$.

Generation forces $m_1>0$.
Indeed, if $m_1=0$, then
$\Hom_{\D{A}}(\cpx{K},\iota_1(F)[r])=0$
for every $r$, contradicting
$A\in\thick(\cpx{K})$.
Likewise, $F\otimes_RN\neq0$:
otherwise the exact functor
$F\otimes_Re_0(-)$ would vanish on $\cpx{K}$,
and hence on $A$, although its value on $Ae_0$
is $F$.
Write $\ell=\dim_F(F\otimes_RN)>0$.
Thus the middle term of
\eqref{eq:singular-self-extension-obstruction}
is isomorphic to $F^{\ell m_1}$.

Because $R$ is noetherian and $N$ is finitely
generated, $\End_R(N)$ is a finitely generated
$R$-module.
The ring
$\End_{\K{A}}(\cpx{V})$ is finite dimensional
over $k$, since $\cpx{V}$ is a bounded complex
built from finitely many copies of
$Ae_1,\ldots,Ae_d$, whose mutual Hom spaces are
finite dimensional over $k$.

The restriction of $\Delta$ to $\End_R(N)$ is
$R$-linear, while its restriction to
$\End_{\K{A}}(\cpx{V})$ is $k$-linear.
Indeed, these actions are given respectively by
postcomposition and precomposition with $g$.
Since $k\subseteq R$, the image of $\Delta$ is
therefore contained in a finitely generated
$R$-submodule of $F^{\ell m_1}$.
After choosing an $F$-basis and clearing the
finitely many denominators, there is a nonzero
$s\in R$ such that $\Img(\Delta)\subseteq(s^{-1}R)^{\ell m_1}$.
The element $u\in R$ is nonzero and is not a unit,
so $\frac{1}{su}\notin s^{-1}R$.
Hence $\Delta$ is not surjective.
Exactness of
\eqref{eq:singular-self-extension-obstruction}
gives $\Hom_{\D{A}}(\cpx{K},\cpx{K}[1])\neq0$,
contrary to the self-orthogonality of $\cpx{K}$.

Finally, suppose that
$\mathcal F:\D{\Gamma}\xrightarrow{\simeq}\D{A}$
is a triangle equivalence and that
$\mathcal F(W)\simeq{\bf T}$ for a $\Gamma$-module
$W$.
Then $\cpx{K}=\mathcal F(\Gamma)$ is a compact
tilting complex, and
\[
 \Hom_{\D{A}}(\cpx{K},{\bf T}[r])
 \simeq
 \Hom_{\D{\Gamma}}(\Gamma,W[r])
 =0
 \qquad(r\neq0).
\]
This contradicts the first assertion.
\end{proof}

\begin{Rem}\label{rem:singular-three-term}
Taking $d=2$ gives $R=\bigl(k[u,v,w]/(uv-w^2)\bigr)_{(u,v,w)}, F=k(u,w)$.
The corresponding ring $A$ has vertex rings
$R,k,k$, consecutive arrow bimodules $F,k$, and
zero products of two arrow elements.
The big tilting object is
\[
 {\bf T}=\iota_0(U)\oplus\iota_1(F)[1]
                  \oplus\iota_2(F)[2].
\]
It has nonzero cohomology in degrees $-2,-1,0$,
and its associated exact category consists of
the modules
$(X,Y_1,Y_2;\phi,\psi_2)$ satisfying
$fY_1=fY_2=0$
and
$0\lra X_2
 \lraf{\binom{a_1}{0}}X_1\oplus Y_1
 \lraf{(a_0,\phi_0)}X_0
 \lra0$.
The latter sequence is required to be exact;
$\psi_2:Y_2\to Y_1$ is the $C$-linear structure
map of the $B$-module.

More generally, taking $d=n-1$ gives the asserted
examples for every $n\geq3$.
Their projective support and their cohomological
support both have length $n$.
The construction uses no auxiliary infinite
amplification.
The lower bound $n\geq3$ is necessary for the
failure of homological-ring-epimorphism
realisation, by Corollary~\ref{coro:two-term}.

Finally, the indecomposable finitely generated
projective modules $Ae_0,\ldots,Ae_d$ give
$K_0(\Dc{A})\simeq\mathbb Z^{d+1}$.
Thus the ambient rings belonging to distinct
values of $d$ are not derived equivalent.
\end{Rem}

\end{document}